\documentclass[11pt]{article}

\usepackage[titletoc,toc,title]{appendix}

\usepackage{geometry}
\usepackage{amsmath, amsthm, amsfonts, amssymb}
\usepackage{graphicx}
\usepackage{txfonts}

\usepackage{microtype}

\swapnumbers

\usepackage{colordvi}
\usepackage{hyperref}
\usepackage[latin1]{inputenc}
\usepackage[english]{babel}
\usepackage{xcolor}
\usepackage{verbatim}

\usepackage[all]{xy}

\usepackage{wrapfig}
\usepackage{lipsum}
\usepackage{boxedminipage}

\usepackage{wasysym}

\usepackage{stackrel}

\theoremstyle{plain}
  \begingroup
        \newtheorem{theorem}[equation]{Theorem}
        \newtheorem{conjecture}[equation]{Conjecture}
        \newtheorem{lemma}[equation]{Lemma}
        \newtheorem{proposition}[equation]{Proposition}
        \newtheorem{corollary}[equation]{Corollary}

        \newtheorem{assumption}[equation]{Assumption}
        
	    \newtheorem{definition}[equation]{Definition}
        \newtheorem{notation}[equation]{Notation}

\endgroup

\theoremstyle{definition}
  \begingroup
        \newtheorem{remark}[equation]{Remark}

        \newtheorem{sinnadastandard}[equation]{}

\endgroup

\numberwithin{equation}{section}	
	
\newcommand{\cc}{\mathcal}

\newcommand{\mr}[1]{\buildrel {#1} \over \longrightarrow}

\newcommand{\xr}[1]{\xrightarrow {#1}}

\newcommand{\ml}[1]{\buildrel {#1} \over \longleftarrow}
\newcommand{\Mr}[1]{\buildrel {#1} \over \Longrightarrow}

\newcommand{\dd}{\ar@2{-}[d]}

\newcommand{\adjuntos}[2][]{ \ar@/^1ex/[r]^{#1} \ar@{}[r]|{\bot} \ar@{<-}@/_1ex/[r]_{#2} }

\newcommand{\adjuntosd}[2]{ \ar@{<-}@/_1ex/[d]_{#1} \ar@{}[d]|{\dashv} \ar@/^1ex/[d]^{#2} }

\newcommand{\dosflechasr}[2]{ \ar@<.5ex>[r]^{#1} \ar@<-.5ex>[r]_{#2} } 

\newcommand{\igu}[2][]{ \llbracket #1 \! = \! #2 \rrbracket }

\newcommand{\df}[3]{ \ar@/^1ex/[#1]^{#2} \ar@{<-}@/_1ex/[#1]_{#3} } 
\newcommand{\dfbis}[3]{ \ar@/_1ex/[#1]_{#2} \ar@{<-}@/^1ex/[#1]^{#3} } 

\newcommand{\op}[1]
          {
           \ar@{-}[ld] 
           \ar@{-}[rd] 
           \ar@{}[d]|{#1}  
          }

\newcommand{\cl}[1]
          { 
           \ar@{-}[ur] 
           \ar@{}[u]|{#1} 
           \ar@{-}[ul] 
          }

\newcommand{\ig}[1]
          {
           \ \ \ \ar@{}[d]|{\stackrel{#1}{=}}
          }

\newcommand{\dcell}[1]
          {
           \ar@<4pt>@{-}'+<0pt,-6pt>[d] 
           \ar@<-4pt>@{-}'+<0pt,-6pt>[d]
           \ar@{}[d]|{#1}
          }

\newcommand{\dcellb}[1]
          {
           \ar@<5pt>@{-}'+<0pt,-6pt>[d] 
           \ar@<-5pt>@{-}'+<0pt,-6pt>[d]
           \ar@{}[d]|{#1}
          }

\newcommand{\dcellbb}[1]
          {
           \ar@<6pt>@{-}'+<0pt,-6pt>[d] 
           \ar@<-6pt>@{-}'+<0pt,-6pt>[d]
           \ar@{}[d]|{#1}
          }

\newcommand{\did}       
         {
          \ar@{=}[d]
         }

\newcommand{\pmr}[2]
{
\xymatrix@C=5ex@R=2.4ex
         {
          {} \ar@<1.6ex>[r]^{#1} 
	     \ar@<-1.1ex>[r]^{#2} & {}
         }
}

\newcommand{\pml}[2]
{
\xymatrix@C=5ex@R=2.4ex
         {
            {} 
          & {} \ar@<1.0ex>[l]_{#1} 
	       \ar@<-1.7ex>[l]_{#2}
         }
}

\newcommand{\cellr}[3]
{
\xymatrix@C=7ex@R=2.4ex
         {
          {} \ar@<1.6ex>[r]^{#1} 
          \ar@{}@<-1.3ex>[r]^{\!\! #2 \, \!\Downarrow}
                                         \ar@<-1.1ex>[r]_{#3} & {}
         }
}

\newcommand{\celll}[3]
{
\xymatrix@C=7ex@R=2.4ex
         {
            {} 
          & {} \ar@<1.0ex>[l]^{#1} 
          \ar@{}@<-1.7ex>[l]^{\!\! #2 \, \!\Downarrow}
	                                 \ar@<-1.7ex>[l]_{#3}
         }
}

\newcommand{\dl}    
          {                        
           \ar@<-2pt>@{-}[d]+<4pt,8pt>
          }

\newcommand{\dr}    
          {                        
           \ar@<2pt>@{-}[d]+<-4pt,8pt> 
          }
\newcommand{\dc}[1]    
          {                        
           \ar@{}[d]|{#1}  
          }
\newcommand{\dcr}[1]    
          {                        
           \ar@{}[dr]|{#1}  
          }

\newcommand{\du}[1]  
          {
					 \ar@<-2pt>@{-}[d]+<-4pt,8pt> 
           \ar@<3pt>@{-}[d]+<4pt,8pt>
           \ar@{}[d]|{#1}
          }
          
\newcommand{\drho}[1]  
          {
			      \ar@{-}[d] 
           \ar@{-}[lld] 
           \ar@{}[ld]|{#1}  
          }          

\newcommand{\colim}[2] {\displaystyle \lim_{\overrightarrow{#1}} {#2}}

\newcommand{\mmr}[1]{\buildrel {#1} \over \hookrightarrow}

\newcommand{\BE}{\begin{equation}}
\newcommand{\EE}{\end{equation}}

	\def \O{\mathcal{O}}
	\def \Sat{\mathcal{S}}
	\def \Cat{\mathcal{C}}
	
	\def \Eat{\mathcal{E}}
	
	\def \Xat{\mathcal{X}}
	
	\def \Vat{\mathcal{V}}

	\def \be{\begin{enumerate}}
	\def \en{\end{enumerate}}
	\def \eps{\varepsilon}
  \def \-{\textdblhyphenchar}
	\def \Brr {\ar@{}[rr]|*+<.6ex>[o][F]{\scriptscriptstyle{B}}}
	\def \Brrr {\ar@{}[rrr]|*+<.6ex>[o][F]{\scriptscriptstyle{B}}}

\def \de{definition} 

\def \prop{proposition}

\entrymodifiers={+!!<0pt,\fontdimen22\textfont2>}

\date{}
\title{Tannaka Theory for Topos}
\author{Eduardo J. Dubuc and Mar\'in Szyld.}

\begin{document}

%
%
%
%
%
\maketitle

\begin{abstract}
We consider locales $B$ as algebras in the tensor category $s\ell$ of sup-lattices. We show the \mbox{equivalence} between the Joyal-Tierney descent theorem for open localic surjections 
\mbox{$sh(B) \mr{q} \cc{E}$} 
in \mbox{Galois} theory [\emph{An extension of the Galois Theory of Grothendieck}, AMS Memoirs 151] and a Tannakian recognition theorem over $s\ell$ for the $s\ell$-functor 
\mbox{$Rel(E) \mr{Rel(q^*)}  Rel(sh(B)) \cong (B$-$Mod)_0$} 
into the \mbox{$s\ell$-category} of
discrete \mbox{$B$-modules}. 
Thus, a new Tannaka recognition theorem is obtained, essentially different from those known so far. 
This equivalence follows from two independent results. We develop an explicit construction of  the localic groupoid   
\mbox{$G: \xymatrix@C=3ex{G \stackrel[G_0]{\textcolor{white}{G_0}}{\times} G \ar[r] & G \ar@<1.3ex>[r] \ar@<-1.3ex>[r] & G_0 \ar[l]}$}
 associated by Joyal-Tierney to 
$q$, 
and do an exhaustive comparison with the Deligne Tannakian construction of the Hopf algebroid $L$: 
$\xymatrix@C=3ex
       {
        L \stackrel[B]{\textcolor{white}{B}}{\otimes} L 
      & L \ar[l] 
          \ar[r]   
      & B \ar@<1.3ex>[l] 
          \ar@<-1.3ex>[l]
       }
$ 
associated to 
$Rel(q^*)$, 
and show they are isomorphic, that is, $L \cong \cc{O}(G)$.  On the other hand, we show that the $s\ell$-category of relations of the classifying topos of any localic groupoid $G$, is equivalent to the $s\ell$-category of $L$-comodules with discrete subjacent $B$-module, \mbox{where $L = \cc{O}(G)$.} 
%
%

We are forced to work over an arbitrary base topos because, contrary to the neutral case developed over Sets in [\emph{A Tannakian Context for Galois Theory}, Advances in Mathematics 234], here change of base techniques are unavoidable. 
\end{abstract}



\noindent \textbf{ Introduction}

\vspace{1ex}

\noindent {\bf Galois context.}
In \cite[Expos\'e V section 4]{G2}, ``Conditions axiomatiques d'une theorie de Galois'' (see also \cite{DSV}), Grothendieck interprets Artin formulation of Galois Theory as a theory of representation for suitable categories $\cc{A}$ furnished with a functor (fiber functor) into the category of finite sets 
$\cc{A} \mr{F} \cc{S}_{<\infty} \subset \cc{S}$. He explicitly constructs the group $G$ of automorphisms of 
$F$ as a pro-finite group, and shows that the lifting 
\mbox{$\cc{A} \mr{\widetilde{F}} \beta^G_{<\infty}$} into the category of continuous (left) actions on finite sets is an equivalence. The proof is based on inverse limit techniques. Under Grothendieck assumptions the subcategory \mbox{$\cc{C} \subset \cc{A}$} of non-empty connected objects is an atomic site and the restriction 
\mbox{$\cc{C} \mr{F} \cc{S}_{<\infty} \subset \cc{S}$}  is a point (necessarily open surjective). The SGA1 result in this language means that the lifting
$$
\xymatrix
        {
         \cc{A} \subset \Eat \ar[rr]^{\widetilde{F}} \ar[rd]_{F} 
        && \beta^G \ar[ld]
        \\
        & \cc{S}
        }
$$
is an equivalence. Here $\cc{E}$ is the atomic topos of sheaves on $\cc{C}$, $F$ is the inverse image of the point, and $\beta^G$ is the topos of all continuous (left) actions on sets, the classifying topos of $G$ ($\cc{A}$ becomes the subcategory of finite coproducts of connected objects).

\vspace{1ex}

\noindent
{\bf \emph{Neutral Galois context.}} Joyal-Tierney in \cite{JT} generalize this result to any 
pointed atomic topos. They viewed it as a descent theorem, $G$ is now a localic group, and $\beta^G$, as before, is the topos of continuous (left) actions on sets, i.e., the classifying topos of $G$.  
Dubuc in \cite{D1} gives a proof based, as in SGA1, on an explicit construction of the (localic) group $G$ of automorphisms of $F$ (which under the finiteness assumption is in fact a profinite group).
\emph{Given any pointed atomic topos $\cc{S} \mr{} \cc{E}$, the lifting (of the inverse image functor) is an equivalence.}
\vspace{1ex}

\noindent 
{\bf \emph{General Galois context.}} More generally, Joyal-Tierney in \cite{JT} consider a localic point \mbox{$shH \mr{} \Eat $} ($H$ a locale in $\cc{S}$) over an arbitrary Grothendieck topos $\Eat \mr{} \cc{S}$ over $\cc{S}$, with inverse image  \mbox{$\Eat \mr{F} shH$.} They obtain a localic groupoid $G$ and a lifting into $\beta^G$, the classifying topos of $G$: 
$$
\xymatrix
        {
         \Eat \ar[rr]^{\widetilde{F}} \ar[rd]_{F} 
        && \beta^G \ar[ld]
        \\
        & shH
        }
$$
and prove the following: 
%
%
\emph{Given a localic open surjective point $shH \mr{} \Eat \;$,  the lifting (of the inverse image functor) is an equivalence.} 

This is a descent theorem for open surjections of topoi. When $H = \Omega,\; shH = \cc{S}$, then the point is open surjective precisely when the topos is atomic. Thus this particular case furnishes the theorem for the neutral Galois context. 

\vspace{1ex}

\noindent {\bf Tannakian context.} 
Saavedra Rivano \cite{Sa}, Deligne \cite{De} and Milne \cite{DM} interpret Tannaka theory \cite{T} as a theory of representations of (affine) $K$-schemas. 

\vspace{1ex}

\noindent 
{\bf \emph{General Tannakian context.}} Deligne in \cite[6.1, 6.2, 6.8]{De} considers a field $K$, 
a $K$-algebra $B$, and a linear functor 
\mbox{$\Xat \mr{T} B$-${Mod}_{ptf}$,} from a linear category $\cc{X}$ into the category of projective 
$B$-modules of finite type (note that these modules have a dual module). He constructs  a \emph{cog\`ebro\"ide} $L$ \emph{sur} $B$ and a lifting 
$$
\xymatrix
        {
         \cc{X} \ar[rr]^{\widetilde{T}} \ar[rd]_{T} 
        && \cc{C}md_{ptf}(L) \ar[ld]
        \\
        & B\text{-}{Mod}_{ptf}
        }
$$
into the category of $L$-comodules (called representations of $L$) whose subjacent \mbox{$B$-module} is in 
$B$-${Mod}_{ptf}$. He proves the following: 
%
\emph{if $\cc{X}$ is tensorielle sur $K$ (\cite[1.2, 2.1]{De}) and $T$ is faithful and exact, the lifting is an equivalence.}



\noindent 
{\bf \emph{Neutral Tannakian context.}} If $B = K$, 
$\;B\text{-}{Mod}_{ptf} \,=\, K$-$Vec_{<\infty}$, 
$$
\xymatrix
        {
         \cc{X} \ar[rr]^{\widetilde{T}} \ar[rd]_{T} 
        && \cc{C}md_{<\infty}(L) \ar[ld]
        \\
        & K\text{-}Vec_{<\infty}
        }
$$
In this case $L$ is a $K$-coalgebra. Joyal-Street in \cite{JS} give an explicit coend construction of $L$ as the \mbox{$K$-coalgebra} of endomorphisms of $T$, and they prove:
\emph{if $\cc{X}$ is abelian and $T$ is faithful and exact, the lifting is an equivalence.}


\vspace{1ex}

\noindent {\bf Tannakian context over $\cc{V}$.} 
The general Tannakian context can be developed for a cocomplete monoidal closed category, abbreviated \emph{cosmos}, $(\cc{V},\: \otimes,\: K)$ and $\cc{V}$-categories $\cc{X}$ (\cite{SP} \cite{McC}, \cite{SC}). Although the constructions of Tannaka theory and some of its results regarding for example the \emph{reconstruction theorem} (see \cite{Day}, \cite{McC}) have been obtained, it should be noted that no proof has been made so far of a
\emph{recognition theorem} of the type described above for a cosmos $\cc{V}$ essentially different to the known linear cases. In particular, these results can't be applied to obtain a recognition theorem over the cosmos 
$s\ell$ since in this case the unit of the tensor
product is not of finite presentation.
.
 
In appendix A we develop the Tannakian context for an arbitrary $\cc{V}$ in a way that isn't found in the literature, following closely the lines of Deligne in the linear case \cite{De}. Consider an algebra $B$ in $\cc{V}$, a category $B$-${Mod}_0$ of $B$-modules admitting a right dual, and a $\cc{V}$-category $\cc{X}$ furnished with a $\cc{V}$-functor (fiber functor) 
$\cc{X} \mr{T}$ $B$-${Mod}_0$. We obtain a coalgebra $L$ in the monoidal category of $B$-bimodules (i.e, a 
$B$-bimodule with a coassociative comultiplication and a counit, a \emph{cog\`ebro\"ide agissant sur $B$} in the 
$K$-linear case) and a lifting 
$$
\xymatrix
        {
         \cc{X} \ar[rr]^{\widetilde{T}} \ar[rd]_{T} 
        && \cc{C}md_{0}(L) \ar[ld]
        \\
        & B \text{-} {Mod}_0
        }
$$
where $\cc{C}md_{0}(L)$ is the $\cc{V}$-category of discrete $L$-comodules, that is, $B$-modules in $B$-${Mod}_0$ furnished with a co-action of $L$. 
Adding extra hypothesis on $\cc{C}$ and $T$, $L$ acquires extra structure:

\vspace{-1.5ex}

\begin{itemize}
 \item[(a)] If $\cc{X}$ and $T$ are monoidal, and $\cc{V}$ has a symmetry, then $L$ is a $B\otimes B$-algebra.
 
\vspace{-1.5ex}
 
\item[(b)] If $\cc{X}$ has a symmetry and $T$ respects it, then  $L$ is commutative (as an algebra).

\vspace{-1.5ex}

\item[(c)] If $\cc{X}$ has a duality, then $L$ has an antipode. 
\end{itemize}



\noindent {\bf On the relations between both theories.}  
Strong similarities are evident to the naked eye, and have been long observed between different versions of Galois and Tannaka representation theories. However, these similarities are just of form, and don't allow to transfer any result from one theory to another, in particular Galois Theory and Tannaka theory (over vector spaces) remain independent. 
 
Observing that the category of relations of a Grothendieck topos is a category enriched over sup-lattices, we take this fact as the starting point for our research: 
%
%
%
\emph{The Galois context should be related to the Tannakian context over the cosmos $s\ell$ of sup-lattices.}


In \cite{DSz} we developed this idea and obtained an equivalence between the recognition theorems of Galois and Tannaka in the neutral case over the category of Sets. In this paper we develop the general case. We are forced to work over an arbitrary base topos because here \emph{change of base} techniques become essential and unavoidable.

 

\vspace{1ex}

\noindent {\bf The content of the paper.}

\vspace{1ex}

\noindent \emph{Notation.} Following Joyal and Tierney in \cite{JT}, we fix an elementary topos $\Sat$ (with subobject classifier $\Omega$), and work in this universe using the internal language of this topos, as we would in naive set theory (but without axiom of choice or law of the excluded middle). The category $\cc{V} = s\ell(\cc{S}) = s\ell$ is the symmetric cosmos of sup-lattices in $\cc{S}$. 

In particular, given $X \in \Sat$ and elements $x,x' \in X$, we will denote $\llbracket x \! = \! x' \rrbracket = \delta(x,x') \in \Omega$, where $\delta$ is the characteristic function of the diagonal $X \mr{\triangle} X \times X$. 
Recall that a sup-lattice structure correspond to an $\Omega$-module structure, and that $\Omega$ is the initial locale. Given a locale $H$ we think of $\Omega$ as a sub-locale of $H$, omitting to write the inclusion.

We use the \emph{elevators calculus} described in Appendix \ref{ascensores} to denote arrows and write equations in symetric monoidal categories. 

\vspace{1ex}

\noindent \emph{Section \ref{relintopos}.} This section concerns a single elementary topos that we denote $\cc{S}$. For a locale $G$ in $\cc{S}$, we study G-modules and their duality theory. For any object $X \in \cc{S}$, we show how $G^X$ is self-dual. We consider relations with values in $G$, that is, maps $X \times Y \mr{\lambda} G$, that we call $\ell$-relations, and we study the four Gavin Wright axioms \cite{GW} expressing when a $\ell$-relation is \emph{everywhere defined}, \emph{univalued}, \emph{surjective} and \emph{injective}. We establish in particular that univalued everywhere defined relations correspond exactly with actual arrows in the topos. Finally, we introduce two type of diagrams, the $\rhd$ and $\lozenge$ diagrams, which express certain equations between $\ell$-relations, and that will be extensively used to relate natural transformations with \emph{coend} constructions (not with the usual \emph{end} formula).

\vspace{1ex}

\noindent \emph{Section \ref{sub:EshP}.} This section is the most technical section of the paper. Given a locale $P$ in 
$\cc{S}$ we consider the geometric morphism $sh{P} \mr{\gamma} \cc{S}$ and show how to transfer statements in the topos $sh(P)$ to equivalent statements in $\cc{S}$.   
Recall that Joyal and Tierney develop in \cite{JT} the \emph{change of base} for sup-lattices and locales. In particular they show that  $s\ell(shP) \mr{\gamma_*} P$-$Mod$ is a tensor \mbox{$s\ell$-equivalence} that restricts to a \mbox{$s\ell$-equivalence} $Loc(shP) \mr{\gamma_*} P$-$Loc$. 
We further these studies by examining how $\ell$-relations \mbox{behave} under these equivalences. We examine the correspondence between relations  \mbox{$\gamma^* X \times \gamma^* Y \mr{} \Omega_P$} in $shP$ and $\ell$-relations $X \times Y \mr{} P = \gamma_*\Omega_P$ in $\cc{S}$. We also consider 
$\ell$-relations in $shP$ and the four Gavin Wraith  axioms, and establish how they transfer to formulae in 
$\Sat$. We also transfer the formulae which determine the self-duality of $\Omega_P^X$. 
 
\vspace{1ex}

\noindent \emph{Section \ref{sec:cones}.}
In this section we introduce the notions of $\rhd$- and $\lozenge$-cones in a topos and study how they relate. This allows us to consider natural transformations between functors in terms of their associated cones of relations.   
Concerning the existence of the large coends needed in the Tannakian constructions, we show that cones defined over a site of a topos can be extended uniquely to cones defined over the whole topos. 

\vspace{1ex}

\noindent \emph{Section \ref{sec:Cmd=Rel}.} In this section we establish the relation between the Galois concept of action of a groupoid and the Tannaka concept of comodule of a Hopf algebroid. Given a \emph{localic groupoid} \mbox{$G$: 
$\xymatrix{G \stackrel[G_0]{\textcolor{white}{G_0}}{\times} G \ar[r] & G \ar@<1.3ex>[r]^{\partial_0} \ar@<-1.3ex>[r]_{\partial_1} & G_0 \ar[l]|{\ \! i \! \ }}$}
(we abuse notation by using the same letter $G$ for the object of arrows of $G$), we consider its formal dual \emph{localic Hopf algebroid} $L$:  
$\xymatrix
       {
        L \stackrel[B]{\textcolor{white}{B}}{\otimes} L 
      & L \ar[l] 
          \ar[r]|{\ \! i^* \! \ }   
      & B \ar@<1.3ex>[l]^{\partial_0^*} 
          \ar@<-1.3ex>[l]_{\partial_1^*}
       }
$,
$L = \cc{O}(G)$, $B = \cc{O}(G_0)$.
We establish the equivalence between discrete $G$-actions (i.e, actions on an \emph{etale} family $X \mr{} G_0$, $\cc{O}(X) = Y_d = \gamma_*\Omega_B^Y$, $Y \in sh(B)$), and discrete $L$-comodules (i.e, a comodule structure 
$Y_d \mr{\rho} L \otimes_B Y_d$ on a $B$-module of the form $Y_d$). We also show that  
comodule morphisms correspond to relations in the category of discrete actions. 

All this subsumes in the establishment of a tensor $s\ell$-equivalence  
\mbox{$Rel(\beta^G) \cong Cmd_0(\cc{O}(G))$} between the tensor $s\ell$-categories of relations of the classifying topos of $G$ and that of $\cc{O}(G)$-comodule whose underlying module is discrete.   

\vspace{1ex}

\noindent \emph{Section \ref{sec:Contexts}.} In this section we establish the relation between Joyal-Tierney's Galoisian construction of localic categories (groupoids) $G$ associated to a pair of inverse-image functors 
$\xymatrix{\cc{E} \dosflechasr{F}{F'} & \cc{F}}$, and Deligne's Tannakian construction of cog\`ebro\"ides (Hopf algebroids) $L$ associated to the pair of \mbox{$s\ell$-functors} $\xymatrix{Rel(\cc{E}) \dosflechasr{Rel(F)}{Rel(F')} & Rel(\cc{F})}$. 
Using the results of sections 2 and 3 we show that Joyal-Tierney's construction of $G$ satisfies a universal property equivalent to the universal property which defines $L$. An isomorphism $\cc{O}(G) \cong L$ follows.
 
\vspace{1ex}

\noindent \emph{Section \ref{sec:MainTheorems}.}
A localic point of a topos $shB \mr{q} \cc{E}$, with inverse image $\cc{E} \mr{F} shB$, determines the situation described in the following commutative diagram, where the isomorphisms 
labeled ``a'' 
and ``b'' are obtained in sections \ref{sec:Cmd=Rel} and \ref{sec:Contexts}.

$$ 
\xymatrix
        {
          \beta^G            \ar[r]
                             \ar[rdd]
        & \cc{R}el(\beta^G)  \ar[rr]^{\cong \, a}  
                             \ar[rdd]
        &           
        & Cmd_0(L)           \ar[ldd]
        \\
        & \cc{E}             \ar[d]^{F} 
                             \ar[r]
                             \ar[lu]_{\widetilde{F}}
        & \cc{R}el(\cc{E})   \ar[d]^T
                             \ar[ur]^{\widetilde{T}}
                         \ar[ul]_{\cc{R}el(\widetilde{F})}
        \\
        & shB            \ar[r]
        & Rel(shB) \cong_{b} (B\hbox{-}Mod)_0 
       }
$$

\vspace{1ex}

\noindent
Here $T = \cc{R}el(F)$, $L$ is the 
Hopf algebroid of the Tannakian context over $s\ell$, and $G$  is the localic groupoid of Joyal-Tierney's Galois context. Observe that the triangle on the left is the one of the Galois context, and the triangle on the right is the one of the Tannakian context. 
It follows the equivalence between the Joyal-Tierney recognition theorem for the inverse image functor $F$ of a localic point, and the Tannaka recognition theorem for the $s\ell$-functor $T = Rel(F)$. When the point is open surjective, the first holds, yielding the validity of a Tannaka recognition theorem for $s\ell$-categories of the form $Rel(\cc{E})$.
By the results in \cite{Pitts} this theorem can be interpreted as a recognition theorem for a bounded complete    
distributive category of relations $\cc{A}$ furnished with an open and faithful morphism $\cc{A} \mr{T} (B$-$Mod)_0$.

We end the paper by considering the possible validity of a recognition theorem for general $s\ell$-enriched categories, and conjecture that it will hold for any bounded complete 
\mbox{$s\ell$-category} $\cc{A}$ furnished with an open and faithful $s\ell$-functor $\cc{A} \mr{T} (B$-$Mod)_0$.

\vspace{1ex}

\noindent {\bf Acknowledgements.} The first author thanks Andr\'e Joyal for many stimulating and helpful discussions on the subject of this paper.

\section{Preliminaries on $\ell$-relations in a topos} \label{relintopos}

We begin this paper by showing how the results of \cite[sections 2 and 3]{DSz}, which are developed in $\cc{S}et$, can also be developed in $\Sat$ without major difficulties. This is done with full details in \mbox{\cite[chapters 2 and 3]{tesis},} and we include here only the main results that we will need later. 

\noindent The following lemma will be the key for many following computations (see \mbox{\cite[Lemma 2.11]{tesis}).}

\begin{lemma} \label{ecuacionenL}
If $H$ is a $\Omega$-module (i.e. a sup-lattice), then any arrow $f \in H^X$ satisfies $$
\forall \ x,y \in X \quad \quad \delta( x , y)  \cdot f(x) = \delta( x , y)  \cdot f(y); \quad \quad i.e. \quad \quad 
\llbracket x \! = \! y\rrbracket  \cdot f(x) = \llbracket x \! = \! y\rrbracket  \cdot f(y). \quad \quad \square $$
\end{lemma}

A \emph{relation} between $X$ and $Y$ is a subobject $R \mmr{} X \times Y$ or, equivalently, an arrow $X \times Y \mr{\lambda} \Omega$. We have a category $Rel = Rel(\cc{S})$ of relations in $\cc{S}$. 
A generalization of the concept of relation, that we will call $\ell$-relation, is obtained by letting $\Omega$ be any sup-lattice $H$ (we omit to write the $\ell$ for the case $H=\Omega$).

\begin{\de}
Let $H \in s \ell$. An $\ell$-relation (in $H$) is an arrow 
\mbox{$X \times Y \mr{\lambda} H$.} 
\end{\de}

\begin{assumption}
 In the sequel, whenever we consider the $\wedge$ or the $1$ of $H$, we assume implicitly that $H$ is a locale.
\end{assumption}

\begin{sinnadastandard} 
The following axioms for $\ell$-relations are considered in \cite{GW} (for relations), see also \cite{DSz} and compare with \cite{F} and \cite[16.3]{McLarty}.

\begin{\de} \label{4axioms}
An $\ell$-relation $X \times Y \mr{\lambda} H$ is:

\vspace{1ex}

\noindent ed) Everywhere defined, if for each $x \in X$, $\displaystyle \bigvee_{y \in Y} \lambda(x,y) = 1$.

\vspace{1ex}

\noindent uv) Univalued, if for each $x \in X$, $y_1,y_2 \in Y$, $\lambda(x,y_1) \wedge \lambda(x,y_2) \leq \llbracket y_1 \! = \! y_2\rrbracket $.

\vspace{1ex}

\noindent su) Surjective, if for each $y \in Y$, $\displaystyle \bigvee_{x \in X} \lambda(x,y) = 1$.

\vspace{1ex}

\noindent in) Injective, if for each $y \in Y$, $x_1,x_2 \in X$, $\lambda(x_1,y) \wedge \lambda(x_2,y) \leq \llbracket x_1 \! = \! x_2\rrbracket $.

\end{\de}

\begin{remark} 
Notice the symmetry between ed) and su), and between uv) and in). Many times in this paper we will work with axioms ed) and uv), but symmetric statements always hold with symmetric proofs. 
\end{remark}

\begin{remark} 
Axiom uv) is equivalent to: 

\noindent \emph{uv) for each $x \in X$, $y_1,y_2 \in Y$, $\lambda(x,y_1) \wedge \lambda(x,y_2) = \llbracket y_1 \! = \! y_2\rrbracket  \cdot \lambda(x,y_1)$.} \qed
\end{remark}

\begin{definition} \label{function}
We say that an $\ell$-relation $X \times Y \mr{\lambda} H$ is an \begin{itemize}
\item \emph{$\ell$-function} if it is $uv)$ and $ed)$, 
\item \emph{$\ell$-op-function} if it is $in)$ and $su)$,
\item \emph{$\ell$-bijection} if it is simultaneously an $\ell$-function and an $\ell$-op-function.
\end{itemize}
\end{definition}

\begin{sinnadastandard} {\bf On the structure of $H^X$.} We fix a locale $H$. $H^X$ has the locale structure given pointwise by the structure of $H$. The arrow $H \otimes H^X \mr{\cdot} H^X$ given by $(a \cdot \theta) (x) = a \wedge \theta(x)$ is a $H$-module structure for $H^X$.
We have a $H$-singleton $X \mr{ \{ \}_H } H^X$ defined by $\{x\}_H(y)= \llbracket x \! = \! y \rrbracket$.

\begin{\prop} [{\cite[2.45]{tesis}}] \label{formulainternaparaG} 
For each $\theta \in H^X$, $\displaystyle \theta = \bigvee_{x \in X} \theta(x) \cdot \{x\}_H$. This shows how any arrow $X \mr{f} M$ into a $H$-module can be extended uniquely to $H^X$ as $\displaystyle f(\theta) = \bigvee_{x \in X} \theta(x) \cdot f(x)$, so the $H$-singleton $X \mr{ \{ \}_H } H^X$ is a free-$H$-module structure. \qed
\end{\prop}

\begin{remark}
 A $H$-module morphism $H^X \mr{} M$ is completely determined by its restriction to $\Omega^X$ as in the diagram 
 $\xymatrix{\Omega^X \lhook\mkern-7mu \ar[r] & H^X \ar[r]^f & M \\
	      & X \ar@/^2ex/[ul]^{\{-\}} \ar[u]_{\{-\}_H} \ar@/_2ex/[ur]_f }$
\end{remark}

\begin{lemma} [{\cite[2.46]{tesis}}] \label{reldomegaYparaG}
The $H$-singleton arrow  $Y \mr{\{-\}_H} H^Y$ determines a \emph{presentation} of the $H$-locale $H^Y$ in the following sense:
$$
 i)\;  1 = \displaystyle \bigvee_{y \in Y} \{y\}_H, \quad \quad \quad 
 ii) \; \{x\}_H \wedge \{y\}_H \leq \llbracket x \! = \! y\rrbracket .
$$
Given any other arrow $Y \mr{f} L$ into a $H$-locale $L$ such that: 
$$
 i) \; 1 = \bigvee_y f(y), \quad \quad \quad 
 ii)\; f(x) \wedge f(y) \leq \llbracket x \! = \! y\rrbracket 
$$
there exists a unique $H$-locale morphism $H^Y \mr{f} L$ such that $f(\{y\}_H) = f(y)$. \qed
\end{lemma}

\begin{remark} \label{reldomegaYremarkparaG} 
The previous lemma can be divided into the following two statements: given any arrow $Y \mr{f} L$ into a $H$-locale, its extension as a $H$-module morphism to $H^Y$  preserves $1$ if and only if equation $i)$ holds in $L$, and preserves $\wedge$ if and only if equation $ii)$ holds in $L$.
\end{remark}
\end{sinnadastandard}

\begin{sinnadastandard} {\bf The inverse and the direct image of an $\ell$-relation.} We have the correspondence between an $\ell$-relation, its direct image and its inverse image given by proposition \ref{formulainternaparaG}:

\begin{equation} \label{lambdaphipsiG} 
\begin{tabular}{c} 
$X \times Y \mr{\lambda} H$ an $\ell$-relation \\ \hline \noalign{\smallskip}
$H^Y \mr{\lambda^*} H^X$ a $H$-$Mod$ morphism \\ \hline \noalign{\smallskip}
$H^X \mr{\lambda_*} H^Y$  a $H$-$Mod$ morphism \\
\\
$\lambda^*(\{y\}_H)(x) = \lambda(x, y) = \lambda_*(\{x\}_H)(y)$
\end{tabular}
\end{equation}

$$ \lambda^*(\{y\}_H) = \bigvee_{x \in X} \lambda(x,y) \cdot \{x\}_H, \quad \lambda_*(\{x\}_H) = \bigvee_{y \in Y} \lambda(x,y) \cdot \{y\}_H$$

\end{sinnadastandard}

Since the locale structure of $H^X$ is given pointwise, remark \ref{reldomegaYremarkparaG} immediately implies

\begin{proposition} [{\cite[2.50]{tesis}}]  \label{edyuvdafunctionG} In the correspondence \eqref{lambdaphipsiG}, $\lambda^*$ respects $1$ (resp $\wedge$) if and only if $\lambda$ satisfies axiom $ed)$ (resp. $uv)$).
In particular an $\ell$-relation $\lambda$ is a 
$\ell$-function if and only if its inverse image $H^Y \mr{\lambda^*} H^X$  is a $H$-locale morphism. \qed
\end{proposition}

\begin{remark} \label{edyuvdafunctionmedioG}
 We can also consider $H = \Omega$ in \ref{formulainternaparaG} to obtain the equivalences 
 
 \begin{equation} 
\begin{tabular}{c} 
$X \times Y \mr{\lambda} H$ an $\ell$-relation \\ \hline  \noalign{\smallskip}
$\Omega^Y \mr{\lambda^*} H^X$ a $s \ell$ morphism \\ \hline  \noalign{\smallskip}
$\Omega^X \mr{\lambda_*} H^Y$ a $s \ell$ morphism \\
\end{tabular}
\end{equation}
 
\noindent A symmetric reasoning shows that $\lambda$ is an $\ell$-op-function if and only if $\lambda_*$ is a locale morphism.
\end{remark}
\end{sinnadastandard}

\begin{sinnadastandard} \label{arrowsvsfunctions} {\bf Arrows versus functions.} Consider an arrow $X \mr{f} Y$ in the topos $\Sat$. We define its \emph{graph} 
$R_f = \{(x,y) \in X \times Y \ | \ f(x) = y\}$, and denote its characteristic function by $X \times Y \mr{\lambda_f} \Omega$,  \mbox{$\lambda_f(x,y) =  \igu[f(x)]{y}$.}

\begin{remark} \label{1erremarkbasico}
 Using the previous constructions, we can form commutative diagrams
 
 $$\vcenter{\xymatrix{\Sat \ar[r]^{\lambda_{(-)}} \ar@/_1pc/[rr]_{P} & Rel \ar[r]^{(-)_*} & s \ell & \quad}}
 \vcenter{\xymatrix{\quad & \Sat \ar[r]^{\lambda_{(-)}} \ar@/_1pc/[rr]_{\Omega^{(-)}} & Rel \ar[r]^{(-)^*} & s \ell^{op} }     }$$
 
%
%
%
 In other words, $P(f)$ is the direct image of (the graph of) $f$, and $\Omega^f$ is its inverse image. We will use the notations $f_* = P(f) = (\lambda_f)_* \:$, $\ f^* = \Omega^f = (\lambda_f)^*$.
\end{remark}
 
The relations which are the graphs of arrows of the topos are characterized as follows, for example in \cite[theorem 16.5]{McLarty}. 

 \begin{proposition} \label{arrowsiipi1}
Consider a relation $X \times Y \mr{\lambda} \Omega$, the corresponding subobject   $R \hookrightarrow X \times Y$ and the span $X \ml{p} R \mr{q} Y$ obtained by composing with the projections from the product. There is an arrow $X \mr{f} Y$ of the topos such that  $\lambda = \lambda_f$ if and only if $p$ is an isomorphism, and in this case $f = q \circ p^{-1}$.  \qed
 \end{proposition}

We will now show that $p$ is an isomorphism if and only if $\lambda$ is ed) and uv), concluding in this way that functions correspond to actual arrows of the topos. Even though this is a folklore result (see for example \cite[2.2(iii)]{Pitts}), we include a proof because we couldn't find an appropriate reference.
 
\begin{remark} \label{hojafinal}
 Let $Y \mr{f} X$. For each subobject $A \hookrightarrow X$, with characteristic function $X \mr{\phi_A} \Omega$, by pasting the pull-backs, it follows that the characteristic function of the subobject $f^{-1}A \hookrightarrow Y$ is $\phi_{f^{-1}A} = \phi_A \circ f$.
 This means that the square $\vcenter{\xymatrix{Sub(X) \ar@<1ex>[d]^{f^{-1}} \ar@{}[d]|{\dashv} \ar[r]^{\phi_{(-)}}_{\cong} & [X,\Omega] \ar@<1ex>[d]^{f^*} \ar@{}[d]|{\dashv} \\
					Sub(Y) \ar@<1ex>[u]^{Im_f} \ar[r]^{\phi_{(-)}}_{\cong} & [Y,\Omega] \ar@<1ex>[u]^{\exists_f}}}$ is commutative when considering the arrows going downwards, then also when considering the left adjoints going upwards.

\end{remark}

\begin{proposition} \label{parafunctionconallegories}
In the hypothesis of proposition \ref{arrowsiipi1}, $\lambda$ is ed) if and only if $p$ is epi, and $\lambda$ is uv) if and only if $p$ is mono. 
\end{proposition}

\begin{proof}
For each $\alpha \in \Omega^X$, $\bigvee_{y \in Y} \lambda(-,y) \leq \alpha $ if and only if $\forall x \in X, y \in Y$, $\lambda(x,y) \leq \alpha(x)$, which happens if and only if 
$\lambda \leq {\pi_1}^* (\alpha)$. It follows that $\exists_{\pi_1} (\lambda) = \bigvee_{y \in Y} \lambda(-,y)$. 

%
%


Now, 
by remark \ref{hojafinal} applied to the projection $X \times Y \mr{\pi_1} X$, 
 we have $\phi_{Im_{\pi_1}(R)} = \exists_{\pi_1}(\lambda)$, in particular $R \mr{p} X$ is an epimorphism if and only if $\exists_{\pi_1} (\lambda)(x) = 1$ for each $x \in X$. 
It follows that $\lambda$ is ed) if and only if $p$ is epi.

%
%

\vspace{.3cm}

Also by remark \ref{hojafinal}, the characteristic functions of $(X \times \pi_1)^{-1} R$ and $(X \times \pi_2)^{-1} R $ are respectively $\lambda_1(x,y_1,y_2) = \lambda(x,y_1)$ and $\lambda_2(x,y_1,y_2) = \lambda(x,y_2)$.

Then axiom uv) is equivalent to stating that for each $x \in X$, $y_1, y_2 \in Y$, 
$$
\lambda_1(x,y_1,y_2) \wedge \lambda_2(x,y_1,y_2) \leq \llbracket y_1 \! = \! y_2 \rrbracket,
$$
i.e. that we have an inclusion of subobjects of $X \times Y \times Y$
$$ 
(X \times \pi_1)^{-1} R \cap (X \times \pi_2)^{-1} R \subseteq X \times \triangle_Y. 
$$
But this inclusion is equivalent to stating that for each $x \in X$, $y_1, y_2 \in Y$, $(x, y_1) \in R$ and $(x,y_2) \in R$ imply that $y_1 = y_2$, i.e. that $p$ is mono.
\end{proof}

Combining proposition \ref{parafunctionconallegories} with \ref{arrowsiipi1}, we obtain 

\begin{proposition} \label{functionconallegories}
 A relation $\lambda$ is a function if and only if there is an arrow $f$ of the topos such that $\lambda = \lambda_f$. \qed
\end{proposition}

\begin{remark}\label{simetria}
A symmetric arguing shows that a relation $\lambda$ is an op-function if and only if $\lambda^{op}$ corresponds to an actual arrow in the topos.

Then a relation $\lambda$ is a bijection if and only if there are two arrows in the topos such that $\lambda = \lambda_f$, $\lambda^{op} = \lambda_g$. Then we have that for each $x \in X$, $y \in Y$, 
$$\llbracket f(x) \! = \! y\rrbracket  = \lambda_f(x,y) = \lambda(x,y) = \lambda^{op} (y,x) = \lambda_g(y,x) = \llbracket g(y) \!  =  \! x\rrbracket ,$$
i.e. $f(x)=y$ if and only if $g(y)=x$, in particular $fg(y)=y$ and $gf(x)=x$, i.e.
 $f$ and $g$ are mutually inverse. In other words, bijections correspond to isomorphisms in the topos in the usual sense.
\end{remark}
\end{sinnadastandard}

\begin{sinnadastandard} \label{aplication} {\bf An application to the inverse image.} As an application of our previous results, we will give an elementary proof of \cite[IV.2 Prop. 1]{JT}. 
 The \emph{geometric aspect of the concept of locale} is studied in op. cit. by considering the category of spaces $S \! p = Loc^{op}$ \cite[IV, p.27]{JT}. If $H \in Loc$, we denote its corresponding space by $\overline{H}$, and if $X \in S \! p$ we denote its corresponding locale (of open parts) by $\cc{O}(X)$. If $H \mr{f} L$, then we denote $\overline{L} \mr{\overline{f}} \overline{H}$, and if $X \mr{f} Y$ then we denote $\cc{O}(Y) \mr{f^{-1}} \cc{O}(X)$.

 We have the points functor $S \! p \mr{| \:\; |} \Sat$, $|\overline{H}| = S \! p(1,\overline{H}) = Loc(H,\Omega)$. It's not hard to see that a left adjoint $(-)_{dis}$ of $| \:\; |$ has to map $X \mapsto X_{dis} = \overline{\Omega^X}$, $f \mapsto \overline{f^*}$ (see \cite[p.29]{JT}).
 
%
%
%
%
%
 Combining propositions \ref{edyuvdafunctionG} and \ref{functionconallegories}, we obtain that a relation $\lambda$ is of the form $\lambda_f$ for an arrow $f$ if and only if its inverse image is a locale morphism. Then we obtain: 

\begin{\prop} [{cf. \cite[IV.2 Prop. 1]{JT}}] \label{discretespace}
 We have a full and faithful functor $\Sat \mr{(-)_{dis}} S \! p$, satisfying $(-)_{dis} \dashv | \;\: |$, that maps $X \mapsto X_{dis} = \overline{\Omega^X}$, $f \mapsto \overline{f^*}$. \qed
\end{\prop}


\end{sinnadastandard}

\begin{sinnadastandard} \label{imageninversausandoautodualparaG} {\bf The self-duality of $H^X$.}
We show now that $H^X$ is self-dual as a $H$-module. We then show how this self-duality relates with the inverse (and direct) image of an $\ell$-relation.

\begin{remark} 
 Given $X,Y \in \Sat$, $H^X \stackrel[H]{}{\otimes} H^Y$ is the free $H$-module on $X \times Y$, with the singleton given by the composition of $X \times Y \xrightarrow{<\{-\}_H,\{-\}_H>} H^X \times H^Y$ with the universal bi-morphism \mbox{$H^X \times H^Y \mr{} H^X \stackrel[H]{}{\otimes} H^Y$} (see \cite[II.2 p.8]{JT}). 
\end{remark}

\begin{proposition} [{\cite[2.55]{tesis}}] \label{autodualparaG}
$H^X$ is self-dual as a $H$-module
, with $H$-module morphisms   \mbox{$H \mr{\eta} H^X \stackrel[H]{}{\otimes} H^X$,} $H^X \stackrel[H]{}{\otimes} H^X \mr{\eps} H$ given by the formulae
$$\eta(1) = \bigvee_{x \in X} \{x\}_H \otimes \{x\}_H, \quad \eps(\{x\}_H \otimes \{y\}_H ) = \llbracket x \! = \! y\rrbracket . $$

\vspace{-.8cm}

\qed
\end{proposition}

\begin{\prop}  [{\cite[2.56]{tesis}}]  \label{prop:imageninversausandoautodualparaG}
 Consider the extension of an $\ell$-relation $\lambda$ as a $H$-module morphism $H^X \stackrel[H]{}{\otimes} H^Y \mr{\lambda} H$, and the corresponding $H$-module morphism $H^Y \mr{\mu} H^X$ given by the self-duality of $H^X$. Then $\mu = \lambda^*$. \qed
\end{\prop}

\begin{corollary}  [{\cite[2.57]{tesis}}]  \label{dualintercambiaparaG}
 Taking dual interchanges direct and inverse image, i.e. 
 
 $$H^X \xr{\lambda_* = (\lambda^*)^{\lor}} H^Y, \quad H^Y \xr{\lambda^* = (\lambda_*)^{\lor}} H^X. $$

 \vspace{-.7cm}

\qed
 
 \end{corollary}
\end{sinnadastandard}

\begin{sinnadastandard} 
{\bf $\lozenge$ and $\rhd$ diagrams.} 
 Like we mentioned before, the definitions and propositions of \cite[section 3]{DSz}, can also be developed in an arbitrary elementary topos $\Sat$ without major difficulties. 
 Consider the following situation (cf. \cite[3.1]{DSz}).
\end{sinnadastandard} 
 
\begin{sinnadastandard} \label{diagramadiamante12}
Let $X \times Y \mr{\lambda} H$, \mbox{$X' \times
Y'\mr{\lambda'}H$,} be two $\ell$-relations and $X \mr{f} X'$, $Y \mr{g} Y'$ be two maps, or, more generally, consider two spans, $X \ml{p} R \mr{p'} X'$, $Y \ml{q} S \mr{q'} Y'$, (which induce relations that we also denote $R = p' \circ p^{op}$, $S = q' \circ q^{op}$),
and a third $\ell$-relation $R \times S \mr{\theta} H$. 
These data give rise to the following diagrams in $Rel(\Sat)$: 

\noindent
\begin{tabular}{cccc} 
  \\ $\rhd(f,g)$ & $\lozenge = \lozenge(R,S)$ & $\lozenge_1 = \lozenge_1(f,g)$ & $\lozenge_2 = \lozenge_2(f,g)$ \\  \noalign{\smallskip}
  $\!\!\!\! \vcenter{\xymatrix@C=3ex@R=1ex { X \times Y  \ar[rrd]^{\lambda} \ar[dd]_{f \times g} \\
                     & \hspace{-4ex} {}^{\!\!\geq} & H\,,  \\
                    X' \times Y'  \ar[rru]_{\lambda'} }} \!\!\!\!  $   &
  $\vcenter{\xymatrix@C=1.4ex@R=3ex { & X \times Y  \ar[rd]^{\lambda} \\
	    X \times Y' \ar[rd]_{R \times Y'} \ar[ru]^{X \times S^{op}} & \equiv & H\,, \\
			       & X' \times Y' \ar[ru]_{\lambda '} }}$ &
  $\vcenter{\xymatrix@C=1.4ex@R=3ex{ & X \times Y  \ar[rd]^{\lambda} \\
	    X \times Y' \ar[rd]_{f \times Y'} \ar[ru]^{X \times g^{op}} & \equiv & H\,, \\
			       & X' \times Y' \ar[ru]_{\lambda '}  }} \!\!\!\!  $    & 
  $\vcenter{\xymatrix@C=1.4ex@R=3ex { & X \times Y  \ar[rd]^{\lambda} \\
	    X' \times Y \ar[rd]_{X' \times g} \ar[ru]^{f^{op} \times Y} & \equiv & H\,, \\
			       & X' \times Y' \ar[ru]_{\lambda '} }} \!\!\!\!  $   
\end{tabular}

\end{sinnadastandard}

\begin{remark}  \label{equations}
 The diagrams above correspond to the following equations:
 
\begin{tabular}{ccccc} \label{ecuacionesdiagramas}
\\$\rhd:$ & $\hbox{ for each } a \in X, b \in Y,$ &
$\lambda( a,b ) $ & $\leq$ & $ \lambda'( f(a),g(b) ), $ \\ \noalign{\smallskip}
$\lozenge:$ & $\hbox{ for each } a \in X, b' \in Y', $ &
           $\displaystyle \bigvee_{y \in Y} \llbracket ySb'\rrbracket \cdot \lambda( a,y ) 
         $ & = & $ \displaystyle \bigvee_{x' \in X'} \llbracket aRx'\rrbracket \cdot \lambda'( x',b' ),$ \\ \noalign{\smallskip}
$\lozenge_1:$ & $\hbox{ for each } a \in X, b' \in Y',$ &
$\lambda'( f(a),b' ) $ & = & $ \displaystyle \bigvee_{y \in Y} \llbracket g(y) \! = \! b'\rrbracket \cdot \lambda( a,y ), $ \\ \noalign{\smallskip}
$\lozenge_2:$ & $\hbox{ for each } a' \in X', b \in Y,$ &
$\lambda'( a',g(b) )  $ & = & $  \displaystyle \bigvee_{x \in X} \llbracket f(x) \! = \! a'\rrbracket \cdot \lambda( x,b ). $ 
\end{tabular}

\qed
\end{remark}

The proof that the Tannaka and the Galois constructions of the group (or groupoid) of automorphisms of the fiber functor yield isomorphic structures is based in an analysis of the relations between the          
 $\rhd$ and $\lozenge$ diagrams. 
         

\begin{proposition} \label{diamante12igualdiamante}
 Diagrams $\lozenge_1$ and  $\lozenge_2$ are particular cases of \mbox{diagram $\lozenge$.} Also, the general $\lozenge$ diagram follows from these two particular cases:
let $R$, $S$ be any two spans connected by an \mbox{$\ell$-relation} $\theta$ as above. If  $\lozenge_1(p',q')$ and $\lozenge_2(p,q)$ hold, then so does 
$\lozenge(R,S)$. This last statement is actually a corollary of the more general fact, observed to us by A. Joyal, that $\lozenge$ diagrams respect composition of relations.  \qed
\end{proposition}

%
%

Either $\lozenge_1(f,g)$ or $\lozenge_2(f,g)$ imply the $\rhd(f,g)$ diagram, and the converse holds under some extra hypothesis, in particular if $\lambda$ and $\lambda'$ are $\ell$-bijections (see \cite[3.8, 3.9]{tesis} for details). 

\begin{sinnadastandard} \label{ida}
Assume (in \ref{diagramadiamante12}) that $\lambda$ and $\lambda'$ are bijections, and that the $\rhd(p,q)$ and $\rhd(p',q')$ diagrams hold. Then, if $\theta$ is an $\ell$-bijection, we obtain from $\rhd(p,q)$ and $\rhd(p',q')$ the diagrams $\lozenge_2(p,q)$ and $\lozenge_1(p',q')$, which together imply $\lozenge(R,S)$ (see \ref{diamante12igualdiamante}).
\end{sinnadastandard}

The \mbox{product relation} $\lambda \boxtimes \lambda'$ is defined 
as the following composition (where $\psi$ is the symmetry) $$X \times X' \times Y \times Y' 
 \mr{X \times \psi \times Y'} X \times Y \times X' \times Y' 
 \mr{\lambda \times \lambda'} H \times H \mr{\wedge} H.$$

When $R$, $S$ are relations, it makes sense to consider $\theta$ the restriction of $\lambda \boxtimes \lambda'$ to $R \times S$. For this $\theta$, $\rhd(p,q)$ and $\rhd(p',q')$ hold trivially, and the converse of the implication in \ref{ida} holds. We summarize this in the following proposition.
 
\begin{proposition} \label{combinacion} 
Let $R \subset X \times X'$, $S \subset Y \times Y'$ be any two relations, and $X \times Y \mr{\lambda} H$, 
$X' \times Y' \mr{\lambda'} H$ be $\ell$-bijections.  Let $R \times S \mr{\theta} H$ be the restriction of $\lambda \boxtimes \lambda'$ to $R \times S$. Then, 
$\lozenge(R, S)$ holds if and only if $\theta$ is an $\ell$-bijection. \qed
\end{proposition}

\section{The case $\Eat = sh P$} \label{sub:EshP}

\begin{sinnadastandard} \label{enumeratedeJT} Assume now we have a locale $P \in Loc := Loc(\Sat)$ and we consider $\Eat = sh P$. We recall from \cite[VI.2 and VI.3, p.46-51]{JT} the different ways in which we can consider objects, sup-lattices and locales in $\Eat$.

\begin{enumerate}
 \item We consider the inclusion of topoi $shP \hookrightarrow \Sat^{P^{op}}$ given by the adjunction $\# \dashv i$. A sup-lattice $H \in s \ell(shP)$ yields a sup-lattice $iH \in \Sat^{P^{op}}$, in which the supremum of a sub-presheaf $S \mr{} iH$ is computed as the supremum of the corresponding sub-sheaf $\#S \mr{} H$ (see \cite[VI.1 Proposition 1 p.43]{JT}). The converse actually holds, i.e. if $iH \in s \ell(\Sat^{P^{op}})$ then $H \in s \ell(shP)$, see \cite[VI.3 Lemma 1 p.49]{JT}.
 
 
 \item We omit to write $i$ and consider a sheaf $H \in shP$ as a presheaf $P^{op} \mr{H} \Sat$ that is a sheaf, i.e. that believes covers are epimorphic families. A sup-lattice structure for $H \in shP$ corresponds in this way to a sheaf $P^{op} \mr{H} s \ell$ satisfying the following two conditions (these are the conditions 1) and 2) in \cite[VI.2 Proposition 1 p.46]{JT} for the particular case of a locale):
 
 \begin{itemize}
  \item[a)] For each $p' \leq p$ in $P$, the $s \ell$-morphism $H_{p'}^p: H(p) \mr{} H(p')$, that we will denote by $\rho_{p'}^p$, has a left adjoint $\Sigma_{p'}^p$. 
  \item[b)] For each $q \in P$, $p \leq q$, $p' \leq q$, we have $\rho_{p'}^q \Sigma_p^q = \Sigma_{p \wedge p'}^{p'} \rho_{p \wedge p'}^p$.
 \end{itemize}

 Sup-lattice morphisms correspond to natural transformations that commute with the $\Sigma$'s. 
 
  When interpreted as a presheaf, $\Omega_P(p) = P_{\leq p} := \{ q \in P | q \leq p\}$, with $\rho_q^p = (-) \wedge q$ and $\Sigma_q^p$ the inclusion. The unit $1 \mr{1} \Omega_P$ is given by $1_p = p$.
 
 \item If $H \in s \ell(\cc{S}^{P^{op}})$
 , the supremum of a sub-presheaf $S \mr{} H$ can be computed in $\cc{S}^{P^{op}}$ as the global section $1 \mr{s} H$, $s_q = \displaystyle \bigvee_{\stackrel{p \leq q}{x \in S(p)}} \Sigma_p^q x$ (see   \mbox{\cite[VI.2 proof of proposition 1, p.47]{JT}).}
 
 \item Locales $H$ in $shP$ correspond to sheaves $P^{op} \mr{H} Loc$ such that, in addition to the $s \ell$ \mbox{condition,} satisfy Frobenius reciprocity: if $q \leq p$, $x \in H(p)$, $y \in H(q)$, then   \mbox{$\Sigma^p_{q} (\rho^p_{q} (x) \wedge y) = x \wedge \Sigma^p_{q} y$.}
 
 Note that since $\rho \Sigma = id$, Frobenius implies that if $q \leq p$, $x,y \in H(q)$ then   $\Sigma^p_{q} (x \wedge y) = \Sigma^p_{q} (\rho^p_{q} \Sigma^p_{q} (x) \wedge y) = \Sigma^p_{q} x \wedge \Sigma^p_{q} y$, in other words that $\Sigma$ commutes with 
$\wedge$.
 
 \item The direct image functor establishes an equivalence of tensor categories   $(s \ell(sh P),\otimes) \mr{\gamma_*} (P$-$Mod,\otimes_P)$ (\cite[VI.3 Proposition 1 p.49]{JT}). Given $H \in s \ell(sh P)$ and $p \in P$ multiplication by $p$ in $\gamma_* H = H(1)$ is given by $\Sigma_p^1 \rho_p^1$ (\cite[VI.2 Proposition 3 p.47]{JT}). 
 
 The pseudoinverse of this equivalence is $P$-$Mod \mr{\widetilde{(-)}} s \ell(shP)$, $M \mapsto \widetilde{M}$ defined by the formula \mbox{$\widetilde{M}(p) = \{x \in M \ | \ p \cdot x = x\}$} for $p \in P$.
 
 \item The equivalence of item 5 restricts to an equivalence $Loc(sh P) \mr{\gamma_*} P$-$Loc$, where the last category is the category of locale extensions $P \mr{} M$ (\cite[VI.3 Proposition 2 p.51]{JT}).

 \end{enumerate}

\end{sinnadastandard}

\begin{sinnadastandard} \label{relativizarellrelations} 
We will now consider relations in the topos $shP$ and prove that $\ell$-functions $X \times Y \mr{} P$ in $\Sat$ correspond to arrows $\gamma^*X \mr{} \gamma^*Y$ of the topos $shP$. 

\noindent The unique locale morphism $\Omega \mr{\gamma} P$ induces a topoi morphism $\xymatrix{\Sat \cong sh\Omega \quad \ar@/^2ex/[r]^{\gamma^*} \ar@{}[r]|{\bot} & sh P \ar@/^2ex/[l]^{\gamma_*} }$. Let's denote by $\Omega_P$ the subobject classifier of $sh P$. Since $\gamma_* \Omega_P = P$, we have the correspondence

\begin{center}
 \begin{tabular}{c}
  $X \times Y \mr{\lambda} P$ an $\ell$-relation \\ \hline \noalign{\smallskip}
  $\gamma^*Y \times \gamma^*X \mr{\varphi} \Omega_P$ a relation in $sh P$
 \end{tabular}
\end{center}

\begin{\prop} \label{prop:relativizarellrelations}
In this correspondence, $\lambda$ is an $\ell$-function if and only if $\varphi$ is a function. Then, by proposition \ref{functionconallegories}, 
$\ell$-functions correspond to arrows $\gamma^*X \mr{\varphi} \gamma^*Y$ in the topos $shP$, and by remark \ref{simetria} $\ell$-bijections correspond to isomorphisms.
\end{\prop}

\begin{proof}
 
Consider the extension $\widetilde{\lambda}$ of $\lambda$ as a $P$-module, and $\widetilde{\varphi}$ of $\varphi$ as a $\Omega_P$-module, i.e. in $s \ell(shP)$ (we add the $\widetilde{(-)}$ to avoid confusion). 
We have the binatural correspondence between $\widetilde{\lambda}$ and $\widetilde{\varphi}$:

\begin{center}
 \begin{tabular}{c}
  $\xymatrix@C=3.5pc{  X \times Y \ar@/^4ex/[rr]^{\lambda}  \ar[r]_{\{-\}_P \otimes \{-\}_P }  & P^X \stackrel[P]{}{\otimes} P^Y \ar[r]_>>>>>>>>{\widetilde{\lambda}} & P }$   \\ \hline \noalign{\smallskip}
  $\xymatrix@C=3pc{  \gamma^* X \times  \gamma^*Y  \ar@/_4ex/[rr]_{\varphi} \ar[r]^{\{-\} \otimes \{-\}}  & \Omega_P^{\gamma^*X} \otimes \Omega_P^{\gamma^*Y} \ar[r]^>>>>>>>>{\widetilde{\varphi}} & \Omega_P} $
 \end{tabular}
\end{center}

given by the adjunction $\gamma^* \dashv \gamma_*$. But $\gamma_*(\Omega_P^{\gamma^*X}) = (\gamma_* \Omega_P)^X = P^X$ and $\gamma_*$ is a tensor functor, then $\gamma_*(\Omega_P^{\gamma^*X} \otimes \Omega_P^{\gamma^*Y}) = P^X \stackrel[P]{}{\otimes} P^Y$ and $\gamma_*(\ \widetilde{\varphi} \ ) = \widetilde{\lambda}$.

Now, the inverse images $\lambda^*,\varphi^*$ are constructed from $\widetilde{\lambda}, \widetilde{\varphi}$ using the self-duality of $\Omega_P^{\gamma_*X}, P^X$ (see proposition  \ref{prop:imageninversausandoautodualparaG}), and since $\gamma^*$ is a tensor functor that maps $\Omega_P^{\gamma_*X} \mapsto P^X$ we can take $\eta$, $\eps$ of the self-duality of $P^X$ as $\gamma^*(\eta')$, $\gamma^*(\eps')$, where $\eta'$, $\eps'$ are the self-duality structure of $\Omega_P^{\gamma_*X}$. It follows that $\gamma_*(\varphi^*) = \lambda^*$, then by \ref{enumeratedeJT} (item 6) we obtain that $\varphi^*$ is a locale morphism if and only if $\lambda^*$ is so. Proposition \ref{edyuvdafunctionG} finishes the proof.
\end{proof}

\begin{remark} \label{remark:relativizarellrelations}
 Though we will not need the result with this generality, we note that proposition \ref{prop:relativizarellrelations} 
 also holds for an arbitrary topos over $\Sat$, $\cc{H} \mr{h} \Sat$, in place of $shP$. Consider $P = h_* \Omega_{\cc{H}}$, the hyperconnected factorization $\vcenter{\xymatrix{\cc{H} \ar[rr]^q \ar[rd]_{h} && shP \ar[dl]^{\gamma} \\ & \cc{S} }}$ (see \cite[VI. 5 p.54]{JT}) and recall that $q_* \Omega_{\cc{H}} \cong \Omega_{P}$ and that the counit map $q^* q_* \Omega_{\cc{H}} \mr{} \Omega_{\cc{H}}$ is, up to isomorphism, the comparison morphism $q^* \Omega_{P} \mr{} \Omega_{\cc{H}}$ of remark \ref{caractsubobjeto} (see \cite[1.5, 1.6]{JohnstoneFactorization}). The previous results imply that the correspondence between relations $X \times Y \mr{} \Omega_P$ and relations $q^* X \times q^* Y \mr{} \Omega_{\cc{H}}$ given by the adjunction $q^* \dashv q_*$ is simply the correspondence between a relation \mbox{$R \hookrightarrow X \times Y$} in $shP$ and 
its image by the full 
and 
faithful morphism $q^*$, therefore functions correspond to functions. Since by proposition \ref{prop:relativizarellrelations} we know that the same happens for $shP \mr{\gamma} \cc{S}$, by composing the adjunctions it follows for $\cc{H} \mr{h} \cc{S}$.
\end{remark}

\end{sinnadastandard}

\begin{notation} \label{abusosvarios}
 Let $p \in P$, we identify by Yoneda $p$ with the representable presheaf $p = [-,p]$. If $q \in P$, then $[q,p] = \llbracket q \! \leq \! p \rrbracket \in \Omega$. In particular if $a \leq p$ then $[a,p] = 1$. 
Given $X \in shP$, and $a \leq p \in P$, $x \in X(p)$, consider $X(p) \mr{X_a^p} X(a)$ in $\Sat$. We will denote $x|_a := X_a^p(x)$. 
%
%
\end{notation}




Consider now a sup-lattice $H$ in $shP$, we describe now the sup-lattice structure of the exponential $H^X$. Recall that as a presheaf, $H^X(p) = [p \times X, H]$, and note that if $\theta \in H^X(p)$, and $a \leq p$, by notation \ref{abusosvarios} we have $X(a) \mr{\theta_a} H(a)$.

\begin{sinnadastandard} \label{defderhopq} 
$\theta$ corresponds to \mbox{$X \mr{\hat{\theta}} H^p$, $X(q) \mr{\hat{\theta}_q} H^p(q) \cong [q \wedge p, H] \cong H(q \wedge p)$} 
by the exponential law,
under this correspondence we have $\hat{\theta}_q(x) = \theta_{q \wedge p} (x|_{q \wedge p}).$ 
This implies that $\theta \in H^X(p)$ is completely characterized by its components $\theta_a$ for $a \leq p$. From now on we make this identification, i.e. we consider $\theta \in H^X(p)$ as a family \mbox{$\{X(a) \mr{\theta_a} H(a)\}_{a \leq p}$} natural in $a$. Via this identification, if $q \leq p$, the morphism \mbox{$H^X(p) \mr{\rho_q^p} H^X(q)$} is given by \mbox{$\{X(a) \mr{\theta_a} H(a)\}_{a \leq p} \mapsto \{X(a) \mr{\theta_a} H(a)\}_{a \leq q}$.}
\end{sinnadastandard}

\begin{lemma} \label{exponencialenshP}
 Let $X \in shP$, $H \in s \ell(shP)$. Then the sup-lattice structure of $H^X$ is given as follows: 
 \begin{enumerate}
  \item For each $p \in P$, $H^X(p) = \{ \{X(a) \mr{\theta_a} H(a)\}_{a \leq p} \hbox{ natural in a}\}$ is a sup-lattice pointwise, i.e. for a family $\{\theta_i\}_{i \in I}$ in $H^X(p)$, for $a \leq p$,
  $\left( \bigvee_{i \in I} \theta_i \right) _a = \bigvee_{i \in I} {(\theta_i)}_a $
  \item If $q \leq p$ the morphisms $\xymatrix{H^X(q) \adjuntos[\Sigma_q^p]{\rho_q^p} & H^X(p)}$ are defined by the formulae (for $\theta \in H^X(p)$, $\xi \in H^X(q)$):
%
%
%
$\quad \quad \quad F\rho) \quad \quad (\ \rho_q^p \ \theta)_{a} (x) = \theta_{a} (x)$ for $x \in X(a)$, $a \leq q$.
 
 \noindent $\quad \quad \quad \quad \quad \quad \quad \quad F\Sigma) \quad \quad (\Sigma_q^p \ \xi)_{a} (x) = \Sigma_{a \wedge q}^{a} \ \xi_{a \wedge q} \ (x|_{a \wedge q})$ for $x \in X(a)$, $a \leq p$.
\end{enumerate}
 \end{lemma}

\begin{proof}
 It is immediate from \ref{defderhopq} that $\rho_q^p$ satisfies $F\rho)$.
 
 We have to prove that if $\Sigma_q^p$ is defined by $F\Sigma)$ then the adjunction holds, i.e. that 
 
 $$A: \Sigma_q^p \ \xi \leq \theta \quad \quad \hbox{ if and only if } \quad \quad B: \xi \leq \rho_q^p \ \theta.$$
 
 \noindent By $F\Sigma)$, $A$ means that for each $a \leq p$, for each $x \in X(a)$ we have $\Sigma_{a \wedge q}^{a} \ \xi_{a \wedge q} \ (x|_{a \wedge q}) \leq \theta_{a}(x)$.
 
 \noindent By $F\rho)$, $B$ means that for each $a \leq q$, for each $x \in X(a)$ we have $\xi_{a}(x) \leq \theta_{a}(x)$. 
 
 Then $A$ implies $B$ since if $a \leq q$ then $a \wedge q = a$, and $B$ implies $A$ since for each $a \leq p$, for each $x \in X(a)$, by the adjunction $\Sigma \dashv \rho$ for $H$, $\Sigma_{a \wedge q}^{a} \ \xi_{a \wedge q} \ (x|_{a \wedge q}) \leq \theta_{a}(x)$ holds in $H(a)$ if and only if $\xi_{a \wedge q} \ (x|_{a \wedge q}) \leq \rho_{a \wedge q}^{a} \  \theta_{a}(x)$ holds in $H(a \wedge q)$, but this inequality is implied by $B$ since by naturality of $\theta$ we have $\rho_{a \wedge q}^{a} \ \theta_{a}(x) = \theta_{a \wedge q} \ (x|_{a \wedge q})$.
\end{proof}

%
%
%

\begin{remark} \label{1deGalaX}
If $X \in shP$, $H \in Loc(shP)$, the unit $1 \in H^X$ is a global section given by the arrow $X \mr{} 1 \mr{1} \Omega_P \mr{} H$, which by \ref{enumeratedeJT} item 2 maps $1_p(x) = p$ for each $p \in P$, $x \in X(p)$.
\end{remark}


\begin{sinnadastandard} For the remainder of this section, the main idea (to have in mind during the computations) is to consider some of the situations defined in section \ref{relintopos} for the topos $shP$, and to ``transfer'' them to the base topos $\Sat$. In particular we will transfer the four axioms for an $\ell$-relation in $shP$ (which are expressed in the internal language of the topos $shP$) to equivalent formulae in the language of $\Sat$ (proposition \ref{propaxiomparamodulos}), and also transfer the self-duality of $\Omega_P^X$ in $s \ell(shP)$ to an self-duality of $P$-modules (proposition \ref{formulaeetaeps}). These results will be used in section \ref{sec:Cmd=Rel}. 
\end{sinnadastandard}

 Consider $X \in shP$, $H \in s \ell(shP)$ and an arrow $X \mr{\alpha} H$. We want to compute the internal supremum $\displaystyle \bigvee_{x \in X} \alpha(x) \in H$. This supremum is the supremum of the subsheaf of $H$ given by the image of $\alpha$ in $shP$, which is computed as $\#S \hookrightarrow H$, where $S$ is the sub-presheaf of $H$ given by $S(p) = \{\alpha_p(x) \ | \ x \in X(p)\}$. Now, by \ref{enumeratedeJT} item 1 (or, it can be easily verified), this supremum coincides with the supremum of the sub-presheaf $S \hookrightarrow H$, which by \ref{enumeratedeJT} item 3 is computed as the global section $1 \mr{s} H$, $s_q = \displaystyle \bigvee_{\stackrel{p \leq q}{x \in X(p)}} \Sigma_p^q \alpha_p(x)$. Applying the equivalence $\gamma_*$ of \ref{enumeratedeJT}, item 5 it follows:
 
 
\begin{\prop} \label{calcsup} Let $X \mr{\alpha} H$ as above. Then at the level of \mbox{$P$-modules,} the element $s \in H(1)$ corresponding to the internal supremum $\displaystyle \bigvee_{x \in X} \alpha(x) \;\;$ is $ \; \displaystyle \bigvee_{\stackrel{p \in P}{x \in X(p)}} \Sigma_p^1 \alpha_p(x) $. \qed
\end{\prop}

\begin{definition} \label{defdeXd}
 Given $X \in shP$, recall that we denote by $\Omega_P$ the object classifier of $shP$ and consider the sup-lattice in $shP$, $\Omega_P^X$ (which is also a locale). We will denote by $X_d$ the $P$-module (which is also a locale extension $P \mr{} X_d$) corresponding to $\Omega_P^X$, in other words \mbox{$X_d := \gamma_*(\Omega_P^X) = \Omega_P^X(1)$.}
  Given $p \in P$, $x \in X(p)$ we define the element $\delta_x := \Sigma_p^1 \{x\}_p \in X_d$.
\end{definition}

Consider now $\theta \in X_d$, that is $\theta \in \Omega_P^X(1)$, i.e. $X \mr{\theta} \Omega_P$ in $shP$. Let $\alpha$ be $X \mr{\theta \cdot \{-\}} \Omega_P^X$, \mbox{$\alpha(x) = \theta(x) \cdot \{x\}$.} Then proposition \ref{formulainternaparaG} states that $\theta = \displaystyle \bigvee_{x \in X} \alpha(x)$ (this is internally in $shP$). Applying proposition \ref{calcsup} we compute in $X_d$:

$$\theta = \bigvee_{\stackrel{p \in P}{x \in X(p)}} \Sigma_p^1 (\theta_p(x) \cdot \{x\}_p) = \bigvee_{\stackrel{p \in P}{x \in X(p)}} \theta_p(x) \cdot \Sigma_p^1 \{x\}_p = \bigvee_{\stackrel{p \in P}{x \in X(p)}} \theta_p(x) \cdot \delta_x.$$

We have proved the following:

\begin{\prop} \label{formula}
The family $\{\delta_x\}_{p \in P, x \in X(p)}$ generates $X_d$ as a $P$-module, and furthermore,
for each $\theta \in X_d$, we have $\theta = \displaystyle \bigvee_{\stackrel{p \in P}{x \in X(p)}} \theta_p(x) \cdot \delta_x$. \qed
\end{\prop}



\begin{remark} \label{pavada}
 Given $q \leq p \in P$, $x \in X(p)$, by naturality of $X \mr{\{-\}} \Omega_P^X$ we have $\{x|_q\}_q = \rho_q^p \{x\}_p$.
\end{remark}

\begin{lemma} \label{restringirlosdelta}
 For $p,q \in P$, $x \in X(p)$, we have $q \cdot \delta_x = \delta_{x|_{p \wedge q}}$. In particular $p \cdot \delta_x = \delta_x$.
\end{lemma}

\begin{proof}
 Recall that multiplication by $a \in P$ is given by $\Sigma_a^1 \ \rho_a^1$, and that $\rho_a^1 \ \Sigma_a^1 = id$. Then   \mbox{$p \cdot \delta_x = \Sigma_p^1 \ \rho_p^1 \ \Sigma_p^1 \ \{x\}_p = \Sigma_p^1 \ \{x\}_p = \delta_x$,} and
 
 \flushleft $q \cdot \delta_x = q \cdot p \cdot \delta_x = (p \wedge q) \cdot \delta_x = \Sigma_{p \wedge q}^1 \ \rho_{p \wedge q}^1 \ \Sigma_p^1 \ \{x\}_p = $
 
 \hfill $ = \Sigma_{p \wedge q}^1 \ \rho_{p \wedge q}^p \ \rho_p^1 \ \Sigma_p^1 \ \{x\}_p =  \Sigma_{p \wedge q}^1 \ \rho_{p \wedge q}^p \ \{x\}_p \stackrel{\ref{pavada}}{=} \Sigma_{p \wedge q}^1 \ \{x|_{p \wedge q}\}_{p \wedge q} = \delta_{x|_{p \wedge q}}. $ 
 \end{proof}

\begin{corollary} \label{pyqsaltan}
 For $X,Y \in shP$, $p,q \in P$, $x \in X(p)$, $y \in Y(q)$, we have \mbox{$\delta_x \otimes \delta_y = \delta_{x|_{p \wedge q}} \otimes \delta_{y|_{p \wedge q}}$}.
\end{corollary}

\begin{proof}
 $\delta_x \otimes \delta_y = p \cdot \delta_x \otimes q \cdot \delta_y = q \cdot \delta_x \otimes p \cdot \delta_y =  \delta_{x|_{p \wedge q}} \otimes \delta_{y|_{p \wedge q}}$.
\end{proof}

\begin{\de} \label{notacioncorchetes}
Consider now $X \times X \mr{\delta_X} \Omega_P$ in $shP$, for each $a \in P$ we have $$X(a) \times X(a) \mr{{\delta_X}_a} \Omega_P(a) \stackrel{\ref{enumeratedeJT}, \ item \ 2.}{=} P_{\leq a}.$$
If $x \in X(p)$, $y \in X(q)$ with $p,q \in P$, we denote $$\llbracket x \! = \! y \rrbracket_P := \Sigma_{p \wedge q}^1 {\delta_X}_{p \wedge q} (x|_{p \wedge q}, y|_{p \wedge q}) \in P.$$
\end{\de}

This shouldn't be confused with the internal notation $\llbracket x \! = \! y \rrbracket \in \Omega_P$ in the language of $shP$ introduced in section \ref{relintopos}, 
here all the computations are \emph{external}, i.e. in $\Sat$, and $x,y$ are variables in the language of $\Sat$. 

\begin{lemma}[cf. lemma \ref{ecuacionenL}] \label{ecuacionenOmegaX}
 For $p,q \in P$, $x \in X(p)$, $y \in X(q)$, $\quad \llbracket x \! = \! y\rrbracket_P  \cdot \delta_x = \llbracket x \! = \! y\rrbracket_P  \cdot \delta_y$. 
\end{lemma}
\begin{proof}
 Applying lemma \ref{ecuacionenL} to $X \mr{\{ \}} \Omega_P^X$ it follows that for each $p,q \in P$, $x \in X(p)$, $y \in X(q)$, $${\delta_X}_{p \wedge q} (x|_{p \wedge q}, y|_{p \wedge q})  \cdot \{x|_{p \wedge q}\}_{p \wedge q} = {\delta_X}_{p \wedge q} (x|_{p \wedge q}, y|_{p \wedge q})  \cdot \{y|_{p \wedge q}\}_{p \wedge q}$$
 in $\Omega_P^X(p \wedge q)$, where ``$\cdot$'' is the $p \wedge q$-component of the natural isomorphism $\Omega_P \otimes \Omega_P^X \mr{\cdot} \Omega_P^X$. Apply now $\Sigma_{p \wedge q}^1$ and use that ``$\cdot$'' is a $s \ell$-morphism (therefore it commutes with $\Sigma$) to obtain 
 $$\llbracket x \! = \! y\rrbracket_P  \cdot \delta_{x|_{p \wedge q}} = \llbracket x \! = \! y\rrbracket_P  \cdot \delta_{y|_{p \wedge q}}.$$
 Then, by lemma \ref{restringirlosdelta},
 $$\llbracket x \! = \! y\rrbracket_P  \cdot q \cdot \delta_{x} = \llbracket x \! = \! y\rrbracket_P  \cdot p \cdot \delta_{y},$$
 which since $\llbracket x \! = \! y\rrbracket_P \leq p \wedge q$ is the desired equation.
 \end{proof}

\begin{sinnadastandard} \label{axiomasparamodulos}
Let $X,Y \in shP, H \in Loc(shP)$, then we have the correspondence

$$\begin{tabular}{c}
 $X \times Y \mr{\lambda} H$ an $\ell$-relation \\ \hline \noalign{\smallskip}
 $\Omega_P^X \otimes \Omega_P^Y \mr{\lambda} H$ a $s \ell$-morphism \\ \hline \noalign{\smallskip}
 $X_d \otimes_P Y_d \mr{\mu} H(1)$ a morphism of $P$-$Mod$
\end{tabular}$$

The following propositions show how $\mu$ is computed from $\lambda$ and vice versa.
\end{sinnadastandard}

\begin{proposition} \label{muapartirdelambda}
 In \ref{axiomasparamodulos}, for $p,q \in P$, $x \in X(p)$, $y \in Y(q)$, $\mu(\delta_x \otimes \delta_y) = \Sigma_{p \wedge q}^1 \lambda_{p \wedge q} (x|_{p \wedge q},y|_{p \wedge q}).$
\end{proposition}
\begin{proof}
 $\mu(\delta_x \otimes \delta_y) \stackrel{\ref{pyqsaltan}}{=} \lambda_1(\delta_{x|_{p \wedge q}} \otimes \delta_{y|_{p \wedge q}}) = \lambda_1 \ \Sigma_{p \wedge q}^1 \ (\ \{x|_{p \wedge q}\}_{p \wedge q} \otimes \{y|_{p \wedge q}\}_{p \wedge q} \ ) =$
 
 \vspace{1ex}
 
\hfill $  =   \Sigma_{p \wedge q}^1 \ \lambda_{p \wedge q} \ (\ \{x|_{p \wedge q}\}_{p \wedge q} \otimes \{y|_{p \wedge q}\}_{p \wedge q} \ ) = \Sigma_{p \wedge q}^1 \ \lambda_{p \wedge q} \ (\ x|_{p \wedge q} \ ,\ y|_{p \wedge q} \ ).$
\end{proof}

\begin{corollary} \label{lambdaapartirdemu}
 Applying $\rho_{p \wedge q}^1$ and using $\rho_{p \wedge q}^1 \ \Sigma_{p \wedge q}^1 = id$, we obtain the reciprocal computation
 $$\lambda_{p \wedge q} \ (\ x|_{p \wedge q} \ , \ y|_{p \wedge q} \ ) = \rho_{p \wedge q}^1 \ \mu(\delta_x \otimes \delta_y). $$
 
 \vspace{-4ex}
 
 $\hfill \qed$
\end{corollary}

\begin{remark} \label{notacionlinda}
In \ref{axiomasparamodulos}, if $\lambda = \delta_X: X \times X \mr{} \Omega$, then \mbox{$\mu(\delta_{x_1} \otimes \delta_{x_2}) = \igu[x_1]{x_2}_P$} (recall \ref{notacioncorchetes}).
\end{remark}

\begin{lemma} \label{multporr}
  In \ref{axiomasparamodulos}, for each $p,q,r \in P$, $x \in X(p)$, $y \in Y(q)$, 
   $$r \cdot \mu(\delta_x \otimes \delta_y) = \Sigma_{p \wedge q \wedge r}^1 \; \rho_{p \wedge q \wedge r}^{p \wedge q} \; \lambda_{p \wedge q} \ (\ x|_{p \wedge q}\ ,\ y|_{p \wedge q}\ ) = \Sigma_{p \wedge q \wedge r}^1 \; \lambda_{p \wedge q \wedge r} \ ( \ x|_{p \wedge q \wedge r} \ , \ y|_{p \wedge q \wedge r} \ ).$$
\end{lemma}

\begin{proof} The second equality is just the naturality of $\lambda$. To prove the first one, we compute:

\flushleft $r \cdot \mu(\delta_x \otimes \delta_y) \stackrel{\ref{muapartirdelambda}}{=} \Sigma_r^1 \ \rho_r^1 \ \Sigma_{p \wedge q}^1 \ \lambda_{p \wedge q} \ (\ x|_{p \wedge q} \ , \ y|_{p \wedge q} \ ) \stackrel{\ref{enumeratedeJT} \ item \ 2.b)}{=}$
 
\hfill $ = \Sigma_r^1 \ \Sigma_{p \wedge q \wedge r \ }^r \ \rho_{p \wedge q \wedge r \ }^{p \wedge q} \ \lambda_{p \wedge q} \ (\ x|_{p \wedge q} \ , \ y|_{p \wedge q} \ ) = \Sigma_{p \wedge q \wedge r \ }^1 \ \rho_{p \wedge q \wedge r \ }^{p \wedge q} \ \lambda_{p \wedge q} \ (\ x|_{p \wedge q} \ , \ y|_{p \wedge q} \ )$.
\end{proof}

The following proposition expresses the corresponding formulae for the four axioms of an $\ell$-relation $X \times Y \mr{\lambda} H$ in $shP$ (see definition \ref{4axioms}), at the level of $P$-modules. 

\begin{\prop} \label{propaxiomparamodulos}
In \ref{axiomasparamodulos},
$\lambda$ is \emph{ed} (resp. \emph{uv, su, in}) if and only if:
\begin{itemize}
 \item ed) for each $p \in P$, $x \in X(p)$, $\displaystyle \bigvee_{\stackrel{q \in P}{y \in Y(q)}} \mu(\delta_x \otimes \delta_y) = p$.
 \item uv) for each $p,q_1,q_2 \in P$, $x \in X(p)$, $y_1 \in Y(q_1)$, $y_2 \in Y(q_2)$, 
 
 \hfill $\mu(\delta_x \otimes \delta_{y_1}) \wedge \mu(\delta_x \otimes \delta_{y_2}) \leq \llbracket y_1 \! = \! y_2\rrbracket_P.$
 \item su) for each $q \in P$, $y \in Y(q)$, $\displaystyle \bigvee_{\stackrel{p \in P}{x \in X(p)}} \mu(\delta_x \otimes \delta_y) = q$.
 \item in) for each $p_1,p_2,q \in P$, $x_1 \in X(p_1)$, $x_2 \in X(p_2)$, $y \in Y(q)$, 

 \hfill $\mu(\delta_{x_1} \otimes \delta_y) \wedge \mu(\delta_{x_2} \otimes \delta_y) \leq \llbracket x_1 \! = \! x_2\rrbracket_P.$ 
\end{itemize}
\end{\prop}

\begin{proof}

By proposition \ref{edyuvdafunctionG} and remark \ref{reldomegaYremarkparaG}, $\lambda$ is $ed)$ if and only if $\displaystyle \bigvee_{y \in Y} \lambda^*(y) = 1$ in $H^X$.
By proposition \ref{calcsup} and remark \ref{1deGalaX}, this is an equality of global sections $\displaystyle \bigvee_{\stackrel{q \in P}{y \in Y(q)}} \Sigma_q^1 \lambda^*_q (y) = 1$ in \mbox{$H^X(1) \stackrel[\textcolor{white}{M}]{}{=} [X,H]$.} Then $\lambda$ is $ed)$ if and only if for each $p \in P$, $x \in X(p)$, $\displaystyle \bigvee_{\stackrel{q \in P}{y \in Y(q)}} (\Sigma_q^1 \lambda^*_q (y))_p (x) = p$ in $H(p)$. But by $F\Sigma)$ in lemma \ref{exponencialenshP} we have 
$$(\ \Sigma_q^1 \ \lambda^*_q (y) \ )_p (x) \stackrel[\textcolor{white}{M}]{}{=}  \Sigma_{p \wedge q}^p \ (\ \lambda^*_q (y) \ )_{p \wedge q} \ (x|_{p \wedge q}) = \Sigma_{p \wedge q}^p \ \lambda_{p \wedge q} \ ( \ x|_{p \wedge q} \ ,\  y|_{p \wedge q} \ ),$$
where last equality holds since by definition of $\lambda^*$ we have $(\lambda^*_q (y))^{\textcolor{white}{M}}_{p \wedge q} (x|_{p \wedge q}) = \lambda_{p \wedge q} (x|_{p \wedge q}, y|_{p \wedge q})$.

\vspace{.2cm}

We conclude that $\lambda$ is $ed)$ if and only if for each $p \in P$, $x \in X(p)$, 
$$\displaystyle \bigvee_{\stackrel{q \in P}{y \in Y(q)}} \Sigma_{p \wedge q}^p \ \lambda_{p \wedge q} \ (\ x|_{p \wedge q} \ , \ y|_{p \wedge q} \ ) = p \hbox{ in } H(p).$$ 
Since $\rho_p^1 \ \Sigma_p^1 = id$, this holds if and only if it holds after we apply $\Sigma_p^1$. Then, proposition \ref{muapartirdelambda} yields the desired equivalence.

\vspace{.3cm}

\noindent We now consider axiom uv):

$\lambda$ is $uv)$ if and only if for each $p,q_1,q_2 \in P$, $x \in X(p)$, $y_1 \in Y(q_1)$, $y_2 \in Y(q_2)$,

\vspace{.2cm}

\noindent $ \displaystyle \rho_{p \wedge q_1 \wedge q_2}^{p \wedge q_1} \ \lambda_{p \wedge q_1} \ ( x|_{p \wedge q_1} \ , \ y_1|_{p \wedge q_1} ) \; \wedge \; \rho_{p \wedge q_1 \wedge q_2}^{p \wedge q_2} \ \lambda_{p \wedge q_2} \ ( x|_{p \wedge q_2} \ , \ y_2|_{p \wedge q_2} ) \; \; \leq $

\vspace{.1cm}

\hfill $\rho_{p \wedge q_1 \wedge q_2}^{q_1 \wedge q_2} \ {\delta_Y}_{q_1 \wedge q_2} \ ( y_1|_{q_1 \wedge q_2} \ , \ y_2|_{q_1 \wedge q_2} )$.

\vspace{.2cm}

We apply $\Sigma_{p \wedge q_1 \wedge q_2}^1$ and use that it commutes with $\wedge$ to obtain that this happens if and only if 

\vspace{.2cm}

\noindent $ \displaystyle \Sigma_{p \wedge q_1 \wedge q_2}^1 \ \rho_{p \wedge q_1 \wedge q_2}^{p \wedge q_1} \ \lambda_{p \wedge q_1} \ ( x|_{p \wedge q_1} \ , \ y_1|_{p \wedge q_1}  ) \; \wedge \; \Sigma_{p \wedge q_1 \wedge q_2}^1 \ \rho_{p \wedge q_1 \wedge q_2}^{p \wedge q_2} \ \lambda_{p \wedge q_2} \ ( x|_{p \wedge q_2} \ , \ y_2|_{p \wedge q_2} ) \;\; \leq $

\vspace{.1cm}

\hfill $ \Sigma_{p \wedge q_1 \wedge q_2}^1 \ \rho_{p \wedge q_1 \wedge q_2}^{q_1 \wedge q_2} \ {\delta_Y}_{q_1 \wedge q_2} \ ( y_1|_{q_1 \wedge q_2} \ , \ y_2|_{q_1 \wedge q_2} )$,

\vspace{.2cm}

\noindent which by lemma \ref{multporr} (see remark \ref{notacionlinda}) is equation 

$$ \displaystyle q_2 \cdot \mu(\delta_x \otimes \delta_{y_1} ) \wedge q_1 \cdot \mu(\delta_x \otimes \delta_{y_2} ) \leq p \cdot \igu[y_1]{y_2}_P,$$
but since $q_i \cdot \delta_{y_i} = \delta_{y_i}$ ($i=1,2$), this is equivalent to the equation 
$$\mu(\delta_{x} \otimes \delta_{y_1} ) \wedge \mu(\delta_{x} \otimes \delta_{y_2} ) \leq p \cdot \llbracket y_1 \! = \! y_2 \rrbracket_P.$$

This equation is equivalent to the desired one since the right term is smaller or equal than $\llbracket y_1 \! = \! y_2 \rrbracket_P$, and multiplying by $p$ the left term doesn't affect it.
\end{proof}

\begin{\de} \label{definicionesparamu}
In \ref{axiomasparamodulos}, we say that $\mu$ is \emph{ed} (resp. \emph{uv, su, in}) if it satisfies the corresponding condition of proposition \ref{propaxiomparamodulos} above. We say that $\mu$ is an \emph{$\ell$-function} if it is $ed$ and $uv$, and that $\mu$ is an $\ell$-bijection if it is $ed$, $uv$, $su$ and $in$.

Note that $\mu$ has each of the properties defined above if and only if $\lambda$ does.
\end{\de}

Consider now the self-duality of $\Omega_P^X$ in $s \ell(shP)$ given by proposition \ref{autodualparaG}. Applying the tensor equivalence $s \ell(shP) \mr{\gamma_*} P$-$Mod$ it follows that $X_d$ is self-dual as a $P$-module, (see definition \ref{moddual} and remark \ref{nuevodualcomobimod}). We will now give the formulae for the $\eta$, $\eps$ of this duality.

\begin{\prop} \label{formulaeetaeps}
 The $P$-module morphisms $P \mr{\eta} X_d \stackrel[P]{}{\otimes} X_d$, $X_d \stackrel[P]{}{\otimes} X_d \mr{\eps} P$ are given by the formulae $\eta(1) = \displaystyle \bigvee_{\stackrel{p \in P}{x \in X(p)}} \delta_x \otimes \delta_x$, $\eps(\delta_x \otimes \delta_y) = \llbracket x \! = \! y \rrbracket_P$ for each $p,q \in P$, $x \in X(p)$, $y \in X(q)$.
\end{\prop}

\begin{proof}
 The internal formula for $\eta$ given in the proof of proposition \ref{autodualparaG}, together with proposition \ref{calcsup} yield the desired formula for $\eta$. The internal formula for $\eps$, together with our definition of the notation \mbox{$\llbracket x \! = \! y \rrbracket_P$} yield that for each $p,q \in P$, $x \in X(p)$, $y \in X(q)$, we have 
 \mbox{$\eps_{p \wedge q}(\{x|_{p \wedge q}\}_{p \wedge q} \otimes \{y|_{p \wedge q}\}_{p \wedge q}) = \llbracket x \! = \! y \rrbracket_P$} in $\Omega_P(p \wedge q)$. Apply $\Sigma_{p \wedge q}^1$, use that it commutes with the $s \ell$-morphism $\eps$ and recall remark \ref{1deGalaX} to obtain 
 \mbox{$\eps_1( \delta_{x|_{p \wedge q}} \otimes \delta_{y|_{p \wedge q}}) = \llbracket x \! = \! y \rrbracket_P$ in $P$,} which by corollary \ref{pyqsaltan} is the desired equation.
 \end{proof}

\begin{sinnadastandard} \label{particular} {\bf A particular type of $\ell$-relation.} 
Assume $P$ is a coproduct of two locales, \mbox{$P = A \otimes B$.} Then the inclusions into the coproduct yield projections from the product of topoi   \mbox{$shA \stackrel{\pi_1}{\longleftarrow} sh(A \otimes B) \stackrel{\pi_2}{\longrightarrow} shB$.}
\end{sinnadastandard}

Consider now $X \in shA,$ $Y \in shB$, $H \in Loc(sh(A \otimes B))$. We can consider an \mbox{$\ell$-relation} \mbox{$\pi_1^*X \times \pi_2^*Y \mr{\lambda} H$,} and the corresponding $(A \otimes B)$-module morphism \mbox{$(\pi_1^*X)_d \stackrel[A \otimes B]{}{\otimes} (\pi_2^*Y)_d \mr{\mu} H(1)$.}

To compute $(\pi_1^*X)_d$, note that $X_d$ is the $A$-module corresponding to the locale of open parts of the discrete space $X_{dis}$ (recall corollary \ref{discretespace}). By \cite[VI.3 Proposition 3, p.51]{JT}, $A \mr{} X_d$ is the morphism of locales corresponding to the etale (over $\overline{A}$) space $X_{dis} = \Omega_A^X$. 
Then we have the following pull-back of spaces (push-out of locales)

$$\xymatrix{(\pi_1^*X)_{dis} \ar[r] \ar[d] & X_{dis} \ar[d] \\
	      \overline{A \otimes B} \ar[r] & \overline{A}} \quad \quad \quad \quad
 \xymatrix{(\pi_1^*X)_d & X_d \ar[l]  \\
	      A \otimes B \ar[u] & A \ar[l] \ar[u] }$$       

\noindent which shows that $(\pi_1^*X)_d = X_d \otimes B$, and similarly $(\pi_2^*Y)_d = A \otimes Y_d$. Then we have 
$$(\pi_1^*X)_d \stackrel[A \otimes B]{}{\otimes} (\pi_2^*Y)_d = (X_d \otimes B) \stackrel[A \otimes B]{}{\otimes} (A \otimes Y_d) \cong X_d \otimes Y_d,$$
where the last tensor product is the tensor product of sup-lattices in $\Sat$, i.e. as $\Omega$-modules. The isomorphism maps $\delta_x \otimes \delta_y \mapsto (\delta_x \otimes 1) \otimes (1 \otimes \delta_y)$, then we have the following instance of proposition \ref{propaxiomparamodulos}. 

\begin{\prop} \label{propaxiomparamodulos2}
Let $X \in shA$, $Y \in shB, H \in Loc(sh(A \otimes B))$, and an $\ell$-relation   \mbox{$\pi_1^*X \times \pi_2^*Y \mr{\lambda} H$.} Consider the corresponding $(A \otimes B)$-module morphism \mbox{$X_d \otimes Y_d \mr{\mu} H(1)$.} Then $\lambda$ is $ed, uv, su, in$ resp. if and only if: 
\begin{itemize}
 \item ed) for each $a \in A$, $x \in X(a)$, $\displaystyle \bigvee_{\stackrel{b \in B}{y \in Y(b)}} \mu(\delta_x \otimes \delta_y) = a$.
 \item uv) for each $a \in A$, $b_1,b_2 \in B$, $x \in X(a)$, $y_1 \in Y(b_1)$, $y_2 \in Y(b_2)$, 
 
 \hfill $\mu(\delta_x \otimes \delta_{y_1}) \wedge \mu(\delta_x \otimes \delta_{y_2}) \leq \llbracket y_1 \! = \! y_2\rrbracket_P.$
 \item su) for each $b \in B$, $y \in Y(b)$, $\displaystyle \bigvee_{\stackrel{a \in A}{x \in X(a)}} \mu(\delta_x \otimes \delta_y) = b$.
 \item in) for each $a_1,a_2 \in A$, $b \in B$, $x_1 \in X(a_1)$, $x_2 \in X(a_2)$, $y \in Y(b)$, 
 
 \hfill $\mu(\delta_{x_1} \otimes \delta_y) \wedge \mu(\delta_{x_2} \otimes \delta_y) \leq \llbracket x_1 \! = \! x_2\rrbracket_P.$
\end{itemize}
\qed
\end{\prop}

\section{$\rhd$ and $\lozenge$ cones} \label{sec:cones}

In this section we generalize the results of \cite[4.]{DSz}, in two ways, both needed for our purpose. Like before, we work over any arbitrary topos $\Sat$ instead of over $Set$, and we develop a theory of $\rhd$ and $\lozenge$ cones for two different functors $F,F'$ instead of just one. As in the previous section, we omit the proofs when they are easily obtained adapting the ones in op. cit.

\begin{sinnadastandard} \label{extensionarel}
Recall that the category $Rel(\Eat)$ of relations of a topos $\Eat$ is a $s \ell$-category. 
An inverse image $\xymatrix{\cc{E} \ar[r]^{F} & \cc{S}}$ of a geometric morphism respects products and subobjects, thus it induces a \mbox{$s \ell$-functor} $Rel(\cc{E}) \xr{T = Rel(F)} Rel(\cc{S})$. On objects $TX = FX$, and the value of $T$ in a relation $R \mmr{} X \times Y$ in $\cc{E}$ is the relation $FR \mmr{} FX \times FY$ in $\cc{S}$. In particular, for arrows $f$ in $\cc{E}$, $T(R_f) = R_{F(f)}$ (see \ref{arrowsvsfunctions}), or, if we abuse the notation by identifying $f$ with the relation given by its graph, $T(f) = F(f)$. It is immediate from the definition that for every relation $R$ in $\cc{E}$ we have $T(R^{op}) = (TR)^{op}$.
\end{sinnadastandard}

%
%
%

Consider now two geometric morphisms with inverse images $\xymatrix{\cc{E} \ar@<1ex>[r]^{F} \ar@<-1ex>[r]_{F'} & \cc{S}}$, and their respective extensions to the $Rel$ categories 
 $\xymatrix{Rel(\cc{E}) \ar@<1ex>[r]^{T} \ar@<-1ex>[r]_{T'} & Rel(\cc{S})}$.

\begin{\de} \label{defdeconos}
Let $H$ be a sup-lattice in $\Sat$. A cone $\lambda$ (with vertex $H$) is a family of $\ell$-relations $FX \times F'X \mr{\lambda_{X}} H$, one for each $X \in \cc{E}$. Note that, a priori, a cone is just a family of arrows without any particular property. This isn't standard terminology, but we do this in order to use a different prefix depending on which diagrams commute. Each arrow $X \mr{f} Y$ in 
$\cc{E}$ and each arrow 
$X \mr{R} Y$ in $\cc{R}el(\cc{E})$ (i.e relation 
$R \mmr{} X \times Y$ in $\cc{E}$), determine the following diagrams:

\begin{center}
 \begin{tabular}{cc}
  $\rhd(f) = \rhd(F(f),F'(f))$ & $\lozenge(R) = \lozenge(TR,T'R)$ \\ \noalign{\smallskip}
  $\xymatrix@C=4ex@R=3ex
        {
         FX \times F'X \ar[rd]^{\lambda_{X}}  
                      \ar[dd]_{F(f) \times F'(f)}  
        \\
         {} \ar@{}[r]^(.3){\geq}
         &  H    
        \\
         FY \times F'Y \ar[ru]_{\lambda_{Y}} 
         } $ &
  $\xymatrix@C=4ex@R=3ex
        {
         & TX \times T'X \ar[rd]^{\lambda_{X}}  
        \\
           TX \times T'Y 
               \ar[rd]_{TR \times T'Y  \hspace{2.5ex}} 
			   \ar[ru]^{TX \times T'R^{op} \hspace{2.5ex}} 
	     & \equiv 
         & H
        \\
         & TY \times T'Y \ar[ru]_{\lambda_{Y}} 
         }$
\\ \noalign{\smallskip} \noalign{\smallskip} 
$\lozenge_1(f) = \lozenge_1(F(f),F'(f))$ & $\lozenge_2(f) = \lozenge_2(F(f),F'(f)) $ \\ \noalign{\smallskip}
 ${ \xymatrix@C=1.4ex@R=3ex {
                  & FX \times F'X  \ar[rd]^{\lambda_X} 
                  \\
			        FX \times F'Y \ar[rd]_{F(f) \times F'Y \hspace{2.5ex}} 
			                    \ar[ru]^{FX \times F'(f)^{op} \hspace{2.5ex}} & \equiv & H
			      \\
			       & FY \times F'Y \ar[ru]_>>>>{\lambda_Y} 
			      } 
}$ &
$\xymatrix@C=1.4ex@R=3ex
                 {
                  & FX \times F'X  \ar[rd]^{\lambda_X} 
                  \\
			        FY \times F'X \ar[rd]_{FY \times F'(f) \hspace{2.5ex}} 
			                    \ar[ru]^{F(f)^{op} \times F'X \hspace{2.5ex}} & \equiv & H
			      \\
			       & FY \times F'Y \ar[ru]_>>>>{\lambda_Y} 
		} $
 \end{tabular}
\end{center}

\noindent $\lambda$ is a \emph{$\rhd$-cone} if the $\rhd(f)$ diagrams  hold, and $\lambda$ is a 
\emph{$\lozenge$-cone} if the $\lozenge(R)$ diagrams hold. Similarly we define \emph{$\lozenge_1$-cones} and \emph{$\lozenge_2$-cones} if the 
$\lozenge_1(f)$ and $\lozenge_2(f)$ \mbox{diagrams} hold. 
If $H$ is a locale and the $\lambda_X$ are $\ell$-functions, $\ell$-bijections, we say that we have a cone \textit{of $\ell$-functions, $\ell$-bijections}.
\end{\de}

The propositions \ref{diamante12igualdiamante} and \ref{combinacion} yield the following corresponding results for cones.
  
\begin{proposition} \label{dim1ydim2esdim}
A cone $FX \times F'X \mr{\lambda} H$ is a $\lozenge$-cone if and only if it is both a $\lozenge_1$ and a $\lozenge_2$-cone. \qed
\end{proposition}

\begin{proposition}\label{trianguloesdiamante}
A cone $FX \times F'X \mr{\lambda} H$ of $\ell$-bijections is a $\rhd$-cone if and only if it is a $\lozenge$-cone. \qed
\end{proposition}

\begin{sinnadastandard} {\bf Cones and natural transformations.} In order to express the universal property defining the groupoid of \cite[VIII.3 Theorem 2]{JT} as a universal property of $\rhd$-cones (see theorem \ref{JTGeneral}), it is necessary to relate cones with natural transformations and to analyze their behavior through topoi morphisms. The following proposition shows that $\lozenge_1$-cones of functions correspond to natural transformations. 

\begin{proposition} \label{naturaligualcone}
 Consider a family of arrows $FX \mr{\theta_X} F'X$, one for each $X \in \cc{E}$. Each $\theta_X$ corresponds by proposition \ref{functionconallegories} to a function $FX \times F'X \xr{\varphi_X = \lambda_{\theta_X}} \Omega$ yielding in this way a cone $\varphi$. Then $\theta$ is a natural transformation if and only if $\varphi$ is a $\lozenge_1$-cone. 
\end{proposition}

\begin{proof}

By proposition \ref{prop:imageninversausandoautodualparaG} (recall also remark \ref{1erremarkbasico}) we have that the correspondence between $\ell FX \xr{P(\theta_X) = (\varphi_X)_*} \ell F'X$ and $\ell FX \otimes \ell F'X  \mr{\varphi_X} \Omega$ is given by the self-duality of $F'X$. 
As with every duality, this correspondence is given by the following diagrams in the monoidal category $s\ell$ (we omit to write the $\ell$, and think of these diagrams as diagrams of relations):

$$\vcenter{\xymatrix{ \varphi_X:    }} 
\vcenter{\xymatrix@C=-.3pc{  FX \dcell{\theta_X} && F'X \did \\ 
			    F'X && F'X \\
			  & \quad \cl{\eps} }}
\vcenter{\xymatrix{ \quad\quad \quad \theta_X:    }}
\vcenter{\xymatrix@C=-.3pc{  FX \did && & \op{\eta} \\
		    FX && F'X && F'X \did \\
		    & \quad  \cl{\varphi_X} &&& F'X    }}$$

Also, the naturality $N$ of theta and the $\lozenge_1$ diagrams (recall from corollary \ref{dualintercambiaparaG} that $f^{op} = f^{\wedge}$) can be expressed as follows: for each $X \mr{f} Y$, 
$$\vcenter{\xymatrix{ N(f): \;\; }}
\vcenter{\xymatrix@C=-0.3pc{  FX \dcellbb{F(f)} \\ FY \dcellb{\theta_Y} \\ F'Y  }}
\vcenter{\xymatrix{ \;\;\; = \;\;\; }}
\vcenter{\xymatrix@C=-0.3pc{  FX \dcellb{\theta_X} \\ F'X \dcellbb{F' \! (f)} & \quad , \;\;\;\; \\ F'Y  }}
\vcenter{\xymatrix{  \lozenge_1(f):  }}
\vcenter{\xymatrix@C=-0.3pc{ FX \did &&& \op{\eta} &&& F'Y \did \\              
	      FX \did && F'X \did && F'X \dcellbb{F' \! (f)} && F'Y \did \\
	      FX && F'X  && F'Y && F'Y  \\ 
            & \quad  \cl{\varphi_X}  &&&&   \quad \cl{\eps}          }}
\vcenter{\xymatrix{ \;\;\; = \;\;\; }}
\vcenter{\xymatrix@C=-0.3pc{FX \dcellbb{F(f)} & & F'Y \did
          \\              
             FY & & F'Y
          \\
           & \quad  \cl{\varphi_Y} }}$$

\noindent $\underline{N(f) \Rightarrow \lozenge_1(f):}$ replace $\theta$ as in the correspondence above in $N(f)$ to obtain

$$\vcenter{\xymatrix@C=-0.3pc{ FX \did &&& \op{\eta} \\
			    FX && F'X && F'X \dcellbb{F' \! (f)} \\
			     & \quad \cl{\varphi_X} &&& F'Y }}
\vcenter{\xymatrix@C=-0.3pc{ \quad \stackrel{N(f)}{=} \quad }}
\vcenter{\xymatrix@C=-0.3pc{   FX \dcellbb{F(f)} &&& \op{\eta} \\
			    FY && F'Y && F'Y \did \\
			    & \quad \cl{\varphi_Y} &&& F'Y }} $$

Compose with $\eps$ and use a triangular identity to obtain $\lozenge_1(f)$.

\noindent $\underline{\lozenge_1(f) \Rightarrow N(f):}$ replace $\varphi$  as in the correspondence above in $\lozenge_1(f)$ to obtain
$$\vcenter{\xymatrix@C=-0.3pc{ FX \dcellbb{F(f)} && F'Y \did \\
				      FY \dcell{\theta_Y} && F'Y \did \\
				      F'Y && F'Y \\
				      & \cl{\eps} }}
\vcenter{\xymatrix@C=-0.3pc{ \stackrel{\quad \lozenge_1(f)}{=} \quad }}
\vcenter{\xymatrix@C=-0.3pc{  FX \dcell{\theta_X} &&& \op{\eta} &&& F'Y \did \\
				      F'X && F'X && F'X \dcellbb{F' \! (f)} && F'Y \did \\
				      & \cl{\eps} &&& F'Y && F'Y \\
				      &&&&& \cl{\eps} }}
\vcenter{\xymatrix@C=-0.3pc{\quad \stackrel{\triangle}{=} \quad}}
\vcenter{\xymatrix@C=-0.3pc{ FX \dcell{\theta_X} && F'Y \did \\
				    F'X \dcellbb{F(f)} && F'Y \did \\
				    F'Y && F'Y \\
				    & \cl{\eps} }}$$

Compose with $\eta$ and use a triangular identity to obtain $N(f)$.
\end{proof}

\begin{sinnadastandard} \label{sin:naturaligualconeconrho}
Consider now the previous situation together with a topos over $\Sat$, 
$$\vcenter{\xymatrix@C=3.5pc@R=3.5pc{& \cc{H}  \adjuntosd{h^*}{h_*}  \\ \cc{E} \dosflechasr{F}{F'} & \cc{S} }}$$
By proposition \ref{naturaligualcone}, a natural transformation $h^* FX \mr{\theta_X} h^* F'X$ corresponds to a $\lozenge_1$-cone of functions \mbox{$h^* FX \times h^* F'X \mr{\varphi_X} \Omega_{\cc{H}}$} in $\cc{H}$. As established in \ref{extensionarel}, $h^*$ can be extended to \mbox{$Rel = s \ell_0$} as a tensor functor (therefore preserving duals), then using the naturality of the adjunction $h^* \dashv h_*$ it follows that $h^* FX \times h^* F'X \mr{\varphi_X} \Omega_{\cc{H}}$ is a $\lozenge_1$-cone if and only if $FX \times F'X \mr{\lambda_X  } h_* \Omega_{\cc{H}}$ is a $\lozenge_1$-cone (in $\cc{S}$). We have:

\begin{\prop} \label{naturaligualconeconrho}
 A family of arrows $h^* FX \mr{\theta_X} h^* F' X$ (one for each $X \in \cc{E}$) is a natural transformation if and only if the corresponding cone $FX \times F'X \mr{\lambda_X} h_* \Omega_{\cc{H}}$ is a $\lozenge_1$-cone. \qed 
\end{\prop}
\end{sinnadastandard}

Combining propositions \ref{naturaligualconeconrho} and \ref{prop:relativizarellrelations} we obtain the following corollary for the case $\cc{H} = shP$ (which by remark \ref{remark:relativizarellrelations} also holds for an arbitrary $\cc{H}$ as in \ref{sin:naturaligualconeconrho}):

\begin{corollary} \label{conodeellrelationsesnattransf} Given 
$\vcenter{\xymatrix@C=2.5pc@R=2.5pc{& shP  \adjuntosd{\gamma^*}{\gamma_*}  \\ \cc{E} \dosflechasr{F}{F'} & \cc{S} }}$, 
 the adjunction $\gamma^* \dashv \gamma_*$ yields a bijective correspondence between $\lozenge_1$-cones of $\ell$-functions (resp $\ell$-bijections) $FX \times F'X \mr{\lambda_X} P$ and natural transformations (resp. isomorphisms) $\gamma^* F \Mr{\varphi} \gamma^* F'$. \qed
\end{corollary}

\begin{sinnadastandard} \label{atravesdeunmorfismodetopos}
 Consider finally the previous situation together with a morphism $\cc{F} \mr{} \cc{H}$ of topoi over $\Sat$, as in the following commutative diagram:

$$\xymatrix@R=3.5pc{& \cc{H} \df{rr}{{{a}}^*}{{{a}}_*} && \cc{F}  \\
	  \cc{E} \ar@<.5ex>[rr]^F \ar@<-.5ex>[rr]_{F'} && \cc{S} \dfbis{ul}{h^*}{h_*} \dfbis{ur}{f^*}{f_*}  }$$

Consider the locales in $\Sat$ of subobjects of $1$ in $\cc{H}$, resp. $\cc{F}$, $H:= h_* \Omega_{\cc{H}}$, $L:= f_* \Omega_{\cc{F}}$. Since ${{a}}^*$ is an inverse image, it maps subobjects of $1$ to subobjects of $1$ and thus induces a locale morphism $H \mr{{{a}}^*} L$. 
\end{sinnadastandard}

\begin{remark} \label{caractsubobjeto}
Consider the comparison morphism ${{a}}^* \Omega_{\cc{H}} \mr{\phi_1} \Omega_{\cc{F}}$, which is the characteristic function of the subobject $1 \hookrightarrow {{a}}^*\Omega_{\cc{H}}$  given by ${{a}}^*(t)$. 
Let $A \hookrightarrow X$ be a subobject in $\cc{H}$. We will apply (the first part of) remark \ref{hojafinal} with $f = {{a}}^*(\phi_A)$, to the subobject $1 \hookrightarrow {{a}}^*\Omega_{\cc{H}}$. Since ${{a}}^*$ preserves pull-backs, ${{a}}^*A = f^{-1}1$. We obtain that $\phi_{{{a}}^*A} = \phi_1 \circ {{a}}^* (\phi_A)$.
%
%
\end{remark}

\begin{proposition} \label{paraelcoro}
In the situation of \ref{atravesdeunmorfismodetopos}, for $X,Y \in \Sat$, if $X \times Y \mr{\lambda} H$ corresponds to the relation  \mbox{$h^* X \times h^* Y \mr{\varphi} \Omega_{\cc{H}}$} via the adjunction $h^* \dashv h_*$, then $X \times Y \mr{\lambda} H \mr{{{a}}^*} L$ corresponds to the relation \mbox{$f^* X \times f^* Y \mr{{{a}}^*(\varphi)} {{a}}^* \Omega_{\cc{H}} \mr{\phi_1} \Omega_{\cc{F}}$} via the adjunction $f^* \dashv f_*$.
\end{proposition} 

\begin{proof} The adjunction $f^* \dashv f_*$ consists of composing the adjunctions $h^* \dashv h_*$ and ${{a}}^* \dashv {{a}}_*$, then we have:

\begin{center}
 \begin{tabular}{c}
  ${{a}}^* h^* X \times {{a}}^* h^* Y  \xr{{{a}}^*(\varphi)} {{a}}^* \Omega_{\cc{H}} \mr{\phi_1} \Omega_{\cc{F}}$ \\ \hline  \noalign{\smallskip}
  $h^* X \times h^* Y  \mr{\varphi} \Omega_{\cc{H}} \mr{\psi_1} {{a}}_* \Omega_{\cc{F}}$ \\ \hline  \noalign{\smallskip}
  $X \times Y  \mr{\lambda} H \xr{h_*(\psi_1)} L, $
 \end{tabular}
\end{center}

\noindent where $\psi_1$ corresponds to $\phi_1$ in the adjunction ${{a}}^* \dashv {{a}}_*$. So we have to check that \mbox{$h_*(\psi_1) = {{a}}^*$.} 
Let a subobject $U \hookrightarrow 1$. 
This subobject can be considered in $H = h_* \Omega_{\cc{H}} = [1, \Omega_{\cc{H}}]$ via its characteristic function $\phi_U$. Now, $h_* (\psi_1) (\phi_U)$ is the composition $1 \mr{\phi_U} \Omega_{\cc{H}} \mr{\psi_1} {{a}}_* \Omega_{\cc{F}}$ in $h_* {{a}}_* \Omega_{\cc{F}}$, and the corresponding arrow $1 \mr{} \Omega_{\cc{F}}$ is given by the adjunction ${{a}}^* \dashv {{a}}_*$. But this arrow is $1 \xr{{{a}}^* (\phi_U)} {{a}}^* \Omega_{\cc{H}} \mr{\phi_1} \Omega_{\cc{F}}$, which by remark \ref{caractsubobjeto} is $\phi_{{{a}}^*U}$, and we are done.
\end{proof}

\begin{corollary}  \label{naturaligualconeconrhoatravesdeunmorfismodetopos}
 In the hypothesis of \ref{atravesdeunmorfismodetopos}, consider a natural transformation \mbox{$h^* FX \mr{\theta_X} h^* F'X$} and the corresponding $\lozenge_1$-cone $FX \times F'X \mr{\lambda_X} H$ obtained in proposition \ref{naturaligualconeconrho}. Then the $\lozenge_1$-cone with vertex $L$ corresponding by proposition \ref{naturaligualconeconrho} to the horizontal composition $id_{{{a}}^*} \circ \theta$ of natural transformations, whose components are \mbox{$f^* FX  \xr{{{a}}^*(\theta_X)} f^* F'X$, is $FX \times F'X \mr{\lambda_X} H \mr{{{a}}^*} L$.} 
\end{corollary}

\begin{proof} Each $h^* FX \mr{\theta_X} h^* F'X$ corresponds to a relation $h^* FX \times h^* F'X \mr{\varphi_X} \Omega_{\cc{H}}$, which corresponds to $FX \times F'X \mr{\lambda_X} H$ via the adjunction $h^* \dashv h_*$. Denote by \mbox{$R_X \hookrightarrow h^* FX \times h^* F'X$} the subobject corresponding to $\varphi_X$.
 
 The subobject corresponding to  $f^* FX \xr{{{a}}^*(\theta_X)} f^* F'X$, is ${{a}}^*R_X \hookrightarrow f^* FX \times f^* F'X$, whose characteristic function (use remark \ref{caractsubobjeto}) is the relation $f^* FX \times f^* F'X \xr{{{a}}^*(\varphi_X)} {{a}}^*\Omega_{\cc{H}} \mr{\phi_1} \Omega_{\cc{F}}$. 
 Proposition \ref{paraelcoro} (with $X = FX$, $Y=F'X$) finishes the proof.
\end{proof}
\end{sinnadastandard}

\begin{sinnadastandard} {\bf Cones over a site.} 
Consider a topos $\cc{E}$ over $\cc{S}$, and a small site of definition $\cc{C}$ for $\cc{E}$. We will show that 
$\lozenge$-cones defined over $\cc{C}$ can be uniquely extended to $\lozenge$-cones defined over $\cc{E}$. This will provide existence theorems for constructions defined by universal properties quantified over large (external) sets, see proposition \ref{EndhasX}.

Let $\cc{C} \mr{F} \cc{S}$ be (the inverse image of) a point of the site, and $\cc{C}^{op} \mr{X} \cc{S}$ be a sheaf, $X \in \cc{E}$. Let 
$\Gamma_F \mr{} \cc{C}$ be the (small) diagram (discrete fibration) of $F$, recall that it is a cofiltered category whose objects are pairs $(c,C)$ with $c \in FC$, and whose arrows $(c,C) \mr{f} (d,D)$ are arrows $C \mr{f} D$ that satisfy $F(f)(c)=d$. Abuse notation and denote also by $F$, 
$\cc{E} \mr{F} \cc{S}$, the inverse image of the corresponding morphism of topoi. Recall the formulae:
\BE \label{yonedaII}
FX = X \otimes_\cc{C} F = \int^{C} XC \times FC 
\;\; \cong \;\; \colim{(c,C) \in \Gamma_F}{XC} \;\; \ml{\rho} \;\; \coprod_{C \in \cc{C}} XC \times FC
\EE
By Yoneda we have $\cc{E}(C, X) \mr{\cong} XC$, and under this identification we have, 

$$\text{for} \;\; C \mr{f} X \;\; \text{and} \;\;
c \in FC, \;\; F(f)(c) = \rho(f, c) \in FX,$$

\vspace{-5ex}

\BE \label{masvale}
\EE 

\vspace{-5ex}

$$\text{for} \;\; E \mr{h} C \text{ in } \cc{C}, \;\; X(h)(f) = fh.$$

\begin{remark} \label{rhoepi}
 Let $a \in FX$. Since $\rho$ is an epimorphism, there exist $C, \; f \in XC$ and $c \in FC$ such that $F(f)(c)=a$.
\end{remark}

\begin{remark} \label{FX}
Let $C, D \in \cc{C}$, $f \in XC$, $c \in FC$,  
$g \in XD$, $d \in FD$, be such that \mbox{$F(f)(c) = F(g)(d)$,} i.e. $\rho(f, c) = \rho(g, d)$. Since the category 
$\Gamma_F$ is cofiltered, by construction of filtered colimits there exist $E,\; e \in FE$ and $E \mr{h} C$,
 $E \mr{\ell} D$ such that $F(h)(e) = c$, 
 $F(\ell)(e) = d$ and $X(h)(f) = X(\ell)(g)$, i.e.
 $f h = g \ell$. \qed
\end{remark}

The following is the key result for the existence theorems in section \ref{sec:Contexts}.

\begin{proposition} \label{extension} Consider a small site of definition $\cc{C}$ of the topos $\cc{E}$. 
Then suitable cones defined over $\cc{C}$ can be extended to $\cc{E}$, more precisely:

1) Let $TC \times T'C \mr{\lambda_C} H$ be a $\lozenge_1$-cone (resp. 
$\lozenge_2$-cone, resp. $\lozenge$-cone) defined over $\cc{C}$. Then, $H$ can be uniquely furnished with $\ell$-relations $\lambda_X$ for all objects $X \in \cc{E}$ in such a way to determine a $\lozenge_1$-cone (resp. $\lozenge_2$-cone, resp. $\lozenge$-cone) over $\cc{E}$ extending $\lambda$. 
 
2) If $H$ is a locale and $\lambda_C$ ($C \in \cc{C}$) is a $\lozenge_1$-cone of $\ell$-functions (resp. $\lozenge_2$-cone of $\ell$-opfunctions, resp. $\lozenge$-cone of $\ell$-bijections), so is $\lambda_X$ ($X \in \cc{E}$).
\end{proposition}
\begin{proof} 

1) Recall that $T = F$ on $\cc{C}$. 
Let $X \in \cc{E}$, then $TX = FX$, $T'X = F'X$, and let   \mbox{$(a,\,b) \in TX \times T'X$.}  By \eqref{yonedaII}, \eqref{masvale} and remark \ref{rhoepi} we can take $C \mr{f} X$ and $c \in TC$ such that $a = T(f)(c) = F(f)(c)$ (see \ref{extensionarel})
If $\lambda_X$ were defined so that the 
$\lozenge_1(f)$ diagram commutes, the equation 
$$
(1) \hspace{2ex} \lambda_X( a,\,b)  = \bigvee_{y \in T'C} \igu[T'(f)(y)]{b} \cdot \lambda_C( c,\,y)
$$ 
should hold (see \eqref{ecuacionesdiagramas}). We define $\lambda_X$ by this equation. This definition is independent of the choice of $c, \, C,$ and $f$. In fact, let $D \mr{g} X$ and $d \in TD$ be such that 
$a = T(g)(d)$. By remark \ref{FX} we can take  
$(e,\, E)$ in the diagram of $T$ (or $F$), $E \mr{h} C$, $E \mr{\ell} D$ such that 
$T(h)(e) = c$, $T(\ell)(e) = d$ and $fh=g\ell$. Then we compute

\vspace{1ex}

\noindent $ \displaystyle
 \bigvee_{y \in T'C} \igu[T'(f)(y)]{b} \cdot \lambda_C( c,\,y) \;\; \stackrel{\lozenge_1(h)}{=} \;\; \bigvee_{y \in T'C} \;\;\bigvee_{w \in T'E} \igu[T'(f)(y)]{b} \cdot \igu[T'(h)(w)]{y} \cdot \lambda_E( e,\,w) = $

 
 \hfill $ \displaystyle
= \bigvee_{w \in T'E} \igu[T'(fh)(w)]{b} \cdot \lambda_E( e,\,w).
$

\vspace{.3cm}

From this and the corresponding computation with $d, \, D,$ and $\ell$, it follows:
$$
\bigvee_{y \in T'C} \igu[T'(f)(y)]{b} \cdot \lambda_C( c,\,y) = \bigvee_{y \in T'D} \igu[T'(g)(y)]{b} \cdot \lambda_{D}( d,\,y).
$$
Given $X \mr{g} Y$ in $\cc{E}$, we check that the  
$\lozenge_1(g)$ diagram commutes: Let $(a,\,b) \in TX \times T'Y$, take $C \mr{f} X$, $c \in TC$ such that $a = T(f)(c)$, and let \mbox{$d = T(g)(a) = T(gf)(c)$.} Then 

$$ \lambda_Y ( d,b ) = \bigvee_{z \in T'C} \igu[T'(gf)(z)]{b} \cdot \lambda_C ( c,z ) = \bigvee_{z \in T'C} \bigvee_{x \in X} \igu[T'(f)(z)]{x} \cdot \igu[T'(g)(x)]{b} \cdot \lambda_C ( c,z ) = $$

\begin{flushright}
$\displaystyle = \bigvee_{x \in T'X} \igu[T'(g)(x)]{b} \cdot \bigvee_{z \in T'Z} \igu[T'(f)(z)]{x} \cdot \lambda_C ( c,z ) = $
$ \displaystyle \bigvee_{x \in T'X} \igu[T'(g)(x)]{b} \cdot \lambda_X ( a,x ).$
\end{flushright}

Clearly a symmetric argument can be used if we assume at the start that the $\lozenge_2$ diagram commutes. In this case, $\lambda_X$ would be defined by taking  $C \mr{f} X$ and $c \in T'C$ such that \mbox{$b = T'(f)(c)$} and computing:
$$
(2) \hspace{2ex}  \lambda_X( a,\,b)  = \bigvee_{y \in TC} \igu[T(f)(y)]{a} \cdot \lambda_{C}( y,\,c).
$$  

If the $TC \times T'C \mr{\lambda_C} H$ form a $\lozenge$-cone (i.e. a $\lozenge_1$-cone and a 
$\lozenge_2$-cone), definitions (1) and (2) coincide. In fact, since they are independent of the chosen $c$, it follows they are both equal to: 

\vspace{1ex}

$ \displaystyle
\bigvee_{C \mr{f} X} \bigvee_{c \in TC} \bigvee_{y \in T'C} \igu[T(f)(c)]{a} \cdot \igu[T'(f)(y)]{b} \cdot \lambda_C( c,\,y) 
\hspace{2ex} = \hspace{2ex} $

\vspace{-1.5ex}

\hfill $ \displaystyle
\bigvee_{C \mr{f} X} \bigvee_{c \in T'C} \bigvee_{y \in TC} \igu[T'(f)(c)]{b} \cdot \igu[T(f)(y)]{a} \cdot \lambda_C( y,\,c).
$

\vspace{1ex}

2) It suffices to prove that if $\lambda_C$ ($C \in \cc{C}$) is a $\lozenge_1$-cone of $\ell$-functions, so is $\lambda_X$ ($X \in \cc{X}$). Let $X \in \cc{E}$, $a \in TX$, $b_1,b_2 \in T'X$. Take as in item 1) $C \mr{f} X$ and $c \in TC$ such that $a = T(f)(c)$.

\vspace{.5cm}

\noindent$ed) \displaystyle \bigvee_{b \in T'X} \lambda_X( a,\,b) \;\; = \;\; \bigvee_{b \in T'X} \bigvee_{y \in T'C} \igu[T'(f)(y)]{b} \cdot \lambda_C( c,\,y) \;\; = \;\; \bigvee_{y \in T'C} \lambda_C( c,\,y) \;\; \stackrel{ed)}{=} \;\; 1$.

\vspace{.5cm}

\noindent $uv)$ $\displaystyle \lambda_X(a,b_1) \wedge \lambda_X(a,b_2) = \displaystyle  \bigvee_{y_1,y_2 \in T'C} \igu[T'(f)(y_1)]{b_1} \cdot \igu[T'(f)(y_2)]{b_2} \cdot \lambda_C(c,y_1) \wedge \lambda_C(c,y_2) \stackrel{uv)}{\leq} $

\hfill $\displaystyle  \bigvee_{y_1,y_2 \in T'C} \igu[T'(f)(y_1)]{b_1} \cdot \igu[T'(f)(y_2)]{b_2} \cdot \igu[y_1]{y_2} \;\; {\leq} $

\hfill $\displaystyle \! \! \! \bigvee_{y_1,y_2 \in T'C} \! \! \! \igu[T'(f)(y_1)]{b_1} \! \cdot \!  \igu[T'(f)(y_2)]{b_2} \! \cdot \! \igu[T'(f)(y_1)]{T'(f)(y_2)} \;\; {\leq} \;\; \igu[b_1]{b_2}$.

\end{proof}

\begin{assumption}
For the rest of this section we consider a small site $\cc{C}$ (with binary \mbox{products and $1$)} of the topos $\cc{E}$, and cones defined over $\cc{C}$.
\end{assumption}
\end{sinnadastandard}

\begin{sinnadastandard} {\bf Compatible cones.} 
We now introduce the notion of compatible cone. 
Any compatible \mbox{$\lozenge$-cone} which covers a commutative algebra $H$ will force $H$ to be a locale, and such a cone will necessarily be a cone of $\ell$-bijections (and vice versa):

\begin{definition} \label{comp}
Let $H$ be a commutative algebra in $s \ell$, with multiplication $*$ and unit $u$. Recall that the product is a map $H \otimes H \mr{*} H$, and that $u \in H$ induces a linear map $\Omega \mr{u} H$.

Let $TC \times T'C \mr{\lambda_C} H$ be a cone. 
We say that $\lambda$ is \emph{compatible} if the following equations hold: 

\flushleft $[C1] \quad \forall \ a \in TC, a' \in T'C, b \in TD, b' \in T'D, \quad \lambda_C ( a,\,a' ) *  \lambda_D ( b,\,b' ) = \lambda_{C \times D}( (a,\,b),\,(a',\,b') )$

\flushleft $[C2] \quad \lambda_1 = u.$
\end{definition}

Given a compatible cone, consider the diagonal $C \mr{\Delta} C \times C$, the arrow $C \mr{\pi} 1$, and the following $\lozenge_1$ diagrams (see \ref{diagramadiamante12}):
$$
\xymatrix@R=4ex@C=2ex
        { 
         & TC \!\times\! T'C \ar[rd]^{\lambda_C} 
         & & & TC \!\times\! T'C \ar[rd]^{\lambda_C} 
        \\
	     TC \!\times\! (T'C \!\times\! T'C) 
	                      \ar[ru]^{TC \times \Delta^{op}} 
	                      \ar[rd]_{\Delta \times (T'C \times T'C) \ \ } 
	     & \equiv 
	     & \hspace{0ex} H, 
	     & TC \!\times\! 1 \ar[ru]^{TC \ \times \ \pi^{op}} 
	                  \ar[rd]_{\pi \ \times \ 1} 
	     & \!\! \equiv 
	     & \hspace{0ex} H.
	    \\
	     & (TC \!\times\! TC) \!\times\! (T'C \!\times\! T'C) 
	                          \ar[ru]_{\lambda_{C \times C}} 
	     & & & 1 \!\times\! 1 \ar[ru]_{\lambda_1} 
	    }
$$

expressing the equations: for each $a \in TC, b_1,b_2 \in T'C,$

\vspace{1ex}

\noindent $\displaystyle \lozenge_1(\triangle) \! : \quad   \lambda_{C \times C}( (a,a),(b_1,b_2) ) \, \stackrel{}{=} \bigvee_{x \in T'C} \igu[(x,x)]{(b_1,b_2)} \cdot \lambda_C ( a,\,x)$,

\vspace{1ex} 

\noindent $\displaystyle \lozenge_1(\pi) \! : \quad   \lambda_1 = \bigvee_{x \in T'C} \lambda_C( a,\,x)  $.

\begin{lemma} \label{lemaparaconocomp}
 Let $TC \times T'C \mr{\lambda} H$ be a compatible $\lozenge_1$-cone (or $\lozenge_2$-cone, or $\lozenge$-cone) with vertex a commutative algebra $H$. Then, for each $a \in TC, b_1,b_2 \in T'C,$
\begin{enumerate}
 \item $\displaystyle \lambda_C ( a,b_1 ) *  \lambda_C ( a,b_2 ) = \igu[b_1]{b_2} \cdot \lambda_C(a,b_1)$.
 \item $\displaystyle u = \bigvee_{x \in T'C} \lambda_C( a,\,x)$.
\end{enumerate}
\end{lemma}
\begin{proof}
 2. is immediate from [C2] and $\lozenge_1(\pi)$ above. To prove 1. we compute
 
  \noindent $\displaystyle \lambda_C ( a,b_1 ) *  \lambda_C ( a,b_2 ) \stackrel{[C1]}{=} \lambda_{C \times C}( (a,a),(b_1,b_2) ) \stackrel{\lozenge_1(\triangle)}{=} \bigvee_{x \in T'C} \igu[x]{b_1} \cdot \igu[x]{b_2} \cdot \lambda_C ( a,\,x) \stackrel{\ref{ecuacionenL}}{=} $
  
  
  \flushright $\displaystyle = \bigvee_{x \in T'C} \igu[x]{b_1} \cdot \igu[b_1]{b_2} \cdot \lambda_C ( a,\,b_1) = \displaystyle \igu[b_1]{b_2} \cdot \lambda_C ( a,\,b_1) $.
  
\end{proof}

%
%
%
%
%
%

\begin{proposition} \label{compislocale}
Let $\lambda$ be a compatible $\lozenge$-cone with vertex a commutative algebra $(H,*,u)$ such that the elements of the form $\lambda_C( a,\,b )$, $a \in TC, b \in T'C$ are sup-lattice generators of $H$. Then $H$ is a locale and $*=\wedge$. 
\end{proposition}
\begin{proof}
By the results of \cite[III.1, p.21, Proposition 1]{JT}, it suffices to show that for all $w \in H$, 
$$(L1) \quad w * w = w \hspace{2cm} and \hspace{2cm} (L2) \quad w \leq u.$$ 

It immediately follows from equations 1. and 2. in the lemma above that (L1) and (L2) hold for $w = \lambda_C( a,\,b )$. Then clearly (L2) holds for any supremum of elements of this form. To show (L1) we do as follows:

$w * w \leq w * u = w$ always holds, and to show $\geq$, if $w = \displaystyle \bigvee_{i \in I} w_i$ satisfying $w_i * w_i = w_i$ we compute:

$$ \displaystyle \bigvee_{i \in I} w_i * \bigvee_{i \in I} w_i \geq \bigvee_{i \in I} w_i * w_i \stackrel{(L1)}{=} \bigvee_{i \in I} w_i.$$
\end{proof}

\begin{proposition} \label{Diamondisbijection}  
Conider a cone $\lambda$ with vertex a locale $H$. 
\begin{enumerate}
 \item If $\lambda$ is a $\lozenge_1$-cone, then $\lambda$ is compatible if and only if it is a $\lozenge_1$-cone of $\ell$-functions.
 \item If $\lambda$ is a $\lozenge_2$-cone, then $\lambda$ is compatible if and only if it is a $\lozenge_2$-cone of $\ell$-op-functions.
 \item If $\lambda$ is a $\lozenge$-cone, then $\lambda$ is compatible if and only if it is a $\lozenge$-cone of $\ell$-bijections.
\end{enumerate}


\end{proposition}
\begin{proof} We prove 1, 2 follows by symmetry and adding 1 and 2 we obtain 3.

($\Rightarrow$): Since $\wedge=*$ and $1=u$ in $H$, equations 1. and 2. in lemma \ref{lemaparaconocomp} become the axioms ed) and uv) for $\lambda_X$. 

($\Leftarrow$) $u = 1$ in $H$, so equation [C2] in definition \ref{comp} is axiom $ed)$ for $\lambda_1$. To prove equation [C1] we consider the projections $C \times D \mr{\pi_1} C$, $C \times D \mr{\pi_2} D$. The $\lozenge_1(\pi_1)$  and $\lozenge_1(\pi_2)$ diagrams express the equations:

\vspace{1ex}

\noindent For each $a \in TC, b \in TD, a' \in T'C, \:\:
 \lambda_C( a,a') = \displaystyle \bigvee_{y \in T'D} \lambda_{C \times D} ( (a,b),(a',y) ),$
 
\noindent For each $a \in TC, b \in TD, b' \in T'D, \:\:
 \lambda_D ( b,b') = \displaystyle \bigvee_{x \in T'C} \lambda_{C \times D} ( (a,b),(x,b') ).
$

\vspace{1ex}

Taking the infimum of these two equations we obtain for each $a \in TC, b \in TD,   a' \in T'C, b' \in T'D$:

\vspace{1ex}

\noindent
$
\lambda_C ( a,a' ) \wedge \lambda_D ( b,b' )  \;=\; \displaystyle \bigvee_{x \in T'C} \bigvee_{y \in T'D} \lambda_{C \times D} ( (a,b),(a',y) )  \wedge \lambda_{C \times D} ( (a,b),(x,b') ) =
$

\flushright $\stackrel{uv)_{\lambda_{C \times D}}}{\;=\;}  \displaystyle \bigvee_{x \in T'C} \bigvee_{y \in T'D} \igu[(a',y)]{(x,b')} \cdot \lambda_{C \times D} ( (a,b),(a',y) ) \stackrel{\ref{ecuacionenL}}{=} \lambda_{C \times D} ( (a,b),(a',b') )$

\end{proof}

\noindent Also, sup-lattice morphisms of cones with compatible domain are automatically locale   \mbox{morphisms:}

\begin{proposition} 
Let $\lambda$ be a  compatible cone with vertex a locale $H$ such that the elements of the form $\lambda_C( a,a' )$, $a \in TC, a' \in T'C$ are sup-lattice generators of $H$. Let $\lambda$ be another compatible cone with vertex a locale $H'$. Then, any sup-lattice morphism $H \mr{\sigma} H'$ satisfying $ \sigma \lambda_C = \lambda_C$ is a locale morphism.
\end{proposition}

\begin{proof}
Equation [C2] in defintion \ref{comp} implies immediately that $\sigma u = u'$ (i.e. $\sigma$ preserves $1$).

Equation [C1] implies immediately that the infima $\wedge$ between two sup-lattice generators   \mbox{$\lambda_C( a,a' )$} and $\lambda_D( b,b' )$ is preserved by $\sigma$, which suffices to show that $\sigma$ preserves $\wedge$ between two arbitrary elements since $\sigma$ is a sup-lattice morphism.
\end{proof}

Combining the previous proposition with proposition \ref{Diamondisbijection} we obtain

\begin{corollary} \label{supisloc}
Let $\lambda$ be a $\lozenge$-cone of $\ell$-bijections with vertex a locale $H$ such that the elements of the form $\lambda_C( a,\,b )$, $a \in TC, b \in T'C$ are sup-lattice generators of $H$. Let $\lambda$ be another \mbox{$\lozenge$-cone} of \mbox{$\ell$-bijections} with vertex a locale $H'$. Then, any sup-lattice morphism $H \mr{\sigma} H'$ satisfying $ \sigma \lambda_C = \lambda_C$ is a locale morphism. \qed
\end{corollary}
\end{sinnadastandard}

\section{The equivalence $Cmd_0(\O(G)) = Rel(\beta^G)$} \label{sec:Cmd=Rel}

\begin{sinnadastandard} \label{sinnadadeequivalence0}
We fix throughout this section a localic groupoid $G$ (i.e. groupoid object in \mbox{$S \! p=Loc^{op}$),} with subjacent structure of localic category (i.e. category object in $S \! p$) given by (\cite[VIII.3 p.68]{JT}) 

\vspace{-2ex}

$$\xymatrix{G \stackrel[G_0]{\textcolor{white}{G_0}}{\times} G \ar[r]^>>>>>>{\circ} & G \ar@<1.3ex>[r]^{\partial_0} \ar@<-1.3ex>[r]_{\partial_1} & G_0 \ar[l]|{\ \! i \! \ }}$$ 

\vspace{-1ex}

\noindent We abuse notation by using the same letter G for the object of arrows of  G.
We denote by $L = \O(G)$, $B = \O(G_0)$ their corresponding locales of open parts, and think of them as (commutative) algebras in the monoidal category $s \ell$. The locale morphisms $\xymatrix{B \ar@<1ex>[r]^{s = \partial_0^{-1}} \ar@<-1ex>[r]_{t = \partial_1^{-1}} & L}$ furnish $L$ with a structure of $B$-bimodule. We establish, following \cite{De}, that $B$ acts on the left via $t$ and on the right via $s$. This is consistent with the pull-back $G \times_{G_0} G$ above which is thought of as the pairs \mbox{$\{(f,g) \in G \times G \ | \ \partial_0(f)=\partial_1(g)\}$} of composable arrows, in the sense that $\O(G \times_{G_0} G) = L \otimes_B L$ (the push-out corresponding to the pull-back is the tensor product of $B$-bimodules).

In this way, the unit $\xymatrix{G_0 \ar[r]^i & G}$ corresponds to a counit $L \mr{e} B$, and the multiplication (composition) $G \times_{G_0} G \mr{\circ} G$ corresponds to a comultiplication $L \mr{c} L \otimes_B L$. Therefore $L$ is a coalgebra in the category $B$-bimod, i.e. a \emph{cog\`ebro\"ide agissant sur B} (\cite[1.15]{De}). In other words, \emph{a localic category structure for $G$ is the same as a cog\`ebro\"ide structure for $L$.}

We define a \emph{localic Hopf algebroid} as the exact formal dual structure of a localic groupoid. The inverse $G \mr{(-)^{-1}} G$ of a localic groupoid corresponds to an \emph{antipode} $L \mr{a} L$. As was observed by Deligne in \cite[p.117]{De}, the structure of cog\`ebro\"ide is the subjacent structure of a Hopf algebroid which is used to define its representations (see definition \ref{defcmd}), exactly like the subjacent localic category structure of the groupoid is the subjacent structure required to define $G$-spaces as $S \! p$-valued functors, namely, actions of the category object on an internal family $X \mr{} G_0$ (see definition \ref{defdeaction}).
\end{sinnadastandard}

\begin{sinnadastandard}
{\bf The category $\beta^G$}  Groupoid objects $G$ in $S \! p$ act on spaces over $G_0$, $X \mr{} G_0$, as groupoids (or categories with object of objects $G_0$) act on families over $G_0$ in sets, defining an internal functor. We consider $G \times_{G_0} X$, the pull-back of spaces over $G_0$ constructed using $\partial_0$, as a space over $G_0$ using $\partial_1$:
\end{sinnadastandard}


\begin{\de} \label{defdeaction}
An \emph{action} of a localic groupoid $G$ in a space over $G_0$, $X \mr{} G_0$, is a morphism $G \times_{G_0} X \mr{\theta} X$ of spaces over $G_0$ such that the following diagrams commute.
$$A1: \vcenter{\xymatrix{G \stackrel[G_0]{\textcolor{white}{G_0}}{\times} G \stackrel[G_0]{\textcolor{white}{G_0}}{\times} X \ar[r]^>>>>>{\circ \times X} \ar[d]_{G \times \theta} & G \stackrel[G_0]{\textcolor{white}{G_0}}{\times} X \ar[d]^{\theta} \\
													G \stackrel[G_0]{\textcolor{white}{G_0}}{\times} X \ar[r]^{\theta} & X  } }
					A2: \vcenter{ \xymatrix{G \stackrel[G_0]{\textcolor{white}{G_0}}{\times} X \ar[r]^{\theta} & X \\
												G_0 \stackrel[G_0]{\textcolor{white}{G_0}}{\times} X \ar[u]^{i \times X} \ar[ur]_{\cong} }} $$

We denote the action as $G \stackrel{\theta}{\curvearrowright} X$, omitting sometimes the $\theta$. An action morphism between two actions $G \stackrel{\theta}{\curvearrowright} X$, $G \stackrel{\theta'}{\curvearrowright} X'$ (which corresponds to a natural transformation between the functors) is a morphism $f$ of spaces over $G_0$ such that the following diagram commute.
$$AM: \; \vcenter{ \xymatrix{G \stackrel[G_0]{\textcolor{white}{G_0}}{\times} X \ar[r]^{\theta} \ar[d]_{G \times f} & X \ar[d]^f \\
						G \stackrel[G_0]{\textcolor{white}{G_0}}{\times} X' \ar[r]^{\theta'} & X' }}$$
\end{\de}

\begin{remark}
The reader can easily check that these definitions are equivalent to the ones of \mbox{\cite[VIII.3, p.68]{JT}.} 
\end{remark}

\begin{remark} Recall from \cite[VI.3 p.51, Proposition 3]{JT} (see also proposition \ref{discretespace} and \ref{enumeratedeJT}, item 5), that the functor 
$$sh B \xr{(-)_{dis}} S \! p(shB) \mr{\gamma_*} B \hbox{-} Loc^{op}
$$ 
\vspace{-3ex} 
$$
\xymatrix@C=2.8pc{\quad Y \quad \ar@{|->}[rr] && (\overline{Y_d} \rightarrow \overline{B}),}
$$ where $Y_d = \gamma_*(\Omega^Y) = \gamma_*\cc{O}(Y_{dis})$ (recall definition \ref{defdeXd}), yields an equivalence of categories $sh B \mr{} Et_B$, where $Et_B$ is the category of etale spaces over $\overline{B}$, i.e. $X \mr{p} \overline{B}$ satisfying that $p$ and the diagonal $X \mr{\triangle} X \times_{\overline{B}} X$ are open (see \cite[V.5 p.41]{JT}). 
\end{remark}
 
\begin{\de}
An action $G \curvearrowright X$ is \emph{discrete} if $X \mr{} G_0$ is etale, i.e. in view of last remark if $X = \overline{Y_{d}}$ (or equivalently $\O(X) = Y_d$) with $Y \in shB$. We denote by $\beta^G$ the category of discrete actions of $G$.
\end{\de}


\begin{sinnadastandard} \label{rel}
Consider $s \ell_0(shB)$ the full subcategory of $s \ell(shB)$ with objects of the form $\Omega_B^Y$. Then we have the equivalence $: Rel(shB) \mr{(-)_*} s \ell_0(shB)$. Consider also the restriction of the equivalence \mbox{$s \ell(shB) \cong B$-$Mod$} to $s \ell_0(shB) \cong (B$-$Mod)_0$, where the latter is defined as the full subcategory of \mbox{$B$-$Mod$} consisting of the $B$-modules of the form $Y_d$. Combining both we obtain the equivalence \mbox{$Rel(shB) \cong (B$-$Mod)_0$,} mapping $Y \leftrightarrow Y_d$.
\end{sinnadastandard}


The objective of this section is to prove that this equivalence lifts to an equivalence 
\mbox{$\cc{R}el(\beta^G) \cong  Cmd_0(L)$} (for the definition of $Cmd_0(L)$ see \ref{defcmd}).

\begin{sinnadastandard}\label{equiv:objects} {\bf The equivalence at the level of objects.} 
Consider an etale space $X \mr{} G_0$, where \mbox{$\O(X) = Y_d$,} with $Y \in shB$. A (discrete) action $G \stackrel{\theta}{\curvearrowright} X$ 
corresponds to a $B$-locale morphism \mbox{$Y_d \mr{\rho} L \otimes_B Y_d$} satisfying C1, C2 in definition \ref{defcmd}. Therefore, to establish an equivalence between discrete actions $G \stackrel{\theta}{\curvearrowright} X$ and comodules $Y_d \mr{\rho} L \otimes_B Y_d$ we need to prove 
\end{sinnadastandard}

\begin{\prop} \label{locmorphgratis}
Every comodule structure $Y_d \mr{\rho} L \otimes_B Y_d$ is automatically a locale morphism (when $L$ is the cog\`ebro\"ide corresponding to the localic category subjacent to a localic groupoid).   
\end{\prop}

Next we prove this proposition (see \ref{esquema} below for a clarifying diagram). In order to do this, we will work in the category of $B \otimes B$-modules. Since $B$ is commutative, we have an isomorphism of categories $B$-bimod $\cong B \otimes B$-mod, but 
we consider the tensor product $\stackrel[B \otimes B]{}{\otimes}$ of $B \otimes B$-modules via the inclusion $B \otimes B$-mod $\hookrightarrow$ $B \otimes B$-bimod, not to be confused with the tensor product $\otimes_B$ as $B$-bimodules. Via this isomorphism, $L$ is a $B \otimes B$-module whose structure is given by $B \otimes B \mr{(t,s)} L$.

We first notice that $L \stackrel[B]{}{\otimes} Y_d \cong L \stackrel[B \otimes B]{}{\otimes} (B \otimes Y_d)$, and via extension of scalars (using the inclusion $B \mr{} B \otimes B$ in the first copy), $\rho$ corresponds to a morphism \mbox{$Y_d \otimes B \mr{\rho} L \stackrel[B \otimes B]{}{\otimes} (B \otimes Y_d)$} of \mbox{$B \otimes B$-modules.} 
From the equivalence of tensor categories recalled in section \ref{enumeratedeJT} items 5,6, with $P = B \otimes B$, $\rho$ corresponds to a morphism $\varphi$ in $s \ell(sh(B \otimes B))$, $\rho = \gamma_*(\varphi)$, and $\rho$ is a locale morphism if and only if $\varphi$ is so.

From the results of \ref{particular}, $Y_d \otimes B = {(\pi_1^* Y)}_d = \gamma_*(\Omega_{B \otimes B}^{\pi_1^*Y})$, and similarly \mbox{$B \otimes Y_d = \gamma_*(\Omega_{B \otimes B}^{\pi_2^*Y})$,} where $\Omega_{B \otimes B}$ is the subobject classifier of $sh(B \otimes B)$. Then  
$$L \stackrel[B \otimes B]{}{\otimes} (B \otimes Y_d) \;\; \stackrel{\ref{enumeratedeJT}}{=} \;\; \gamma_*(\widetilde{L} \ \stackrel{(1)}{\otimes} \ \Omega_{B \otimes B}^{\pi_2^*Y}) \;\; \stackrel{(2)}{=} \;\; \gamma_* (\widetilde{L}^{\pi_2^*Y}),$$ 

where $\widetilde{L}$ is as in \ref{enumeratedeJT} item 5, $\gamma_* \widetilde{L} = L$, the tensor product marked with $(1)$ is as sup-lattices in $sh(B \otimes B)$ and the equality marked with $(2)$ holds since $\widetilde{L} \otimes \Omega_{B \otimes B}^{\pi_2^*Y}$ and $\widetilde{L}^{\pi_2^*Y}$ are the free $\widetilde{L}$-module in $\pi_2^*Y$ (see proposition \ref{formulainternaparaG}).

Then $\varphi$ is $\Omega_{B \otimes B}^{\pi_1^*Y} \mr{\varphi} \widetilde{L}^{\pi_2^*Y}$, therefore by remark \ref{edyuvdafunctionmedioG} there is an $\ell$-relation    $\pi_1^*Y \times \pi_2^*Y \mr{\lambda} \widetilde{L}$ in the topos $sh(B \otimes B)$ such that $\varphi = \lambda_*$ and, to see that $\rho$ is a locale morphism, we can prove that $\lambda$ is an $\ell$-op-function.

\begin{sinnadastandard} \label{esquema}
 We schematize the previous arguing in the following correspondence
 
 \vspace{.2cm}
 
 \noindent \begin{tabular}{ccc}
    $Y_d \mr{\rho} L \stackrel[B]{}{\otimes} Y_d$ & $B$-module morphism & $B$-locale morphism \\ \noalign{\smallskip} \hline \noalign{\smallskip}
    $Y_d \otimes B \mr{\rho} L \stackrel[B \otimes B]{}{\otimes} (B \otimes Y_d)$ & $(B \otimes B)$-module morphism & $(B \otimes B)${-locale morphism} \\  \noalign{\smallskip}\hline \noalign{\smallskip}
    $\Omega_{B \otimes B}^{\pi_1^*Y} \mr{\varphi} \widetilde{L}^{\pi_2^*Y}$ & $s \ell${ morphism in} $sh(B \otimes B)$ & locale morphism \\ \noalign{\smallskip}\hline \noalign{\smallskip}
    $\pi_1^*Y \times \pi_2^*Y \mr{\lambda} \widetilde{L}$ & $\ell${-relation in} $sh(B \otimes B)$ & $\ell$-op-function 
   \end{tabular}

\end{sinnadastandard}


\begin{\prop}[{cf. \cite[5.9]{DSz}}] \label{bijectiongratis}
The $\ell$-relation $\pi_1^*Y \times \pi_2^*Y \mr{\lambda} \widetilde{L}$ corresponding to a comodule structure $Y_d \mr{\rho} L \otimes_B Y_d$, where $L$ is the cog\`ebro\"ide corresponding to the localic category subjacent to a localic groupoid, is an $\ell$-bijection.
\end{\prop}
\begin{proof} We will use the analysis of this particular kind of $\ell$-relations that we did in section \ref{particular}. We have seen that $\lambda$ corresponds to a $B$-bimodule morphism $Y_d \otimes Y_d \mr{\mu} L$. We have also seen, in proposition \ref{propaxiomparamodulos2}, which conditions in $\mu$ are equivalent to the $\ell$-bijection axioms.

Since any duality induces an internal-hom adjunction and $\Omega^Y$ is self-dual, $\mu$ corresponds to $\rho$ via the duality of modules described in $\ref{adjunciondeldual}$. Then by lemma \ref{lemadeunicidad}, the $B1$ and $B2$ subdiagrams in the following diagram are commutative. Also, the pentagon subdiagram $\pentagon$ is commutative by definition of the localic groupoid $G$, where $a$ is the antipode corresponding to the inverse of $G$.          

\begin{equation} \label{diagramapentagonal}
\xymatrix@C=3pc{
          \ar@{}[dr]|>>>>>>>{B2} & Y_d \otimes Y_d \ar[r]^(.35){Y_d \otimes \eta \otimes Y_d} \ar[d]^\mu \ar[dl]_\eps  & Y_d \otimes Y_d  \stackrel[B]{}{\otimes} Y_d \otimes Y_d \ar@{}[dl]|{B1} \ar[d]^{\mu \otimes_B \mu} \\
          B \ar@<.75ex>[dr]^t \ar@<-.75ex>[dr]_s  & L \ar@{}[dr]|{\pentagon} \ar[r]^c \ar[l]_e  & L \stackrel[B]{}{\otimes} L \ar@<-1ex>[d]_{a \otimes L}
                        \ar@<1ex>[d]^{L \otimes a} \\
          & L & L \stackrel[B \otimes B]{}{\otimes} L. \ar[l]_\wedge  }
\end{equation}

          
To prove axiom $ed)$, let $b_0 \in B$, $x \in Y(b_0)$. 
Chasing $\delta_x \otimes \delta_x$ in diagram \eqref{diagramapentagonal} all the way down to $L$ using the arrow $L \otimes a$ we obtain (recall our formulae for $\eta$, $\eps$ in proposition \ref{formulaeetaeps})
$\displaystyle \bigvee_{\stackrel{b \in B}{y \in Y(b)}} \mu  ( \delta_x \otimes \delta_y )  \wedge a \mu  ( \delta_y \otimes \delta_x )  = b_0$, which implies the inequality $\displaystyle \bigvee_{\stackrel{b \in B}{y \in Y(b)}} \mu  ( \delta_x \otimes \delta_y ) \geq b_0$, i.e. $\geq$ in $ed)$ in proposition \ref{propaxiomparamodulos2}, but the inequality $\leq$ always holds.

To prove axiom $uv)$, let $b_0, b_1, b_2 \in B$, $x \in Y(b_0)$, $y_1 \in Y(b_1)$, $y_2 \in Y(b_2)$. Chasing $\delta_{y_1} \otimes \delta_{y_2}$, but this time using the arrow $a \otimes L$, we obtain $$\displaystyle \bigvee_{\stackrel{c \in B}{w \in Y(c)}} a  \mu  ( \delta_{y_1} \otimes \delta_w )  \wedge \mu  ( \delta_w \otimes \delta_{y_2} )  = \llbracket y_1 \! = \! y_2 \rrbracket_{B} ,$$
then in particular $ \quad (1) \quad  a  \mu  ( \delta_{y_1} \otimes \delta_x)  \wedge \mu  ( \delta_x \otimes \delta_{y_2} )  \leq \llbracket y_1 \! = \! y_2 \rrbracket_{B} . $

\vspace{1ex}

To deduce $uv)$ from (1) we need to see that $a  \mu  ( \delta_{y_1} \otimes \delta_x) = \mu  ( \delta_x \otimes \delta_{y_1})$. Since $a^2 = id$, it is enough to prove $\leq$:

\vspace{1ex}
 
\noindent $a  \mu  ( \delta_{y_1} \otimes \delta_x) \stackrel{\ref{restringirlosdelta}}{=}  a  \mu  ( \delta_{y_1} \otimes b_0 \cdot \delta_x) = a  \mu  ( \delta_{y_1} \otimes \delta_x) \wedge b_0  \stackrel{ed)}{=} a  \mu  ( \delta_{y_1} \otimes \delta_x) \wedge \displaystyle\bigvee_{\stackrel{b \in B}{y \in Y(b)}}  \mu  ( \delta_x \otimes \delta_y )$

\noindent $ =  \displaystyle\bigvee_{\stackrel{b \in B}{y \in Y(b)}}  a  \mu  ( \delta_{y_1} \otimes \delta_x) \wedge \mu  ( \delta_x \otimes \delta_y ) \stackrel{(1)}{=} 
 \displaystyle\bigvee_{\stackrel{b \in B}{y \in Y(b)}}  a  \mu  ( \delta_{y_1} \otimes \delta_x) \wedge \mu  ( \delta_x \otimes \delta_y ) \wedge \llbracket y_1 \! = \! y \rrbracket_{B}
 \stackrel{\ref{ecuacionenOmegaX}}{=} $

\hfill $ =  a  \mu  ( \delta_{y_1} \otimes \delta_x) \wedge \mu  ( \delta_x \otimes \delta_{y_1} ) $. 

\vspace{1ex}
 
Axioms $su)$ and $in)$ follow symmetrically.
\end{proof}

We have finished the proof of proposition \ref{locmorphgratis}. For future reference, we record the results of this section:

\begin{\prop} \label{equivdeaccion}
Given a localic groupoid $G$ over $G_0$, with subjacent cog\`ebro\"ide $L$ sur $B$, and $Y \in shB$
, the following structures are in a bijective correspondence:

\vspace{1ex}

a. Discrete actions $G \stackrel{\theta}{\curvearrowright} \overline{Y_{d}}$. 
 
b. $\ell$-relations $\pi_1^*Y \times \pi_2^*Y \mr{\lambda} \widetilde{L}$ with a corresponding $B$-bimodule morphism   $Y_d \otimes Y_d \mr{\mu} L$ such that the following diagrams commute:
 $$B1: \vcenter{\xymatrix{Y_d \otimes Y_d \ar[r]^{\mu} \ar[d]_{Y_d \otimes \eta \otimes Y_d} & L \ar[d]^{c} \\
	    Y_d \otimes Y_d \stackrel[B]{}{\otimes} Y_d \otimes Y_d \ar[r]^>>>>{\mu \otimes_B \mu} & L \stackrel[B]{}{\otimes} L}}
\hspace{5ex}
B2: \vcenter{\xymatrix{Y_d \otimes Y_d \ar[r]^{\mu} \ar[dr]_{\eps} & L \ar[d]^{e} \\ & B}}$$

c. Comodule structures $Y_d \mr{\rho} L \otimes_B Y_d$. \qed

\end{\prop}


\begin{remark} \label{coinciden1}
 In the case where $G$ is a localic group, actions $Aut(X) \mr{} G$ defined as in \mbox{\cite[7.2]{D1},} also correspond to the previous structures (see \cite[5.9]{DSz}).
\end{remark}

\begin{notation}
 We fix until the end of this paper the following notation: we use the symbols $\theta$, $\rho$, $\lambda$, $\mu$ only for the arrows in the correspondence above, adding a $(-)'$ if neccessary.
\end{notation}

\begin{remark} \label{rightcomodules}
 In \cite{De}, the structure considered is the opposite of \ref{defcmd}, i.e. right \mbox{comodules} \mbox{$Y_d \mr{\rho} Y_d \otimes_B L$} (see note \ref{notasobresimetria}). By considering the inverse image $\lambda^*$ we obtain that this structure is also equivalent to the other three, and so are the right discrete actions $\overline{Y_d} \curvearrowleft G$. This situation is analogous to the correspondence between right and left actions of a group given by $x \cdot g = g^{-1} \cdot x$.
\end{remark}


\begin{sinnadastandard}\label{equiv:arrows}{ \bf The equivalence at the level of arrows.}  
We start this section with some results in order to better understand the category $Rel(\beta^G)$. We begin with a proposition that relates action morphisms with $\lozenge_2$-cones as in section \ref{sec:cones}. 
\end{sinnadastandard}

\begin{\prop} \label{accionsiidiam2}
Given two discrete actions $G \times_{G_0} \overline{Y_{d}} \mr{\theta} \overline{Y_{d}}$, $G \times_{G_0} \overline{Y'_{d}} \mr{\theta'} \overline{Y'_{d}}$, a space morphism $\overline{Y_{d}} \mr{f} \overline{Y'_{d}}$ is an action morphism if and only if the corresponding arrow $Y \mr{g} Y'$ in $shB$ satisfies 

$$\vcenter{\xymatrix{\lozenge_2(g): \ \ }}
\vcenter{\xymatrix@C=-1.5ex
         {   Y' \dcellb{g^{op}} & & Y \did 
          \\              
             Y & & Y
          \\
            & \ G \ \cl{\lambda} 
         }}
\vcenter{\xymatrix{ \;=\; }}
\vcenter{\xymatrix@C=-1.5ex{   Y' \did & & Y \dcell{g}
          \\              
             Y' & & Y'
          \\
           & \ G \ \cl{\lambda'}         }}
\vcenter{\xymatrix{\hbox{\hspace{5ex}i.e.}: \quad    }}           
\vcenter{\xymatrix@C=-1.5ex
         {  Y'_d \dcellb{g^{\wedge}} & &  Y_d \did
          \\              
             Y_d & & Y_d
          \\
            & \ L \ \cl{\mu} 
         }}
\vcenter{\xymatrix{ \;=\; }}
\vcenter{\xymatrix@C=-1.5ex{  Y'_d \did & & Y_d \dcell{g} 
          \\              
             Y'_d & & Y'_d
          \\
           & \ L \ \cl{\mu'}         }}$$

\end{\prop}

\begin{proof}
$f^{-1}$ (the formal dual of $f$) is the $B$-locale morphism $Y_d \mr{g^{\wedge}} Y_d$, which is computed with the self-duality of $Y_d$ (see \ref{discretespace} and \ref{autodualparaG}), and the correspondence between $\theta$ and $\mu$ in proposition \ref{equivdeaccion} is also given by this duality, i.e.

$$\vcenter{\xymatrix{f^{-1} = g^{\wedge}: \quad   }}
\vcenter{\xymatrix@C=-0.3pc@R=1pc{  Y'_d \did &&& \op{\eta} \\
					Y'_d \did && Y_d \dcell{g} \Brr & \,\,\, & Y_d \did \\
					Y'_d && Y'_d \Brr && Y_d \did \\
					& B \cl{\eps} & \Brr && Y_d }}
\vcenter{\xymatrix{\quad \quad \theta^{-1} = \rho: \quad }}
\vcenter{\xymatrix@C=-0.3pc@R=1pc{Y'_d \did &&& \op{\eta} \\
				  Y'_d && Y'_d \Brr & \,\,\, & Y'_d \did \\
				  & L \cl{\mu} & \Brr && Y'_d }}$$

Then the commutativity of the diagram AM in definition \ref{defdeaction}, expressing that $f$ is an action morphism, is equivalent when passing to the formal dual to the equality of the left and right terms of the equation (and therefore to the equality marked with an (*))
$$\vcenter{\xymatrix@C=-0.3pc@R=1pc{Y'_d \did &&& \op{\eta} \\
	 Y'_d && Y'_d \Brr & \,\,\, & Y'_d \did &&& \op{\eta} \\
	 & L \cl{\mu'} \did & \Brr && Y'_d \did && Y_d \dcell{g} \Brr & \,\,\, & Y_d \did \\
	 & L \did & \Brr && Y'_d && Y'_d \Brr && Y_d \did \\
	 & L & \Brr &&& B \cl{\eps} & \Brr && Y_d }}
\vcenter{\xymatrix{\stackrel{\triangle}{=} }}
\vcenter{\xymatrix@C=-0.3pc@R=1pc{Y'_d \did &&& \op{\eta} \\
	 Y'_d \did && Y_d \dcell{g} \Brr & \,\,\, & Y_d \did \\
	 Y'_d && Y'_d \Brr && Y_d \did \\
	 & L \cl{\mu'} & \Brr && Y_d}}
\vcenter{\xymatrix{\stackrel{(*)}{= }}}
\vcenter{\xymatrix@C=-0.3pc@R=1pc{Y'_d \did &&& \op{\eta} \\
	 Y'_d \did && Y_d \dcell{g} \Brr & \,\,\, & Y_d \did &&& \op{\eta} \\
	 Y'_d && Y'_d \Brr && Y_d && Y_d \Brr & \,\,\, & Y_d \did \\
	 & B \cl{\eps} & \Brr &&&	L \cl{\mu} & \Brr && Y_d}} $$
But the equality (*) is $\lozenge_2(g)$ composed with $\eta$, to recover $\lozenge_2(g)$ compose with $\eps$.
\end{proof}

\begin{corollary}
Using propositions \ref{accionsiidiam2} and \ref{equivdeaccion}, we can think of the category $\beta^G$ of discrete actions of $G$ in a purely algebraic way (without considering spaces over ${G_0}$) as follows: an action is a $B$-bimodule morphism $Y_d \otimes Y_d \mr{\mu} L$ satisfying B1, B2, and an action morphism is an arrow $Y \mr{g} Y'$ in $shB$ such that $\lozenge_2(g)$ holds.
\end{corollary}

\begin{remark} \label{coinciden2}
Since $\mu$ is an $\ell$-bijection, $\lozenge_2(g)$ holds if and only if $\rhd(g)$ does, therefore definition \ref{defdeaction} coincides with \cite[definition 7.4]{D1} for the case of a localic group.
\end{remark}

\begin{remark} \label{monoenshB}
Since the forgetful functor $\beta^G \mr{F} shB$, $G \curvearrowright \overline{Y_{d}} \mapsto Y$, is left exact, a monomorphism of discrete $G$-actions $Z \mr{g} Y$ is also a monomorphism in $shB$.
\end{remark}

\begin{lemma} \label{monocancela}
Given two actions $Y_d \otimes Y_d \mr{\mu} L$ and $Z_d \otimes Z_d \mr{\mu'} L$ and a monomorphism $Z \mr{g} Y$ of actions, for each $\delta_z$, $\delta_w$ generators of $Z_d$, $\mu'(\delta_z \otimes \delta_w) = \mu(\delta_{g(z)} \otimes \delta_{g(w)})$.
\end{lemma}

\begin{proof} 
$\mu(\delta_{g(z)} \otimes \delta_{g(w)}) \stackrel{\ref{accionsiidiam2}}{=} \displaystyle \bigvee_{\stackrel{b \in B}{x \in Y(b)}} \igu[g(x)]{g(z)}_B \cdot \mu' (\delta_x \otimes \delta_w) \stackrel{\ref{monoenshB}}{=} $

\hfill $= \displaystyle \bigvee_{\stackrel{b \in B}{x \in Y(b)}} \igu[x]{z}_B \cdot \mu' (\delta_x \otimes \delta_w) \stackrel{\ref{ecuacionenOmegaX}}{=} \mu' (\delta_z \otimes \delta_w). $
\end{proof}

\begin{lemma}[{cf. \cite[5.8]{DSz}}] \label{unicaaccionposible}
Given an action $Y_d \otimes Y_d \mr{\mu} L$ and a monomorphism $Z \mr{f} Y$, if the restriction of the action to $Z$ is an $\ell$-bijection, then it is an action. This is the only possible action on $Z$ that makes $f$ a morphism of $G$-actions.
\end{lemma}

\begin{proof}
Unicity is clear from the previous lemma. We have to check B1 and B2 in proposition \ref{equivdeaccion} for $Z_d \otimes Z_d \mr{\mu} L$. The only one that requires some care is B1. By hypothesis we have for $b_0,b'_0 \in B$, $x \in Y(b_0), w \in Y(b'_0)$, $$c \mu (\delta_x \otimes \delta_w) = \displaystyle \bigvee_{\stackrel{b \in B}{y \in Y(b)}} \mu (\delta_x \otimes \delta_y) \stackrel[B]{}{\otimes} \mu (\delta_y \otimes \delta_w)$$

\begin{center}
(we specify in the notation if the tensor product is over $B$).
\end{center}


We have to see that when $x \in Z(b_0), w \in Z(b'_0)$, this equation still holds when restricting the supremum to $Z$. In fact, in this case we have 

$$\displaystyle \bigvee_{\stackrel{b \in B}{y \in Y(b)}} \mu (\delta_x \otimes \delta_y) \stackrel[B]{}{\otimes} \mu (\delta_y \otimes \delta_w) \stackrel{\ref{restringirlosdelta}}{=}
\displaystyle \bigvee_{\stackrel{b \in B}{y \in Y(b)}} b_0 \cdot \mu (\delta_x \otimes \delta_y) \stackrel[B]{}{\otimes} \mu (\delta_y \otimes \delta_w) \cdot b'_0 \stackrel{ed), \ su)}{=}  $$

$$= \displaystyle \bigvee_{\stackrel{b \in B}{y \in Y(b)}} \bigvee_{\stackrel{b_1 \in B}{z_1 \in Z(b_1)}} \bigvee_{\stackrel{b_2 \in B}{z_2 \in Z(b_2)}} 
\mu (\delta_x \otimes \delta_{z_1}) \wedge \mu (\delta_x \otimes \delta_y) \stackrel[B]{}{\otimes} \mu (\delta_y \otimes \delta_w) \wedge \mu (\delta_{z_2} \otimes \delta_w)
\stackrel{uv), \ in), \ \ref{ecuacionenOmegaX}}{=} $$

\hfill $ \displaystyle = \bigvee_{\stackrel{b \in B}{z \in Z(b)}} \mu (\delta_x \otimes \delta_z) \stackrel[B]{}{\otimes} \mu (\delta_z \otimes \delta_w).$
\end{proof}

We are ready to prove the equivalence between the categories  
$\cc{R}el(\beta^G)$ and $Cmd_0(L)$.


\begin{theorem} \label{Comd=Rel}
For any localic groupoid $G$ as in \ref{sinnadadeequivalence0}, there is an equivalence of categories making the square commutative (both $U$ are forgetful functors):
$$
\xymatrix
        {
          \cc{R}el(\beta^G) \ar[rr]^{\cong} \ar[d]_{\cc{R}el(U)}
      & & Cmd_0(L) \ar[d]^{U} 
        \\
         \cc{R}el(shB) \ar[r]^{(-)_*} & s\ell_0(shB) \ar[r]^{\gamma_*} & (B\hbox{-}Mod)_0 .
        }       
$$ 

Note that the commutativity of the square means that the identification between relations   \mbox{$R \subset Y \times Y'$} in $shB$ and $B$-module morphisms \mbox{$Y_d \mr{R} Y'_d$} lifts to the upper part of the square.
\end{theorem} 
\begin{proof}
Since the equivalence $(B\hbox{-}Mod)_0 \cong \cc{R}el(shB)$ maps $Y_d \leftrightarrow Y$, proposition \ref{equivdeaccion} yields a bijection between the objects of $Cmd_0(L)$ and $Rel(\beta^G)$.

We have to show that this bijection respects the arrows of the categories (note that the composition of two relations corresponds to the composition of their direct images). 
Using the lemma \ref{unicaaccionposible}, it is enough to see that for $Y$, $Y'$ any two objects of $\beta^G$, and $R \subset Y \times Y'$ a relation in $shB$, the restriction $\theta$ of the product action $\lambda \boxtimes \lambda'$ to $R$ is a bijection if and only if the corresponding $B$-module map 
$R: Y_d \rightarrow Y'_d$ is a comodule morphism.

We claim that the diagram expressing that $R: Y_d \rightarrow Y_d'$ is a comodule morphism is equivalent to the diagram $\lozenge(R,R)$ in \ref{diagramadiamante12}. The proof follows then by proposition \ref{combinacion}. 

The comodule morphism diagram is the equality
\begin{equation} \label{commorph}
\vcenter{\xymatrix@C=-0.3pc@R=1.5pc {Y_d \dcell{R} &&& \op{\eta} \\
														Y'_d && Y'_d \Brr & \,\,\, & Y'_d \did \\
														& L \cl{\mu'} & \Brr && Y'_d}}
\vcenter{\xymatrix@C=-0.3pc@R=1.5pc{ \quad = \quad }}
\vcenter{\xymatrix@C=-0.3pc@R=1.5pc {Y_d \did &&& \op{\eta} \\
														Y_d && Y_d \Brr & \,\,\, & Y_d \dcell{R} \\
														& L \cl{\mu} & \Brr && Y'_d}}
\end{equation}
while the diagram $\lozenge$ is
\begin{equation} \label{grafdiam}
\vcenter{\xymatrix@C=-0.3pc@R=1.5pc         {  Y_d \did &&& \op{\eta} &&& Y'_d \did \\
																			Y_d \did && Y_d \did \Brr & \,\,\, & Y_d \dcell{R} && Y'_d \did \\
																			Y_d && Y_d \Brr && Y'_d && Y'_d \\
																			& L \cl{\mu} & \Brr &&& B \cl{\eps} }}
\vcenter{\xymatrix@C=-0.3pc@R=1.5pc         {\quad  = \quad } }
\vcenter{\xymatrix@C=-0.3pc@R=1.5pc         {  Y_d \dcell{R} && Y'_d \did \\
																			Y'_d && Y'_d \\
																			& L \cl{\mu'} }}
\end{equation}

{\it Proof of \eqref{commorph} $\implies$ \eqref{grafdiam}}:
$$
\vcenter{\xymatrix@C=-0.3pc@R=1.5pc         {  Y_d \did &&& \op{\eta} &&& Y'_d \did \\
																			Y_d \did && Y_d \did \Brr & \,\,\, & Y_d \dcell{R} && Y'_d \did \\
																			Y_d && Y_d \Brr && Y'_d && Y'_d \\
																			& L \cl{\mu} & \Brr &&& B \cl{\eps} }}
\vcenter{\xymatrix@C=-0.3pc@R=1.5pc         { \quad \stackrel{\eqref{commorph}}{=} \quad }}
\vcenter{\xymatrix@C=-0.3pc@R=1.5pc         {  Y_d \dcell{R} &&& \op{\eta} &&& Y'_d \did \\
																			Y'_d && Y'_d \Brr & \,\,\, & Y'_d && Y'_d \\
																			& L \cl{\mu'} & \Brr &&& B \cl{\eps} }}
\vcenter{\xymatrix@C=-0.3pc@R=1.5pc         { \quad \stackrel{(\triangle)}{=} \quad }}
\vcenter{\xymatrix@C=-0.3pc@R=1.5pc         {  Y_d \dcell{R} && Y'_d \did \\
																			Y'_d && Y'_d \\
																			& L \cl{\mu'} }}
$$
																			
{\it Proof of \eqref{grafdiam} $\implies$ \eqref{commorph}}:		
$$
\vcenter{\xymatrix@C=-0.3pc@R=1.5pc {Y_d \dcell{R} &&& \op{\eta} \\
														Y'_d && Y'_d \Brr & \,\,\, & Y'_d \did \\
														& L \cl{\mu'} & \Brr && Y'_d}}
\vcenter{\xymatrix@C=-0.2pc @R=1pc{ \quad \stackrel{\eqref{grafdiam}}{=} \quad }}
\vcenter{\xymatrix@C=-0.3pc@R=1.5pc {	Y_d \did &&& \op{\eta} &&&& \op{\eta} \\
																			Y_d \did && Y_d \did \Brr & \,\,\, & Y_d \dcell{R} && Y'_d \did \Brr & \,\,\, & Y'_d \did \\
																			Y_d && Y_d \Brr && Y'_d && Y'_d \Brr && Y'_d \did \\
																			& L \cl{\mu} & \Brr &&& B \cl{\eps} & \Brr && Y'_d}}
\vcenter{\xymatrix@C=-0.2pc @R=1pc{ \quad \stackrel{(\triangle)}{=} \quad }}
\vcenter{\xymatrix@C=-0.3pc@R=1.5pc {Y_d \did &&& \op{\eta} \\
														Y_d && Y_d \Brr & \,\,\, & Y_d \dcell{R} \\
														& L \cl{\mu} & \Brr && Y'_d}}
$$
\end{proof}

\section{The Galois and the Tannakian contexts} \label{sec:Contexts}

{\bf The Galois context associated to a topos.} Consider an arbitrary topos over $\Sat$, \mbox{$\Eat \mr{} \Sat$.} In \mbox{\cite[VII.3 p.59-61]{JT},} the following is proved. There is an open spatial cover of $\Eat$ (referred also as a ``surjective localic point''), i.e. an open surjection of topos $\Xat \mr{q} \Eat$ with $\Xat = shG_0$ for a \mbox{$G_0 \in sp$.} 
We use in this section the notation of $\cite{JT}$ for sheaves on a space, $shG_0 = sh (\cc{O}(G_0))$. 

As we mentioned in the introduction, Joyal and Tierney consider the localic point \mbox{$shG_0 \mr{q} \Eat$} as a (general) Galois context, as follows. In VIII.3 p.68-69, they show that the pseudo-kernel pair of $q$, $\xymatrix{\Xat \stackrel[\Eat]{{{}^{\ }}^{\ }}{\times} \Xat \ar@<1ex>[r]^>>>>>>{p_1} \ar@<-1ex>[r]_>>>>>>{p_2} & \Xat}$ satisfies that there is a localic groupoid 
\mbox{$G= \xymatrix{G \stackrel[G_0]{\textcolor{white}{G_0}}{\times} G \ar[r]^>>>>>>{\circ} & G \ar@<1.3ex>[r]^{\partial_0} \ar@<-1.3ex>[r]_{\partial_1} & G_0 \ar[l]|{\ \! i \! \ }}$}
such that 

\begin{equation} \label{igualdadJT} 
\xymatrix{\Xat \stackrel[\Eat]{{{}^{\ }}^{\ }}{\times} \Xat \ar@<1ex>[r]^>>>>>>{p_1} \ar@<-1ex>[r]_>>>>>>{p_2} & \Xat & = & shG \ar@<1ex>[r]^{\partial_0^*} \ar@<-1ex>[r]_{\partial_1^*} & shG_0}
\end{equation}

Joyal and Tierney use this to prove the equivalence $\Eat \cong \beta^G$ (see theorem \ref{fundamentalGT}) via descent techniques. They don't construct $G$, and they don't need to, since their proof relies solely on the fact that open surjections are effective descent morphisms (\cite[VIII.2]{JT}). They make nevertheless the remark (p.70 of op. cit.) that in the case $\cc{X} = \cc{S} \mr{q} \cc{E}$, with $\cc{E}$ an atomic topos, $G$ is the spatial group of automorphisms of the inverse image $F = q^*$ of the point. This idea was developed by Dubuc in \cite[proposition 4.7]{D1}, who explicitly constructed $G = \ell Aut(F)$ in this case, and described it as a universal $\rhd$-cone of $\ell$-bijections. In \cite[section 6]{DSz}, it is shown that this universal property is satisfied by the (neutral) Tannakian coend $End^\wedge(T)$, where $T=Rel(F)$.

Our work in this section is a generalization of these results, since we show that the equation \eqref{igualdadJT} above describes $G$ as a universal $\rhd$-cone of $\ell$-bijections (theorem \ref{JTGeneral}), and that the (non-neutral) Tannakian coend $End^\wedge(T)$ satisfies this property (proposition \ref{LcumpleG}). In this way an explicit construction of the groupoid $G$ is obtained.

Equation \eqref{igualdadJT} means that $\xymatrix{shG \ar@<1ex>[r]^{\partial_0^*} \ar@<-1ex>[r]_{\partial_1^*} & shG_0}$ satisfy the universal property that defines the   \mbox{pseudo-kernel} pair of $q$, i.e.

\begin{equation} \label{2po}
	\xymatrix@C=.5pc@R=.3pc{\Eat \ar[rrr]^{q^*} \ar[ddd]_{q^*} &&& shG_0 \ar[ddd]^{\partial_0^*}  \ar@/^3ex/[dddddrrr]^{f_0^*} \\ 
			&& \ar@{=>}[dl]^{\cong}_{\varphi}  &&&&&&  (\hbox{for each } \cc{F}, f_0^*, f_1^*, f_0^*q^* \Mr{\psi} f_1^*q^*,  \\
			&&&&& \ar@{=>}[dl]^{\cong}_{\psi}  &&&		\hbox{there exists a unique } \ell^* \hbox{ such that }   \\
			shG_0 \ar[rrr]_{\partial_1^*} \ar@/_2ex/[ddrrrrrr]_{f_1^*} &&& shG \ar@{.>}[ddrrr]|{\exists ! \ell^*} &&&&& \ell^* \partial_i^* = f_i^* \hbox{ and } id_{\ell^*} \circ \varphi = \psi)\\
			\\
			&&&&&& \cc{F} } 
\end{equation}

\begin{sinnadastandard} \label{cosasconA}
Take, as in section \ref{sec:Cmd=Rel}, $B = \O({G_0})$. By \ref{enumeratedeJT}, items 5,6, $(B \otimes B)$-locales \mbox{$B \otimes B \xr{g = (g_0,g_1)} A$} correspond to locales $\widetilde{A} \in Loc(sh(B \otimes B))$, $\gamma_* \widetilde{A} = A$ and the following diagram commutes.

\begin{equation} \label{dosdanuna}
\xymatrix{shB \ar[r]^{g_0^*} \ar[rd]_{\pi_1^*} & sh\widetilde{A} & shB \ar[l]_{g_1^*} \ar[dl]^{\pi_2^*} \\
						& sh(B\otimes B) \ar[u]^{g^*} }
\end{equation}



%
%

Consider also the following commutative diagram 

$$\xymatrix{ sh\widetilde{A}  \\
	    sh(B \otimes B) \ar[u]^{g^*} \\
	    \cc{S} \ar@/_8ex/[uu]_{\gamma^*} \ar[u]^{\gamma^*} }$$

Since the composition of spatial morphisms is spatial (see for example \cite[1.1]{JohnstoneFactorization}), then $sh\widetilde{A}$ is spatial (over $\cc{S}$), i.e. $sh\widetilde{A} \cong sh(\gamma_* \Omega_{\widetilde{A}} )$. But $\gamma_* \Omega_{\widetilde{A}} = \gamma_* {g}_* \Omega_{\widetilde{A}} = \gamma_* \widetilde{A} = A$.

\begin{equation} \label{abusotilde}
 \hbox{\emph{In the sequel, we make no distinction between }} shA \hbox{\emph{ and }} sh\widetilde{A}.
\end{equation}
\end{sinnadastandard}

\begin{sinnadastandard} \label{spatialreflection}
Recall from \cite[VI.5 p.53-54]{JT} the fact that
there is a left adjoint $F$ to the full and faithful functor $Loc^{op}(\cc{S}) = S \! p(\cc{S}) \stackrel{sh}{\hookrightarrow} Top/\cc{S}$, 
that maps $\cc{E} \mr{p} \cc{S}$ to $F(\cc{E}) = \overline{p_*(\Omega_{\cc{E}})}$. 
\end{sinnadastandard}

\begin{lemma} \label{2poigual2polocalic}
 The universal property defining the pseudo-push out \eqref{2po} is equivalent to the following universal property for localic topoi:
 
\begin{equation} \label{2polocalic}
	\xymatrix@C=.5pc@R=.2pc{\Eat \ar[rrr]^{q^*} \ar[ddd]_{q^*} &&& shG_0 \ar[ddd]^{\partial_0^*}  \ar@/^3ex/[dddddrrr]^{g_0^*} \\ 
			&& \ar@{=>}[dl]^{\cong}_{\varphi}  &&&&&&  (\hbox{for each } A, g_0^*, g_1^*, g_0^*q^* \Mr{\phi} g_1^*q^*,  \\
			&&&&& \ar@{=>}[dl]^{\cong}_{\phi}  &&&		\hbox{there exists a unique } {{a}}^* \hbox{ such that }   \\
			shG_0 \ar[rrr]_{\partial_1^*} \ar@/_2ex/[ddrrrrrr]_{g_1^*} &&& shG \ar@{.>}[ddrrr]|{\exists ! {{a}}^*} &&&&& {{a}}^* \partial_i^* = g_i^* \hbox{ and } id_{{{a}}^*} \circ \varphi = \phi)\\
			\\
			&&&&&& shA } 
\end{equation}
 
\end{lemma}

\begin{proof}
 Of course \eqref{2po} implies \eqref{2polocalic}. To show the other implication, given $\cc{F}, f_0^*, f_1^*, \psi$ as in \eqref{2po}, consider $\cc{F}$ as a topos over $sh (G_0 \times G_0)$ via $\cc{F} \mr{f = (f_0,f_1)} sh (G_0 \times G_0)$ and apply $F$ as in \ref{spatialreflection}. Then $\cc{O}(F(\cc{F})) = f_* \Omega_{\cc{F}}$ is a locale in $sh (G_0 \times G_0)$. Take $A = \gamma_* f_* \Omega_{\cc{F}}$ the corresponding locale over $B \otimes B$, $B \otimes B \mr{g = (g_0,g_1)} A$, i.e. $\widetilde{A} = f_* \Omega_{\cc{F}}$, then we have the commutative diagram \eqref{dosdanuna}.
 
 The hyperconnected factorization of $f$ is $\vcenter{\xymatrix{\cc{F} \ar[rd]_f \ar[rr]^{\eta} && shA \ar[ld]^g \\ & sh(G_0 \times G_0)}}$, where $\eta$ is the unit of the adjunction described in \ref{spatialreflection}. $\eta$ is hyperconnected (see \cite[VI.5 p.54]{JT}), in particular $\eta^*$ is full and faithful (see \cite[1.5]{JohnstoneFactorization}). Then $\eta^* g_0^* q^* \Mr{\psi} \eta^* g_1^* q^*$ determines uniquely $g_0^* q^* \Mr{\phi} g_1^* q^*$ such that $id_{\eta^*} \circ \phi = \psi$ and applying \eqref{2polocalic} we obtain

  $$
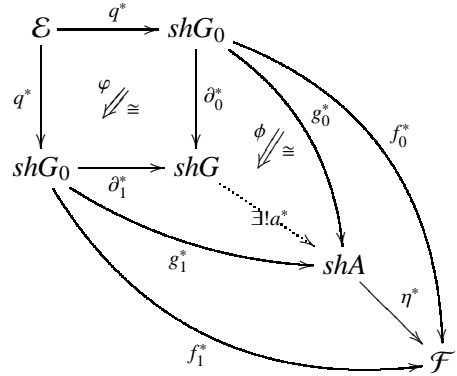
\xymatrix@C=.5pc@R=.7pc{\Eat \ar[rrr]^{q^*} \ar[ddd]_{q^*} &&& shG_0 \ar[ddd]^{\partial_0^*}  \ar@/^3ex/[dddddrrr]^{g_0^*}  \ar@/^6ex/[dddddddrrrrr]^{f_0^*} \\ 
			&& \ar@{=>}[dl]^{\cong}_{\varphi}  &&&&&&    \\
			&&&&& \ar@{=>}[dl]^{\cong}_{\phi}  &&&		  \\
			shG_0 \ar[rrr]_{\partial_1^*} \ar@/_2ex/[ddrrrrrr]_{g_1^*} \ar@/_6ex/[ddddrrrrrrrr]_{f_1^*} &&& shG \ar@{.>}[ddrrr]|{\exists ! {{a}}^*} &&&&& \\
			\\
			&&&&&& shA \ar[rrdd]^{\eta^*} \\ 
			\\
			&&&&&&&& \cc{F}} $$


Now, by the adjunction described in \ref{spatialreflection}, since taking sheaves is full and faithful, we have a bijective correspondence between morphisms ${{a}}^*$ and $\ell^*$ in the following commutative diagram:

\begin{equation} \label{adjuncionconeta}
\xymatrix@R=2.5pc@C=4pc{\cc{F} & shA \ar[l]_{\eta^*} \\
	  & shG \ar[u]_{{{a}}^*} \ar[ul]_{\ell^*} \\
	  & sh(G_0 \times G_0) \ar[uul]^{f^*} \ar[u]_{\partial^*} \ar@/_5ex/[uu]_{g^*}  }
\end{equation}

To finish the proof, we have to show that under this correspondence the conditions of \eqref{2po} and \eqref{2polocalic} are equivalent. The equivalence between $l^* \partial^* = f^*$ and ${{a}}^* \partial^* = g^*$ is immediate considering \eqref{adjuncionconeta}, and the equivalence between $id_{l^*} \circ \varphi = \psi$ and $id_{{{a}}^*} \circ  \varphi = \phi$ follows from $id_{\eta^*} \circ \phi = \psi$ using that $\eta^*$ is full and faithful.

%
%
%

%
%
%
%
%
%

\end{proof}

\begin{sinnadastandard} \label{correspondenceparaGalois} 
Consider a $B\otimes B$-locale $A$ as in \ref{cosasconA}. We have the correspondence

\renewcommand{\arraystretch}{-3}
\begin{tabular}{c c}
$g_0^* q^* \Mr{\phi} g_1^* q^* \hbox{ a natural isomorphism}$  & ${}_{_{\textstyle \hbox{by } \eqref{dosdanuna} }}$ \\ \cline{1-1} \noalign{\smallskip}
$g^* \pi_1^*  q^* \Mr{\phi} g^* \pi_2^* q^* \hbox{ a natural isomorphism}$ & ${}_{_{\textstyle \hbox{by } \ref{conodeellrelationsesnattransf} }}$ \\ \cline{1-1} \noalign{\smallskip}
$\hbox{A } \lozenge_1\hbox{-cone } \pi_1^* q^* X \times \pi_2^* q^* X \mr{\alpha_X} \widetilde{A} \; \hbox{ of } \ell \hbox{-bijections (in } sh(B \otimes B))$ & ${}_{_{\textstyle \hbox{by } \ref{dim1ydim2esdim} \hbox{, } \ref{trianguloesdiamante} }}$ \\ \cline{1-1} \noalign{\smallskip}
$\hbox{A } \rhd\hbox{-cone } \pi_1^* q^* X \times \pi_2^* q^* X \mr{\alpha_X} \widetilde{A} \; \hbox{ of } \ell \hbox{-bijections (in } sh(B \otimes B))$
\end{tabular}

%
%
%
%
%
%
%

In particular for $L = \O(G)$, the locale morphisms $\xymatrix{B \ar@<1ex>[r]^{s=\partial_0^{-1}} \ar@<-1ex>[r]_{t=\partial_1^{-1}} & L}$ induce a locale morphism 
  \mbox{$\xymatrix{B \otimes B \ar[r]^>>>>>{\gamma=(b,s)} & L}$,} and $\partial_0^* q^* \Mr{\varphi} \partial_1^* q^*$ correspond to a $\rhd$-cone $\pi_1^* q^* X \times \pi_2^* q^* X \mr{\lambda_X} \widetilde{L}$ of $\ell$-bijections.


\end{sinnadastandard}

\begin{theorem} \label{JTGeneral}
Given the previous data, \eqref{2po} is a pseudo-push out if and only if $\lambda$ is universal as a $\rhd$-cone of $\ell$-bijections (in the topos $sh(B \otimes B)$) in the following sense:

\begin{equation} \label{AutF}
\xymatrix
        {
         \pi_1^* q^* X \times \pi_2^* q^* X \ar[rd]^{\lambda_{X}}  
                      \ar@(r, ul)[rrd]^{\alpha_{X}} 
                      \ar[dd]_{\pi_1^* q^* (f) \times \pi_2^* q^* (f)}  
        \\
         {} \ar@{}[r]^(.3){\geq}
         & \;\; \widetilde{L} \;\; \ar@{-->}[r]^{\exists ! {{a}}} 
         & \;\widetilde{A}.  
        \\
         \pi_1^* q^* Y \times \pi_2^* q^* Y \ar[ru]^{\lambda_{Y}} 
         \ar@(r, dl)[rru]^{\alpha_{Y}} 
         && {({{a}} \; \text{is a locale morphism})}
        } 
\end{equation}
\end{theorem}

\begin{proof} By lemma \ref{2poigual2polocalic} it suffices to show that \eqref{2polocalic} is equivalent to \eqref{AutF}. We have shown in \ref{correspondenceparaGalois} that $\varphi$, $\phi$ in \eqref{2polocalic} correspond to $\lambda$, $\alpha$ in \eqref{AutF}.

Since taking sheaves is full and faithful, a morphism $\widetilde{L} \mr{{{a}}} \widetilde{A}$ of locales in $sh(B \otimes B)$ corresponds to the inverse image $shL \mr{{{a}}^*} shA$ (recall \eqref{abusotilde}) of a topoi morphism over $sh(B \otimes B)$, i.e. ${{a}}^*$ as in \eqref{2polocalic} satisfying ${{a}}^* \partial_i^* = f_i^*$, $i=0,1$. It remains to show that ${{a}} \lambda_X = \alpha_X$ for each $X$ in \eqref{AutF} if and only if $id_{{{a}}^*} \circ \varphi = \psi$ in \eqref{2polocalic}.

In the correspondence between ${{a}}$ and ${{a}}^*$ above, $\widetilde{L} \mr{{{a}}} \widetilde{A}$ is given by the value of ${{a}}^*$ in the subobjects of $1$ ($\widetilde{L} = \gamma_* \Omega_{shL}$, $\widetilde{A} = f_* \Omega_{shA}$), then we are in the hypothesis of \ref{atravesdeunmorfismodetopos} as the following diagram shows 

$$\xymatrix@R=4.5pc{& shL  \df{rr}{{{a}}^*}{{{a}}_*} && shA  \\
	  \cc{E} \ar@<.5ex>[rr]^{\pi_1^* q^*} \ar@<-.5ex>[rr]_{\pi_2^* q^*} && sh(B \otimes B), \dfbis{ul}{\gamma^*}{\gamma_*} \dfbis{ur}{f^*}{f_*}  }$$

and the proof finishes by corollary \ref{naturaligualconeconrhoatravesdeunmorfismodetopos}.
%
%
%
\end{proof}

\begin{remark} \label{levdeGalois}
 From proposition \ref{equivdeaccion}, we have that for each $X \in \Eat$, $\pi_1^* q^* X \times \pi_2^* q^* X \mr{\lambda_{X}} \widetilde{L}$ is equivalent to a discrete action $G \times_{G_0} X_{dis} \mr{\theta} X_{dis}$. In this way we can construct a lifting $\Eat \mr{\widetilde{q^*}} \beta^G$. This is the lifting $\Eat \mr{\phi} Des(q)$ of \cite[VIII.1 p.64]{JT}, composed with the equivalence $Des(q) \mr{\cong} \beta^G$ given by the correspondence in \ref{correspondenceparaGalois} for each $X$ (see \cite[VIII.3 proof of theorem 2, p.69]{JT}).
\end{remark}


\begin{sinnadastandard} {\bf The Tannakian context associated to a topos.} \label{Tannakacontext}
For generalities and notation concerning Tannaka theory see appendix \ref{sec:Tannaka}. 
Consider the fiber functor associated to the topos $\cc{E}$ (see \ref{rel}): 
$$T: \cc{R}el(\cc{E}) \xr{\cc{R}el(q^*)} \cc{R}el(shB) \mr{(-)_*} s\ell_0(shB) \mr{\gamma_*} (B \hbox{-} Mod)_0, \quad TX = (q^*X)_d.$$
This determines a Tannakian context as in appendix \ref{sec:Tannaka}, with 
$\cc{X} = \cc{R}el(\cc{E})$, $\cc{V} = s \ell$.
\end{sinnadastandard}

The universal property which defines the coend $End^\wedge(T)$ is that of a universal $\lozenge$-cone in the category of $(B \otimes B)$-modules, as described in the following diagram: 
$$
\xymatrix
        {
         & TX \otimes TX \ar[rd]^{\mu_{X}}  
                      \ar@(r, ul)[rrd]^{\phi_{X}}  
        \\
         TX \otimes TY \ar[rd]_{T(R) \otimes TY \quad} 
			           \ar[ru]^{TX \otimes T(R)^{\wedge} \;} 
	     & \equiv 
	     & \;\;End^\wedge(T)\;\; \ar@{-->}[r]^{\phi}  
         & \;Z.  
        \\
         & TY \otimes TY \ar[ru]^{\mu_{Y}} 
                          \ar@(r, dl)[rru]^{\phi_{Y}} 
         && {(\phi \; \text{is a linear map})}
        } 
$$

Via the equivalence $B\otimes B$-$Mod \cong s \ell(sh(B \otimes B))$, we can also think of this coend internally in the topos $sh(B \otimes B)$ as

$$
\xymatrix
        {
         & \pi_1^* T X \times \pi_2^* T X \ar[rd]^{\lambda_{X}}  
                      \ar@(r, ul)[rrd]^{\phi_{X}}  
        \\
         \pi_1^* T X \times \pi_2^* T Y \ar[rd]_{\pi_1^*T(R) \times \pi_2^*TY \quad} 
			           \ar[ru]^{\pi_1^*TX \times \pi_2^*T(R)^{\wedge} \;} 
	     & \equiv 
	     & \;\;End^\wedge(T)\;\; \ar@{-->}[r]^{\phi}  
         & \;Z.  
        \\
         & \pi_1^* T Y \times \pi_2^* T Y \ar[ru]^{\lambda_{Y}} 
                          \ar@(r, dl)[rru]^{\phi_{Y}} 
         && {(\phi \; \text{a linear map})}
        } 
$$

Depending on the context, it can be convenient to think of $End^\wedge(T)$ as a \mbox{$(B \otimes B)$-module} or as a sup-lattice in $sh(B \otimes B)$: to use general Tannaka theory, we consider modules, but to use the theory of $\lozenge$-cones developed in section \ref{sec:cones} we work internally in the topos $sh(B \otimes B)$.

We apply proposition \ref{extension} to obtain: 
\begin{proposition} \label{EndhasX}
The large coend defining $End^\wedge(T)$ exists and can be computed by the coend corresponding to the restriction of $T$ to the full subcategory of $Rel(\Eat)$ whose objects are in any small site $\Cat$ of definition of $\cc{E}$. \qed
\end{proposition}

We fix a small site $\Cat$ (with binary products and $1$) of the topos $\cc{E}$. Then $End^\wedge(T)$ can be constructed as a $(B \otimes B)$-module with generators $\mu_C(\delta_a \otimes \delta_b)$, where $\delta_a$, $\delta_b$ are the generators of $TC = (q^*C)_d$ (see proposition \ref{formula}), subject to the relations that make the $\lozenge$-diagrams commute. We will denote $[C,\delta_a,\delta_b] = \mu_C ( \delta_a \otimes \delta_b )$.

By the general Tannaka theory we know that $End^\wedge(T)$ is a cog\`ebro\"ide agissant sur $B$ and a $(B \otimes B)$-algebra. 
The description of the multiplication $m$ and the unit $u$ given below proposition \ref{bialg} yields in this case, for $C, \, D \in \cc{C}$ (here, $T(I) = T(1_\Cat) = B$):
\begin{equation} 
m([C,\, \delta_a,\delta_{a'}],\, [D, \,\delta_b,\delta_{b'}]) \;= \; [C \times D,\, (\delta_a \otimes \delta_b),(\delta_{a'} \otimes \delta_{b'})],\; \;\; u = \lambda_1.
\end{equation} 
When interpreted internally in $sh(B \otimes B)$, this shows that $\pi_1^* q^* C \times \pi_2^* q^* C \mr{\lambda_C} End^\wedge(T)$ is a compatible 
$\lozenge$-cone, with $End^\wedge(T)$ generated as a sup-lattice in $sh(B \otimes B)$ by the elements $\lambda_C ( a, b )$, thus by proposition \ref{compislocale} it follows that $End^\wedge(T)$ is a locale.

By proposition \ref{hopf}, we obtain that $End^\wedge(T)$ is also a (localic) \textit{Hopf cog\`ebro\"ide}, i.e. the dual structure in $Alg_{s \ell}$ of a localic groupoid.

\vspace{1ex}



\begin{sinnadastandard} 
{\bf The construction of $G$.} 
\end{sinnadastandard} 

\begin{\prop} \label{LcumpleG}
 Take $L = End^\wedge(T)$. Then $G = \overline{L}$ satisfies \eqref{AutF}, i.e. (by theorem \ref{JTGeneral}) satisfies \eqref{2po}.
\end{\prop}

\begin{proof}
 Given a $\rhd$-cone of $\ell$-bijections over a locale $A$, by proposition \ref{trianguloesdiamante} it factors uniquely through a $s \ell$-morphism which by proposition \ref{supisloc} is a locale morphism.
\end{proof}

We show now that $G$ actually is the localic groupoid considered by Joyal and Tierney. By theorem \ref{JTGeneral}, the dual $L$ of a groupoid $G$ satisfying \eqref{2po} is unique as a locale in    $sh(B \otimes B)$, and so are the $\lambda_X$ corresponding to the $\varphi$ in \eqref{2po}.

Now, remark \ref{remarkdeunicidad}, interpreted for $G = \overline{L}$ using proposition \ref{equivdeaccion}, states that $i = \overline{e}$, $\circ = \overline{c}$ are the only possible localic groupoid structure (with inverse given as $(-)^{-1} = \overline{a}$, see proposition \ref{hopf}) such that the lifting $\widetilde{q^*}$ lands in $\beta^G$ (see remark \ref{levdeGalois}). We have proved:


\begin{theorem} \label{G=H}
Given any topos $\Eat$ over a base topos $\Sat$, and a spatial cover \mbox{$sh{G_0} \mr{q} \Eat$,} the localic groupoid $\; G \; = \; \xymatrix{G \stackrel[G_0]{\textcolor{white}{G_0}}{\times} G \ar[r]^>>>>>>{\circ} & G \ar@<1.3ex>[r]^{\partial_0} \ar@<-1.3ex>[r]_{\partial_1} & G_0 \ar[l]|{\ \! i \! \ }}$ considered in \cite{JT} can be constructed as \mbox{$G = \overline{End^\wedge(Rel(q^*))}$,} with $i = \overline{e}$, $\circ = \overline{c}$ and inverse $(-)^{-1} = \overline{a}$. The lifting $\Eat \mr{\widetilde{q^*}} \beta^G$ is also unique and defined as in remark \ref{levdeGalois}. \qed  
\end{theorem}

\section{The main theorems} \label{sec:MainTheorems}
 A topoi morphism $shB \mr{q} \cc{E}$, with inverse image $\cc{E} \mr{q^*} shB$, determines by theorem \ref{G=H} a situation described in the following diagram
\begin{equation} \label{diagramacompleto}
\xymatrix
        {\beta^G \ar[rr] \ar[dd]_U   &&    \cc{R}el(\beta^G) \ar[rr]^{\cong} \ar[dd]_<<<<<<{Rel(U)}   &&    Cmd_0(L) \ar[dd]^U   \\
& \cc{E} \ar[dl]^{q^*} \ar[rr] \ar[lu]_{\widetilde{q^*}} && \cc{R}el(\cc{E}) \ar[ul]_{\cc{R}el(\widetilde{q^*})}  \ar[dr]^{T} \ar[dl]_{Rel(q^*)} \ar[ur]^{\widetilde{T}}  \\
          shB \ar[rr] && Rel(shB)\ar[r]^\cong_{(-)_*} & s\ell_0(shB) \ar[r]^\cong_{\gamma_*} & (B\hbox{-}Mod)_0 . }
\end{equation}
\noindent
where $L = End^\wedge(T)$, $G = \overline{L}$ and the isomorphism in the first row of the diagram is given by \mbox{Theorem \ref{Comd=Rel}.} Note that the triangle on the left is the Galois context associated to the topos, and the triangle on the right is the one of the Tannakian context. Using \ref{dcrequivalence} we obtain

\begin{theorem} \label{nonneutralAA}
The (Galois) lifting functor $\widetilde{q^*}$ is an equivalence if and only if the (Tannaka) lifting functor $\widetilde{T}$ is such. \qed
\end{theorem}

In \cite[VIII.3, theorem 2, p.68]{JT} (see remark \ref{levdeGalois}) it is proved the following:

\begin{theorem} \label{fundamentalGT}
If $q$ is an open surjection then the Galois lifting functor $\widetilde{q^*}$ is an equivalence. \qed
\end{theorem}

Combining both results we obtain the following Tannakian recognition type theorem:  

\begin{theorem} \label{AA}
The Tannaka lifting functor $\widetilde{T}$ corresponding to an open surjection $q$ is an equivalence. \qed
\end{theorem} 

\begin{sinnadastandard}{\bf Tannaka theory for dcr.} 
 We show now that theorem \ref{AA} can be stated purely in the language of $s\ell$ categories as a Tannakian recognition theorem for a particular type of $s\ell$-categories. We also describe our current progress regarding the extension of this theorem to general $s\ell$-categories.
 
 In \cite[chapters 1 and 2]{Pitts}, it is shown that the construction of \ref{extensionarel} extends to a $2$-fully faithful functor $Top^{op} \mr{Rel} DCR$. Here $DCR$ is the $2$-category with
\end{sinnadastandard}

\begin{itemize}
 \item Objects: distributive categories of relations (dcr). These generalize the $s \ell$-categories $Rel(\cc{E})$. By definition, a dcr is a cartesian $s \ell$-category in which every object is discrete (see \mbox{\cite[2.1 p.444]{Pitts},} for details). Dcr are also a generalization (horizontal categorification) of locales, in the sense that a locale is precisely a dcr with a single object. More generally, in any cartesian $s \ell$-category, the hom sup-lattices are actually locales.
 \item Morphisms: a morphism of dcr is a $s \ell$-functor that preserves this structure (see   \mbox{\cite[2.4 p.447]{Pitts}} for details). In op. cit. it is shown that a $s \ell$-functor $\cc{X} \mr{T} \cc{Y}$ between dcr is a morphism of dcr if and only if the sup-lattice morphisms $\cc{X}(X,X') \mr{} \cc{Y}(TX,TX')$ are locale morphisms.
 Then, in particular, an equivalence of $s \ell$-categories is automatically a morphism of dcr (since an isomorphism in $s \ell$ is automatically a locale morphism).
 \item $2$-cells: they are the lax natural transformations $T \mr{\varphi} S$ (\cite[1.3(iii)]{Pitts}) whose components $TX \mr{\varphi} SX$ are \emph{maps}. The notion of map in a dcr (which is the same as in any $2$-category, namely to be a left adjoint) coincides with the notion of function in the dcr $Rel(\cc{E})$. An invertible arrow in a dcr is always a map (just like an invertible relation is a function), and a lax natural isomorphism is automatically natural. Then, the notion of equivalence $T \mr{\varphi} S$ in the $2$-category DCR coincides with the notion of natural equivalence.
\end{itemize}

\begin{sinnadastandard}\label{dcrequivalence} It follows  that the inverse image $q^*$ of a topoi morphism is an equivalence if and only if $T = Rel(q^*)$ is an equivalence in $DCR$, which by the previous observations happens if and only if $T$ is an equivalence of $s \ell$-categories (a fully faithful, essentially surjective $s\ell$-functor).
\end{sinnadastandard}

%
 
 By identifying the essential image of the functor $Top^{op} \mr{Rel} DCR$, Carboni and Walters \cite{CW} obtain an equivalence of $2$-categories $Top^{op} \mr{Rel} bcDCR$, where $bcDCR$ is the full subcategory of $DCR$ consisting of the bounded and complete dcr (see \cite[2.5]{Pitts}). In \cite[lemma 4.3]{Pitts}, it is shown that under this equivalence (the inverse image of) an open surjection corresponds to an open morphism of DCRs (see \cite[2.4 (ii)]{Pitts}) that is faithful as a functor. Combining these results with theorem \ref{AA} we obtain the following recognition theorem.

\begin{theorem} 
 Let $\cc{A} \in bcDCR$, $B \in Loc$, $\cc{A} \mr{T} (B$-$Mod)_0$ a morphism of dcr. Then the coend \mbox{$L = End^\wedge(T)$} of \ref{defdeL} exists, and if $T$ is open and faithful then the lifting $\widetilde{T}$ of \ref{lifting} is an equivalence. \qed
\end{theorem}

\begin{sinnadastandard} {\bf Tannaka theory over sup-lattices.} 
We have shown that theorem \ref{AA} corresponds to a Tannakian recognition theorem for $dcr$, which are a particular case of $s \ell$-enriched categories. But we should note that in a sense, this theorem combines a purely recognition theorem (the lifting is an equivalence) and an ``additional structure'' theorem 
(where extra structure is given to the coend $L$ under extra hypothesis, in this case the cog\"ebro\`ide $L$ is a (localic) Hopf cog\"ebro\`ide).

But we can also consider a Tannakian context for a general $s \ell$-enriched category, not necessarily the category of relations of a topos. This general Tannakian context doesn't correspond to a Galois context, so a priori we don't have a recognition theorem, and we can't obtain one from the results of \cite{JT}. The (more ambitious) objective here is to obtain the results of \cite{JT} via Tannakian methods, since we have shown that they correspond to the particular case of dcr.

Note that the definition of an open morphism between DCRs (\cite[4.1]{Pitts}) uses only their underlying structure of $s \ell$-enriched categories, therefore we may consider open faithful $s \ell$-functors between $s \ell$-categories. The same happens for the definitions of bounded and complete.

We have been able to generalize the result of proposition \ref{extension} to bounded $s \ell$-categories   \mbox{(\cite[8.7]{tesis}),} which allows us to construct the Tannakian coend $L = End^\wedge(T)$ given a fiber functor $T$ from a bounded $s \ell$-category $\cc{A}$:

\begin{theorem} \label{contextoparasl}
 Let $\cc{A}$ be a bounded $s \ell$-category, $B \in Alg_{s \ell}$, $\cc{A} \mr{T} (B$-$Mod)_0$ a functor. Then the coend $L = End^\wedge(T)$ of \ref{defdeL} exists, therefore so does the lifting $\widetilde{T}$ of \ref{lifting}.
 \qed
\end{theorem}

Based on our previous developments, we end this paper with the conjecture that a following more general recognition theorem may hold, which would imply theorem \ref{AA} (and therefore theorem \ref{fundamentalGT} of Joyal-Tierney), for $s \ell$-enriched categories and comodules of a (not necessarily localic Hopf) cog\"ebro\`ide. Note that an analysis of the  properties of $Cmd_0(L)$ and of the forgetful functor $Cmd_0(L) \mr{} (B$-$Mod)_0$ could lead to adding some extra hypothesis to this conjecture. 

\begin{conjecture} \label{conjetura}
 In the hypothesis of theorem \ref{contextoparasl}, if $T$ is a \mbox{$s \ell$-enriched} open and faithful functor then  $\widetilde{T}$ is an equivalence.
\end{conjecture}
\end{sinnadastandard}

\begin{appendices}
\section{Non-neutral Tannaka theory} \label{sec:Tannaka} 

In this section we make the constructions needed to develop a Non-neutral Tannaka theory (as in \cite{De}), over a general tensor category ($\Vat, \otimes, k)$. Let $B', B \in Alg_\Vat$.

%
%

\begin{sinnadastandard}{\bf Duality of modules.}
 
\end{sinnadastandard}

\begin{\de}\label{moddual} Let $M \in B$-Mod. We say that $M$ has a right dual (as a $B$-module) if there exists $M^\wedge \in$ Mod-$B$, $M \otimes M^\wedge \stackrel{\eps}{\rightarrow} B$ morphism of $B$-Mod-$B$ and $k \stackrel{\eta}{\rightarrow} M^\wedge \otimes_B M$ morphism of $\Vat$ such that the triangular equations

\begin{equation} \label{triangular}
\xymatrix@C=-0.3pc@R=0.1pc
         { 			
          &  \ar@{-}[ldd] \ar@{-}[rdd] \ar@{}[dd]|{\eta} 
          & & &  M^{\!\wedge} \ar@2{-}[dd] 					
          & & & & & & & & &  M \ar@2{-}[dd] 
          & & & \ar@{-}[ldd] \ar@{-}[rdd] \ar@{}[dd]|{\eta} 
          & & & &  
          \\
		  & & & & & & & M^{\!\wedge} \ar@2{-}[dd] 
		  & & & & & & & & & & & & &  M \ar@2{-}[dd] 
		  \\
		    M^{\!\wedge} \ar@2{-}[dd] \ar@{}[rr]|*+<.6ex>[o][F]{\scriptscriptstyle{B}}
		  & \,\,\, & M \ar@{-}[rdd] 
		  & \ar@{}[dd]|{\varepsilon} 
		  & M^{\!\wedge} \ar@{-}[ldd] 
		  & \quad = \quad 
		  & & & & & \quad \quad \hbox{and} \quad\quad 
		  & & & M \ar@{-}[ddr] 
		  & \ar@{}[dd]|{\varepsilon} 
		  & M^{\!\wedge} \ar@{-}[ldd] \ar@{}[rr]|*+<.6ex>[o][F]{\scriptscriptstyle{B}} 
		  & \,\,\, & M \ar@2{-}[dd] 
		  & \quad = \quad 
		  & &  
		  \\
		  & & & & & & & M^{\!\wedge}  & & & & & & & & & & & & &  M. 
		  \\
			M^{\!\wedge} \ar@{}[rrr]|*+<.6ex>[o][F]{\scriptscriptstyle{B}}
		  & & & B & & & & & & & & & & & B \ar@{}[rrr]|*+<.6ex>[o][F]{\scriptscriptstyle{B}} & & & M	
		 }
\end{equation}
hold. In this case, we say that $M^\wedge$ is \emph{the} right dual of $M$ and we denote $M \dashv M^\wedge$.
\end{\de}

\begin{remark} \label{nuevodualcomobimod}
 If $B$ is commutative, the notion of dual as a $B$-module coincides via the inclusion $B$-mod $\hookrightarrow B$-bimod with the notion of dual in the monoidal category $B$-bimod.
\end{remark}

\begin{\prop}
A duality $M \dashv M^\wedge$ yields an adjunction $$\xymatrix { B'\hbox{-Mod} \ar@/^2ex/[rr]^{(-)\otimes M^\wedge} \ar@{}[rr]|{\bot} && B'\hbox{-Mod-}B \ar@/^2ex/[ll]^{(-) \stackrel[B]{}{\otimes}  M} }$$
given by the binatural bijection between morphisms

\begin{equation}\label{adjunciondeldual}
\begin{tabular}{c}
 $N \otimes M^\wedge \stackrel{\lambda}{\rightarrow} L$ of $B'$-$Mod$-$B$ \\ \hline \noalign{\smallskip}
 $N \stackrel{\rho}{\rightarrow} L \stackrel[B]{}{\otimes} M$ of $B'$-$Mod$
\end{tabular}
\end{equation}

for each $N \in B'$-Mod, $L \in B'$-Mod-$B$.
\end{\prop}

\begin{proof}
The bijection is given by
\begin{equation} \label{lambdarhoconascensores} \xymatrix@C=-0.3pc@R=1pc{&& N \op{\rho} &&& M^{\!\wedge} \did \\
						\lambda: \quad & L \did \ar@{}[rr]|*+<.6ex>[o][F]{\scriptscriptstyle{B}} && M && M^{\!\wedge} & \quad , \quad \quad \\
						& L \ar@{}[rrr]|*+<.6ex>[o][F]{\scriptscriptstyle{B}} &&& B \cl{\eps} }
\xymatrix@C=-0.3pc@R=1pc{& N \did &&& \op{\eta} \\
\rho: \quad & N && M^{\!\wedge} \ar@{}[rr]|*+<.6ex>[o][F]{\scriptscriptstyle{B}} & \,\,\, & M \did \\
& & L \cl{\lambda} \ar@{}[rrr]|*+<.6ex>[o][F]{\scriptscriptstyle{B}} &&& M.} \end{equation}
All the verifications are straightforward.
\end{proof}

%

\begin{\de} \label{catconduales}
We will denote by $(B$-Mod$)_r$ the full subcategory of $B$-Mod consisting of those modules that have a right dual.
\end{\de}

\begin{\prop} \label{dualesfuntor}
There is a contravariant functor $(-)^\wedge: (B$-Mod$)_r \rightarrow $Mod-$B$ defined on the arrows $M \mr{f} N$ as
$$\xymatrix@C=-0.3pc@R=1pc{&&\op{\eta} &&  & N^{\!\wedge} \did \\
f^\wedge: \quad & M^{\!\wedge} \did \Brr & \,\,\, & M \dcell{f} && N^{\!\wedge} \did \\
&M^{\!\wedge} \did \Brr && N && N^{\!\wedge} \\
&M^{\!\wedge} \Brrr &&& B \cl{\eps}}$$ 
\qed
\end{\prop}

\begin{sinnadastandard}\label{sub:natpredual} {\bf The $Nat^\lor$ adjunction} 
 
\end{sinnadastandard}

Consider now a category $\Cat$ and a functor $H: \Cat \rightarrow $Mod-$B$. We have an adjunction 
\begin{equation} \label{homadjunction}
\xymatrix { (B'\hbox{-Mod})^\Cat \ar@/^2ex/[rr]^{(-) \stackrel[\Cat]{}{\otimes} H} \ar@{}[rr]|{\bot} && B'\hbox{-Mod-}B \ar@/^2ex/[ll]^{Hom_B (H,-)} }
\end{equation}
where the functors are given by the formulae $$F \otimes_\Cat H = \int^{X \in \Cat} FX \otimes HX, \quad Hom_B(H,M)(C)=Hom_B(HC,M).$$

Assume now we have a full subcategory $(B$-Mod$)_0$ of $(B$-Mod$)_r$ (recall definition \ref{catconduales}), i.e. a full subcategory $(B$-Mod$)_0$ of $B$-$Mod$ such that every object has a right dual. Given \mbox{$G: \Cat \rightarrow (B$-Mod$)_0$,} using proposition \ref{dualesfuntor} we construct $G^\wedge: \Cat \rightarrow $Mod-$B$.

\begin{\de} \label{defnatpredual}
Given $G: \Cat \rightarrow (B$-Mod$)_0$, $F: \Cat \rightarrow B'$-Mod, we define $$Nat^\wedge(F,G) = F \otimes_\Cat G^\wedge = \int^{X \in \Cat} FX \otimes GX^\wedge.$$
\end{\de}

\begin{sinnadastandard} \label{notasobresimetria}
 A note regarding left and right duality is in order here. There are two possible different (symmetric) definitions of the Tannakian fundamental object, namely the one above and the one we will denote by $Nat^\lor(F,G) = \int^{X \in \Cat} GX^\lor \otimes FX$, which is constructed under the appropriate hypothesis, symmetric to the ones in definition \ref{defnatpredual}. 
 Tannakian theory can be developed using either one of the constructions, and by considering the opposite of the tensor products each one becomes the other. In the commutative case both constructions (and their variations present in the literature) coincide, but as Deligne points out in \cite{De}, considering the non-commutative case helps us not to mistake left for right. 
 
  If we consider the motivation in \cite{JS} for the definition of a \emph{predual} of $[F,G]$, we observe that both possible definitions satisfy preduality in the following sense. Consider an object $C$ in a tensor category, then a left adjoint to the functor $(-) \otimes C$ is the usual internal Hom functor $Hom(C,-)$, but the functor $C \otimes (-)$ can be considered instead, let's denote by $Hom^r(C,-)$ this Hom functor. Note that both Hom functors satisfy $[I,Hom(C,D)] = [C,D]$, there is no objective reason to prefer one over the other. A left dual $C^\lor$ of $C$ yields the equality $Hom(C,D) = D \otimes C^\lor$, and in this way $Hom(C,I)$ is the unique possible candidate for a left dual of $C$, even if $C$ doesn't admit one. Similarly, $Hom^r(C,I)$ is the unique possible candidate for a right dual of $C$. It can be seen that $Hom(\int^{X \in \Cat} FX \otimes GX^\wedge,I) = [F,G]$, therefore $Nat^\wedge(F,G)$ is the unique possible candidate for a right dual of $[F,G]$, and also $Hom^r(\int^{X \in \Cat}
 GX^\lor \otimes FX,I) = [F,G]$, therefore $Nat^\lor(F,G)$ is the unique possible candidate for a left dual of $[F,G]$.
 
 The reason why we use definition \ref{defnatpredual} in this paper is because, as the reader can see below, it corresponds to left comodules, which in turn correspond to actions of the groupoid as we showed in the beginning of section \ref{equiv:objects}. We note however that, since we're dealing with the commutative case, the other definition is also possible (see remark \ref{rightcomodules}).
 
%
%

\end{sinnadastandard}

\begin{\prop} \label{natpredualadjunction}
Given $G: \Cat \rightarrow (B$-Mod$)_0$, we have an adjunction
\begin{equation} 
\xymatrix { (B'\hbox{-Mod})^\Cat \ar@/^2ex/[rr]^{Nat^\wedge((-),G)} \ar@{}[rr]|{\bot} && B'\hbox{-Mod-}B \ar@/^2ex/[ll]^{(-) \stackrel[B]{}{\otimes} G} }
\end{equation}
where the functor $(-) \otimes_B G$ is given by the formula $(M \otimes_B G) (C) = M \otimes_B (GC)$.
\end{\prop}
\begin{proof}
The value of the functor $Nat^\wedge((-),G)$ in an arrow $F \stackrel{\theta}{\Rightarrow} H$ of $(B'\hbox{-Mod})^\Cat$ is the   \mbox{$B'$-$B$-bimodule} morphism induced by
$$FX \otimes GX^\wedge \mr{\theta_X \otimes (GX)^\wedge} HX \otimes GX^\wedge \mr{\lambda_X} Nat^\wedge(H,G).$$
The adjunction is given by the binatural bijections

\begin{center}
 \begin{tabular}{c}
  $Nat^\wedge (F,G) \rightarrow C$ \\ \hline \noalign{\smallskip}
  $F \stackrel[\Cat]{}{\otimes} G^\wedge \rightarrow C$ \\ \hline \noalign{\smallskip}
  $F \Rightarrow Hom_B(G^\wedge, C)$ \\ \hline \noalign{\smallskip}
  $F \Rightarrow C \stackrel[B]{}{\otimes} G$ 
 \end{tabular}

\end{center}

\noindent justified by the adjunction \eqref{homadjunction}. We leave the verifications to the reader.
\end{proof}

The unit of the adjunction is called the \emph{coevaluation} $F \Mr{\rho = \rho_F} Nat^\wedge(F,G) \otimes_B G$. It can be checked that it is given by $$\rho_C: FC \mr{FC \otimes \eta} FC \otimes GC^\wedge \otimes_B GC \mr{\lambda_C \otimes GC} Nat^\wedge(F,G) \otimes_B GC,$$
i.e. that it corresponds to $\lambda_C$ via the correspondence \eqref{lambdarhoconascensores}.

We also have the counit $Nat^\wedge(L \otimes_B G, G) \mr{e = e_L} L$. It is induced by the arrows   $L \otimes_B GC \otimes GC^\wedge \mr{L \otimes_B \eps} L$.

\begin{center}
We now restrict to the case $B'=B$. 
\end{center}

\begin{\de} \label{defdeL}
Given $F: \Cat \rightarrow (B$-$Mod)_0$ , we define $$L = L(F) = End^\wedge(F) = Nat^\wedge(F,F).$$
\end{\de}


As usual, given $F,G,H: \Cat \rightarrow (B$-Mod$)_0$ we construct from the coevaluation a \emph{cocomposition} $$Nat^\wedge(F,H) \stackrel{c}{\rightarrow} Nat^\wedge(F,G) \stackrel[B]{\textcolor{white}{B}}{\otimes} Nat^\wedge(G,H)$$
This is a $B$-bimodule morphism induced by the arrows $$FC \otimes HC^\wedge \mr{FC \otimes \eta \otimes HC^\wedge} FC \otimes GC^\wedge \stackrel[B]{\textcolor{white}{B}}{\otimes} GC \otimes HC^\wedge \mr{\lambda_C \otimes \lambda_C} Nat^\wedge(F,G) \stackrel[B]{\textcolor{white}{B}}{\otimes} Nat^\wedge(G,H)$$

The structure given by $c$ and $e$ is that of a \emph{cocategory enriched over} $B$-Bimod. Therefore, \mbox{$L = L(F)$} is a coalgebra in the monoidal category $B$-Bimod, i.e. a $B$-bimodule with a coassociative comultiplication $L \mr{c} L \otimes_B L$ and a counit $L \mr{e} B$. This is called a \emph{cog\'ebro\"ide agissant sur B} in \cite{De}. Cog\'ebro\"ides act on $B$-modules as follows

\begin{\de} \label{defcmd}
Let $L$ be a \emph{cog\'ebro\"ide agissant sur B}, i.e. a coalgebra in $B$-Bimod. A (left) representation of $L$, which we will also call a (left) $L$-comodule, is a $B$-module $M$ together with a coaction, or comodule structure $M \mr{\rho} L \otimes_B M$, which is a morphism of $B$-modules such that
$$C1: \vcenter{\xymatrix@C=-0.3pc@R=1pc{&&& & M \ar@{-}[d] \ar@{-}[dlll] \ar@{}[dll]|{\ \ \rho} \\
																		& L \op{c} & \Brr & \,\,\, & M \did \\
																		L \Brr && L \Brr && M}}
\vcenter{\xymatrix{=}}
\vcenter{\xymatrix@C=-0.3pc@R=1pc{&&&& M \ar@{-}[d] \ar@{-}[dllll] \ar@{}[dll]|{\rho} \\
																	L \did & \Brr &&& M \drho{\rho} & \quad \quad \quad \quad \\
																		L \Brr & \,\,\, & L \Brr & \,\,\, & M}}
C2: \vcenter{\xymatrix@C=-0.3pc@R=1pc{&& M \drho{\rho} \\ L \dcell{e} \Brr & \,\,\, & M \did \\ B \Brr && M}}
\vcenter{\xymatrix{=}}
\vcenter{\xymatrix@C=-0.3pc@R=1pc{M \did \\ M}}$$
We define in an obvious way the comodule morphisms, and we have that way a category $Cmd(L)$. We denote by $Cmd_0(L)$ the full subcategory of those comodules whose subjacent $B$-module is in $(B$-Mod$)_0$.
\end{\de}

\begin{\prop} \label{lifting}
Given $F: \Cat \rightarrow (B$-Mod$)_0$, the unit $FC \mr{\rho_C} L \otimes_B FC$ yields a comodule structure for each $FC$. Then we obtain a lifting of the functor $F$ as follows
$$\xymatrix{\Cat \ar@{.>}[r]^{\tilde{F}} \ar[d]_F & Cmd_0(L) \ar[dl]^U \\ (B\hbox{-Mod})_0}$$
\qed
\end{\prop}

%
%


\begin{lemma} \label{lemadeunicidad}
Let $M \in (B$-$Mod)_r$, $L \in B$-$Bimod$, $M \otimes M^\wedge \mr{\lambda} L$ in $B$-$Bimod$, and $\rho$ the corresponding $B$-module morphism via \eqref{lambdarhoconascensores}. Let $L \mr{e} B$, $L \mr{c} L \otimes_B L$ be a structure of cog\'ebro\"ide sur $B$. Then $\rho$ is a comodule structure for $M$ if and only if the following diagrams commute:
$$B1: \vcenter{\xymatrix{ M \otimes M^\wedge \ar[r]^{\lambda} \ar[d]_{M \otimes \eta \otimes M^\wedge } & L \ar[d]^{c} \\
M \otimes M^\wedge \stackrel[B]{\textcolor{white}{B}}{\otimes} M \otimes M^\wedge  \ar[r]^>>>>>{\lambda \stackrel[B]{}{\otimes} \lambda} & L \stackrel[B]{\textcolor{white}{B}}{\otimes} L}}
\quad \quad \quad B2: \vcenter{\xymatrix{M \otimes M^\wedge \ar[r]^{\lambda} \ar[d]_{\eps} & L \ar[dl]^{e} & \quad \quad \quad \quad \\ B}}$$
\end{lemma}

\begin{proof}
We can prove $B1 \iff C1$, $B2 \iff C2$. All the implications can be proved in a similar manner when using a graphical calculus, we show $C1 \implies B1$:

$$\vcenter{\xymatrix@C=-0.3pc@R=1pc{&& M \did &&&& \ar@{-}[dll] \ar@{}[d]|{\eta} \ar@{-}[drr] &&&& M^\wedge \did \\
				   && M \ar@{-}[d] \ar@{-}[dll] \ar@{}[dl]|{\rho} && M^\wedge \did \Brr &&&& M \ar@{-}[d] \ar@{-}[dll] \ar@{}[dl]|{\rho} && M^\wedge \did \\
				   L \did \Brr & \quad & M && M^\wedge \Brr & \quad & L \did \Brr & \quad & M && M^\wedge \\
				   L \Brr &&& B \cl{\eps} & \Brr && L \Brr &&& B \cl{\eps}}}
\vcenter{\xymatrix{\stackrel{\triangle}{=}}}
\vcenter{\xymatrix@C=-0.3pc@R=1pc{&&&& M \ar@{-}[d] \ar@{-}[dllll] \ar@{}[dll]|{\ \ \rho} && M^\wedge \did \\
	    L \did & \Brr &&& M \ar@{-}[d] \ar@{-}[dll] \ar@{}[dl]|{\rho} && M^\wedge \did \\
	    L \did \Brr & \quad & L \did \Brr & \quad & M && M^\wedge \\
	    L \Brr && L \Brrr &&& B \cl{\eps}}}
    \vcenter{\xymatrix{\stackrel{C1}{=}}}
	    \vcenter{\xymatrix@C=-0.3pc@R=1pc{&&& & M \ar@{-}[d] \ar@{-}[dlll] \ar@{}[dll]|{\ \ \rho} && M^\wedge \did \\
	    & L \op{c} & \Brr & \,\,\, & M && M^\wedge \\
	    L \Brr && L \Brrr &&& B \cl{\eps}}}$$

\end{proof}

\begin{remark} \label{remarkdeunicidad}
The previous lemma implies that $L \mr{e} B$, $L \mr{c} L \otimes_B L$ as defined before is the only possible cog\`ebro\"ide structure for $L$ that make each $\rho_X$ a comodule structure.
\end{remark}

We now give $L$ additional structure under some extra hypothesis 

\begin{proposition}\label{bialg}
 If $\Cat$ and $F$ are monoidal, and $\Vat$ has a symmetry, then $L$ is a   $B \otimes B$-algebra. If in addition $\Cat$ has a symmetry and $F$ respects it, $L$ is commutative (as an algebra). \qed 
\end{proposition}

We will not prove this proposition here, but show how the multiplication and the unit are constructed, since they are used explicitly in \ref{Tannakacontext}.
The multiplication \mbox{$L\stackrel[B \otimes B]{\textcolor{white}{B}}{\otimes}L \mr{m} L$} is induced by the composites
$$
m_{X,Y}: (FX \otimes FX^\wedge)\stackrel[B \otimes B]{\textcolor{white}{B}}{\otimes}(FY \otimes FY^\wedge) \mr{\cong} (FX \stackrel[B]{}{\otimes} FY) \otimes (FY^\wedge \stackrel[B]{}{\otimes} FX^\wedge)$$

\hfill $\mr{\cong} F(X \otimes Y) \otimes F(X \otimes Y)^\wedge \mr{\lambda_{X \otimes Y}} L.$

The unit is given by the composition
$$
u: B \otimes B \mr{\cong} F(I) \otimes F(I)^\wedge \mr{\lambda_{I}} L.
$$
  
\begin{proposition}\label{hopf}
 If in addition $\Cat$ has a duality, then $L$ has an antipode. \qed 
\end{proposition}
  
The antipode $L \mr{a} L$ is induced by the composites
$$
a_X: FX \otimes FX^\wedge  \mr{\cong} F(X^\wedge) \otimes FX \mr{\lambda_{X^\wedge}} L.
$$

\section{Elevators calculus} \label{ascensores}
This is a graphic notation invented by the first author in 1969 to write equations in monoidal categories, ignoring associativity and suppresing the neutral object $I$. Given an algebra $B$ 
 we specify with a $\xymatrix@C=0pc{{\ar@{}[rr]|*+<.6ex>[o][F]{\scriptscriptstyle{B}}} && }$ the tensor product $\otimes_B$ over $B$, and leave the tensor product $\;\otimes\;$ of the monoidal category unwritten. Arrows are written as cells, the identity arrow as a double line, and the symmetry as crossed double lines. This notation exhibits clearly the permutation associated 
to a composite of different symmetries, allowing to see if any two composites are the same by simply checking that they codify the same permutation\footnote
         { This is justified by a simple coherence theorem for symmetrical categories, particular case of \cite{JS2} Corollary 2.2 for braided categories.
         }. Compositions are read from top to bottom. 

Given arrows $f:C \rightarrow D$, $f':C' \rightarrow D'$, the  bifunctoriality of the tensor product is the basic equality:
\begin{equation} \label{ascensor}
\xymatrix@C=0ex
         {
             C \dcell{f} & C' \did
          \\
             D \did & C' \dcell{f'}
          \\
             D  &  D' 
         }
\xymatrix@R=6ex{\\ \;\;\;=\;\;\; \\}
\xymatrix@C=0ex
         {
             C \did & C'\dcell{f'}
          \\
             C \dcell{f} & D' \did
          \\
             D & D' 
         }
\xymatrix@R=6ex{ \\ \;\;\;=\;\;\; \\}
\xymatrix@C=0ex@R=0.9ex
         {
             {} & {}
          \\
               C   \ar@<4pt>@{-}'+<0pt,-6pt>[ddd] 
                   \ar@<-4pt>@{-}'+<0pt,-6pt>[ddd]^{f}
             & C'  \ar@<4pt>@{-}'+<0pt,-6pt>[ddd] 
                   \ar@<-4pt>@{-}'+<0pt,-6pt>[ddd]^{f'}
          \\ 
             {} & {}
          \\ 
             {} & {}
          \\
             D & D' 
         }
\end{equation}
This allows to move cells up and down when there are no obstacles, as if they were elevators. There are also  similar elevators with the symbol $\xymatrix@C=0pc{{\ar@{}[rr]|*+<.6ex>[o][F]{\scriptscriptstyle{B}}} && }$.

The naturality of the symmetry is the basic equality:
\begin{equation} \label{swap}
\xymatrix@C=0ex
         {
             C \dcell{f} & C' \did
          \\
             D \did & C' \dcell{f'}
          \\
             D \ar@{=} [dr] & D' \ar@{=} [dl]
          \\
             D' & D 
         }
\xymatrix@R=10ex{ \\ \;\;\;=\;\;\; \\}
\xymatrix@C=0ex
         {
             C \dcell{f} & C' \did
          \\
             D \ar@{=} [dr] & C' \ar@{=} [dl]
          \\
             C' \dcell{f'} & D \did
          \\
             D' & D 
         }
\xymatrix@R=10ex{ \\ \;\;\;=\;\;\; \\}
\xymatrix@C=0ex
         {
             C \ar@{=} [dr] & C' \ar@{=} [dl]
          \\
             C' \did & C \dcell{f}
          \\
             C' \dcell{f'} & D \did
          \\
             D' & D 
         }
\end{equation}
Cells going up or down pass through symmetries by changing the column.  

\vspace{1ex}

\emph{Combining the basic moves \eqref{ascensor} and \eqref{swap} we form configurations of cells that fit valid equations in order to prove new equations.}

\end{appendices}

\bibliographystyle{model1-num-names}
\bibliography{<your-bib-database>}

\begin{thebibliography}{99}

%
%
\bibitem{CW} Carboni A., Walters R.F.C., \textsl{Cartesian Bicategories I}, Journal of Pure and Applied Algebra 49 (1987), p. 11-32.
%
\bibitem{G2} Grothendieck A., \textsl{SGA1 (1960-61)}, Springer Lecture Notes in Mathematics 224 (1971).
%
\bibitem{Day} Day B., \textsl{Enriched Tannaka Reconstruction}, Journal of Pure and Applied Algebra 108 (1996), p.17-22.

\bibitem{De} Deligne P., \textsl{Cat\'egories Tannakiennes}, The Grothendieck Festschrift V.II, Modern Birkh\"auser Classics (1990), p. 111-195.

\bibitem{DM} Deligne P., Milne J.S., \textsl{Tannakian Categories}, Hodge Cocycles Motives and Shimura Varieties, Springer-Verlag (1982), p. 101-228

\bibitem{D1} Dubuc E. J., \textsl{Localic Galois Theory}, Advances in Mathematics 175 (2003), \mbox{p. 144-167.}

%
\bibitem{DSV} Dubuc E.J., Sanchez de la Vega C., \textsl{On the Galois Theory of Grothendieck}, Bol. Acad. Nac. Cienc. Cordoba (2000) p. 111-136.

\bibitem{DSz} Dubuc E.J., Szyld M., \textsl{A Tannakian Context for Galois Theory}, Advances in Mathematics 234 (2013), p. 528-549.

\bibitem{F} Freyd P.J., Scedrov A., \textsl{Categories, Allegories}, Mathematical Library Vol 39, North-Holland (1990).
%
%
%
\bibitem{JS} Joyal A., Street R., \textsl{An Introduction to Tannaka Duality and Quantum Groups}, Category Theory, Proceedings, Como 1990, Springer Lecture Notes in Mathematics 1488 (1991) p.413-492.
%
 \bibitem{JS2} Joyal A., Street R., \textsl{Braided Tensor Categories}, 
 Macquarie Mathematics Reports, Report No. 92-091 (1992).

\bibitem{JT} Joyal A., Tierney M., \textsl{An extension of the Galois Theory of Grothendieck}, Memoirs of the American Mathematical Society 151 (1984).

\bibitem{McC} McCrudden, P., \textsl{Tannaka duality for Maschkean categories}, Journal of Pure and Applied Algebra 168 (2002), p.265-307.
%
\bibitem{McLarty} McLarty C., \textsl{Elementary Categories, Elementary Toposes}, Oxford University Press, (1992).
%
%
\bibitem{Pitts} Pitts A., \textsl{Applications of Sup-Lattice Enriched Category Theory to Sheaf Theory}, Proc. London Math. Soc. (3) 37 (1988), p.433-480.
%
%
\bibitem{Sa} Saavedra Rivano, N., \textsl{Cat\'egories Tannakiennes}, Springer Lecture Notes in Mathematics 265 (1972).
%
\bibitem{SC} Schappi D. \textsl{The Formal Theory of Tannaka Duality}, Ast\'erisque 357 (2013).
%
%
\bibitem{SP} Schauenburg P., \textsl{Tannaka duality for arbitrary Hopf algebras}, Algebra Berichte [Algebra Reports], 66. Verlag Reinhard Fischer, Munich (1992).
%

\bibitem{tesis}  Szyld M., \textsl{Tannaka Theory over sup-lattices}, PhD. thesis, arXiv:1507.04772 (2015).

\bibitem{T}  Tannaka T., \textsl{\"Uber den Dualit\"atssatz der nichtkommutativen topologischen Gruppen}, T\^ohoku Math. J. 45, (1939), p. 1-12.

\bibitem{GW} Wraith G., \textsl{Localic Groups}, Cahiers de Top. et Geom. Diff. Vol XXII-1 (1981). 



\bibitem{JohnstoneFactorization} Johnstone, P. \textsl{Factorization theorems for geometric morphisms, I}, Cahiers de Topologie et G\'eom\'etrie Diff\'erentielle Cat\'egoriques 22.1 (1981), p 3-17.



\end{thebibliography}

\end{document}